\documentclass[11pt]{elsarticle}
\usepackage{amsmath,amssymb,amsthm,amsfonts,amstext,amsbsy,amscd}
\usepackage{graphics}
\usepackage{graphicx}
\usepackage{float}
\usepackage{multirow}
\usepackage{multicol}
\usepackage{latexsym}
\usepackage{verbatim}

\usepackage[utf8]{inputenc}
\usepackage{dsfont}
\usepackage{color}
\RequirePackage[colorlinks,citecolor=blue,urlcolor=blue]{hyperref}
\RequirePackage{hypernat}

\renewcommand{\phi}{\varphi}

\newtheorem{theorem}{Theorem}[section]
\newtheorem{lemma}[theorem]{Lemma}
\newtheorem{proposition}[theorem]{Proposition}

\newtheorem{remark}[theorem]{Remark}

\newtheorem{definition}[theorem]{Definition}

\newtheorem{assumption}[theorem]{Assumption}

\definecolor{forestgreen}{rgb}{0.13, 0.55, 0.13}
\definecolor{celestialblue}{rgb}{0.29, 0.59, 0.82}

\newcommand{\change}[1]{{#1}}

\title{Neural Field Equations and Hawkes processes: long-term stability of traveling wave profiles in the neutral case}

\author{Eric Lu\c{c}on\footnote{Universit\'e d’Orl\'eans, Universit\'e de Tours, CNRS, IDP, UMR 7013, Orl\'eans, France. E-mail: \href{mailto: eric.lucon@univ-orleans.fr}{eric.lucon@univ-orleans.fr}},\  \ Christophe Poquet\footnote{Univ Lyon, Universit\'e Claude Bernard Lyon 1, CNRS UMR 5208, Institut Camille Jordan, F-69622 Villeurbanne, France.  E-mail: \href{mailto: poquet@math.univ-lyon1.fr}{poquet@math.univ-lyon1.fr}}
}

\date{\today}

\begin{document}

\begin{abstract}
We consider the long-time behavior of a population of interacting Hawkes processes on the real line, with spatial extension. The large population behavior of the system is governed by the standard Voltage-based Neural Field Equation (NFE). We prove long-time stability of the system w.r.t. traveling wave solutions for the NFE on a diffusive time scale, in the neutral case, that is when the speed of the traveling wave solution is zero: the position of the traveling wave profile becomes essentially Brownian on a diffusive time scale. We borrow here from the seminal work of Chevallier, Duarte, L\"ocherbach and Ost concerning the approximation of NFE by Hawkes processes, the existence result of traveling waves of Ermentrout and McLeod and from the stability result of the traveling wave profile from Lang and Stannat. We provide also similar results concerning the activity based NFE.
\end{abstract}


\maketitle


\section{Introduction}

\subsection{Neural Field Equations}
\label{sec:NFE}
Neural Field Equations (NFE) have been first introduced by Wilson and Cowan \cite{Wilson:1972aa,Wilson1973} and Amari \cite{Amari:1977:DPF:2731211.2731248} to model mesoscopic dynamics of spatially-extended networks of interacting neurons. The literature generically distinguishes between the \emph{Voltage-based Neural Field Equation (VNFE)}
\begin{equation}
\label{eq:NFE_intro}
\partial_{ t} u_t(x) = - u_t(x) + \int_{ \mathcal{ D}} W(x,y) f(u_t(y)) {\rm d}y
\end{equation}
and the \emph{Rate-based (or Activity-based) Neural Field Equation (RNFE)}
\begin{equation}
\label{eq:rateNFE_intro}
\partial_{ t} v_{ t}(x) = -v_{ t}(x) + f \left(\int_{ \mathcal{ D}}W(x, y) v_{ t}(y) {\rm d}y\right).
\end{equation}
Here, $u_{ t}(x)$ stands for the voltage of a neuron and $v_{ t}(x)$ is the neuronal activity at a mesoscopic position $x\in \mathcal{ D}$ in a cortical region $ \mathcal{ D} \subset \mathbb{ R}^{ d}$, $d\geq1$. We refer to e.g. \cite{MR2871421} for biological heuristics for these equations. $W$ is a spatial kernel denoting inhomogeneous connectivity between neurons and $f$ is a sigmoid function modeling synaptic integration between neurons. NFE-type equations have been shown to exhibit various dynamical patterns, e.g. traveling waves, bump solutions, spiral waves, etc. We refer to e.g. \cite{MR2871421,Coombes2005,Ermentrout1998} for extensive reviews on the subject and to Section~\ref{sec:literature} below for further references.

\subsection{First assumptions and stability of traveling waves solutions for the NFE}
We are concerned in this paper with traveling waves solutions associated to \eqref{eq:NFE_intro} and \eqref{eq:rateNFE_intro}. A major part of the present paper is devoted to the analysis of the VNFE \eqref{eq:NFE_intro}, that we will denote as NFE for simplicity without further notice in the following. Corresponding results concerning the companion RNFE \eqref{eq:rateNFE_intro} will be addressed in Section~\ref{sec:RNFE}. The first rigorous proof for the existence of traveling wave solutions concerns the VNFE \eqref{eq:NFE_intro} and goes back to the paper of Ermentrout and McLeod \cite{ermentrout_mcleod_1993} in dimension $d=1$. We recall in this paragraph the setting of \cite{ermentrout_mcleod_1993} and place ourselves under the appropriate set of hypotheses ensuring the existence of traveling wave solutions for the NFE \eqref{eq:NFE_intro}. 

\subsubsection{First main hypotheses}
We assume in the rest of the paper the following
\begin{assumption}
\label{ass:Wf}
Suppose that $d=1$ and $ \mathcal{ D}= \mathbb{ R}$.
\begin{enumerate}
\item Assumption on the kernel $W$: suppose that $W(x,y)=W(x-y)$, $x,y\in \mathbb{ R}$ where the kernel $W$ is even, nonnegative and satisfies
\[
\int_{ \mathbb{ R}}W(x) {\rm d}x=1.
\]
\item Suppose that $f$ is a smooth sigmoid: $f: \mathbb{ R}\to [0,1]$ is three-times continuously differentiable on $ \mathbb{ R}$, with $f^{ \prime}>0$, \change{with $f^{ \prime \prime}$ and $f^{ (3)}$ bounded}. Suppose also that $x \mapsto f(x)-x$ is bistable, i.e. that $f(x)-x$ has exactly three zeros, $0\leq a_{ 1}<a <a_{ 2}\leq 1$ with $f^{ \prime}(a_{ i})<1$, $i=1,2$ and $f^{ \prime}(a)>1$. We suppose also that $f$ is convex-concave around $a$, that is $ f^{ \prime \prime}(x)\geq 0$ for $x\leq a$ and $ f^{ \prime \prime}(x)\leq 0$ for $x\geq a$.
\end{enumerate}
\end{assumption}
Under Assumption~\ref{ass:Wf}, the NFE \eqref{eq:NFE_intro} can be rewritten in a convolution form as
\begin{equation}
\label{eq:NFE conv}
\partial_{ t} u_t(x) = - u(x) + W*f(u_t)(x),\ x\in \mathbb{ R}.
\end{equation}
Here, constant functions $ u\equiv a_{ 1}, a_{ 2}, a$ are uniform solutions to the NFE \eqref{eq:NFE conv} (the two first being stable, the third one unstable). The results of \cite{ermentrout_mcleod_1993} concern the existence of traveling wave solutions to \eqref{eq:NFE conv}, that is solutions of the form $(x,t) \mapsto u_{ t}(x)= \hat{ u}(x-ct)$, where $ \hat{ u}$ is a traveling wave profile and  $c\in \mathbb{ R}$ is the speed of the traveling wave. Note that plugging the previous expression into \eqref{eq:NFE conv} give that $ (\hat{ u}, c)$ necessarily solves (by the change of variables $ \xi= x-ct$)
\begin{equation}
\label{eq:profile_hatu}
- c \partial_{ \xi} \hat{ u}(\xi)= - \hat{ u}(\xi) + W\ast f( \hat{ u})(\xi).
\end{equation}
The main result of \cite{ermentrout_mcleod_1993} is the following.

\begin{theorem}[Ermentrout and McLeod \cite{ermentrout_mcleod_1993}]
\label{th:EML}
Under Assumption~\ref{ass:Wf}, there exists a unique (modulo translations) couple $ \left( \hat{ u}, c\right)$, where $ \hat{ u}: \mathbb{ R}\to [0, 1]$ strictly increasing on $ \mathbb{ R}$ with $\lim_{ -\infty} \hat{ u}= a_{ 1}$ and $ \lim_{ +\infty} \hat{ u}= a_{ 2}$ and $c\in \mathbb{ R}$ so that $ u^{ TW}(x,t):= \hat{ u}(x-ct)$ solves \eqref{eq:NFE conv}. Moreover, the profile $ \hat{ u}$ is continuously differentiable and the speed $c$ satisfies
\begin{equation}
\label{eq:cVShatu}
c= \frac{ \int_{ a_{ 1}}^{a_{ 2}}(x-f(x)) {\rm d}x}{ \int_{ \mathbb{ R}} \left(\partial_{ x}\hat{ u}\right)^{ 2}(x) f^{ \prime}( \hat{ u}(x)) {\rm d}x}.
\end{equation}
\end{theorem}
For technical reasons, we specify the main results below to the case of exponential kernel~$W$:
\begin{assumption}
\label{ass:W_exp}
Suppose that for fixed $ \sigma>0$, 
\begin{equation}
\label{eq:Wexp}
W(x) = \frac{1}{2\sigma}e^{-\frac{|x|}{\sigma}},\ x\in \mathbb{ R}.
\end{equation}
\end{assumption}
\begin{remark}\label{rem:hyp W}
Theorem~\ref{th:EML} is valid for any generic choice of $W$ satisfying Assumption~\ref{ass:Wf} and not at all specific to exponential interaction kernel $W$ in \eqref{eq:Wexp}. The case \eqref{eq:Wexp} however plays a central role in the analysis of \eqref{eq:NFE conv}: the key argument of \cite{ermentrout_mcleod_1993} is precisely to show that in the case \eqref{eq:Wexp}, solving \eqref{eq:profile_hatu} with $c=0$ reduces to $ \hat{ u}^{ \prime \prime} + \frac{ 1}{ \sigma^{ 2}} \left( f(\hat{ u})- \hat{ u}\right)=0$ and the rest of the proof of \cite{ermentrout_mcleod_1993} is a clever continuation argument from this particular case. The reason why we restrict our analysis below to \eqref{eq:Wexp} is the possibility in this case to obtain explicit tail estimates on the traveling wave profile $ \hat{ u}$ (as noticed by e.g. Lang and Stannat, see \cite[Prop.~9]{2016JDE...261.4275L} and \cite[Section~4]{10113715M1033927}), which will be necessary in many points of our proof. Despite this technical aspect, we do not believe the precise form of $W$ to be really essential for the results of the paper.
\end{remark}

\subsubsection{The neutral case and the manifold of stationary solutions}
There is in general no explicit expression for $ \hat{ u}$ and the wave speed $c$: identity \eqref{eq:cVShatu} only defines $c$ in an implicit way, as $c$ depends itself on the unknown profile $ \hat{ u}$. However, \eqref{eq:cVShatu} gives that $c$ is of the same sign as the \change{numerator} $\int_{ a_{ 1}}^{a_{ 2}}(x-f(x)) {\rm d}x$. The second main assumption of the paper is that we restrict ourselves to the particular case where $c=0$:
\begin{assumption}[The neutral case]
\label{ass:neutral}
We make the following symmetry assumption for $f$ around the unstable fixed-point $a$:
\begin{equation}
\label{eq:f_symmetry_a}
f(a+x)-a=-(f(a-x)-a), \ x \in \mathbb{ R}.
\end{equation}
 In particular, one easily sees from \eqref{eq:f_symmetry_a} that $a=\frac{a_1+a_2}{2}$ and $ \int_{ a_{ 1}}^{a_{ 2}} \left(f(x)-x\right){\rm d}x=0$ so that using \eqref{eq:cVShatu}, we are in a situation where $c=0$.
\end{assumption}
The above conclusion that $c=0$ of Assumption~\ref{ass:neutral} seems quite restrictive. We discuss in Section~\ref{sec:extensions} below the practical challenges of the present case $c=0$ and the possibility to extend the present techniques to the situation where $c\neq 0$.

\medskip

Under Assumption~\ref{ass:neutral}, the existence of the traveling waves $(x,t) \mapsto \hat{ u}(x-ct)$ given by Theorem~\ref{th:main} translates (by translation invariance of \eqref{eq:NFE conv}) into the existence of a family of stationary solutions 
\begin{equation}
\hat u_\varphi(x):= \hat{ u}_{ 0}(x- \varphi),\ x\in \mathbb{ R}, \varphi\in \mathbb{ R}
\end{equation} 
for \eqref{eq:NFE conv} that are the translations of a fixed profile $\hat u_0$ (that we arbitrarily fix as the only translation of the original profile $ \hat{ u}$ such that $\hat{u}_{ 0}(0)=a$, so that $\hat u_\phi (\phi)=a$ for every $ \varphi\in \mathbb{ R}$). Moreover one can easily check that 
\begin{equation*}
\hat u_\phi(\phi+x)-a = -\left(\hat u_\phi(\phi-x)-a\right),\ x\in \mathbb{ R}, \varphi\in \mathbb{ R}.
\end{equation*}
We denote by $ \mathcal{ M}$ the manifold induced by these stationary solutions:
\begin{equation}
\label{eq:M}
\mathcal{ M} :=  \left\{\hat u_\varphi:\, \varphi\in \mathbb{R}\right\}.
\end{equation}
\subsubsection{Functional setting and well-posedness of the NFE}
In order to characterize the stability of $ \mathcal{ M}$ for \eqref{eq:NFE conv}, we follow the $L^{ 2}$-approach of Kr\"uger, Lang and Stannat \cite{2016JDE...261.4275L,https://doi.org/10.14279/depositonce-5019,10113713095094X,10113715M1033927}. Introduce $L^{ 2}:= (L^{ 2} \left(\mathbb{ R}\right), {\rm d}x)$ as the standard flat space of test functions $f$ on $ \mathbb{ R}$ such that $ \left\Vert f \right\Vert_{ L^{ 2}}^{ 2}:=\int_{ \mathbb{ R}} \left\vert f(x) \right\vert^{ 2} {\rm d}x<\infty$. Note that $L^{ 2}$ is \emph{not} a proper framework for the analysis of \eqref{eq:NFE conv} as for example $\hat u_\varphi$ does not belong to $L^{ 2}$. Observe here that finding a solution $ t \mapsto u_t$ to \eqref{eq:NFE conv} is equivalent to consider the difference between $u$ and the stationary profile $ \hat{ u}_{ 0}$
\begin{equation*}
v_t(x):= u_t(x) - \hat u_0(x), x\in \mathbb{ R},\ t\geq0
\end{equation*}
which solves
\begin{equation}
\label{eq:v}
\partial_{ t}v_t = -v_t + W \ast \left( f(\hat u_0+v_t) - f(\hat u_0)\right).
\end{equation}
Since the mapping $ v\in L^{ 2} \mapsto W \ast \left( f(\hat u_0+v) - f(\hat u_0)\right)\in L^2$ is Lipschitz, it is well known (see for example \cite{sell2013dynamics}) that for any initial condition $v_0\in L^2$ there exists a unique strong solution $v \in \mathcal{C}([0,\infty),L^2)$ to \eqref{eq:v}. This well-posedness result for \eqref{eq:v} translates in return into a well-posedness result for \eqref{eq:NFE conv} by setting $u_{ t}:= v_{ t}+ \hat{ u}_{ 0}$. Hence the proper functional setting for \eqref{eq:NFE conv} is to work in the \emph{affine} $L^{ 2}$-space
\begin{equation}
L^{ 2}_{ \hat{ u}_{ 0}}:= L^{ 2} + \hat{ u}_{ 0}.
\end{equation}
An important result concerning the local stability of the manifold of traveling wave solutions has been given by Stannat and Lang in \cite{2016JDE...261.4275L}: they prove in \cite[Th.~8]{2016JDE...261.4275L} the existence of a spectral gap for the linearized dynamics around the stationary profile $ \hat{ u}_{ 0}$, inducing some local stability property for the manifold $ \mathcal{ M}$. This result will be of crucial importance to us as it naturally induces the existence of a neighborhood of $\mathcal{M}$ such that for any initial condition $u_0$ in this neighborhood, the solution $u_t$ to \eqref{eq:NFE conv} converges exponentially fast to some $\hat u_{\Theta(u_0)}$ for some phase $\Theta(u_0)\in \mathbb{ R}$. The mapping $\Theta$ is called the isochron map, and is constant along the trajectories defined by \eqref{eq:NFE conv}. More details about this stability result and the regularity of the isochron map will be given is Section~\ref{sec: stab, isochron}.

\section{The model and main results}
\subsection{The microscopic model}
\label{sec:micro_model}
So far, we have only discussed about the deterministic, macroscopic model given by Neural Field Equations \eqref{eq:NFE conv} and \eqref{eq:rateNFE_intro}. The aim of the present paper is to introduce a probabilistic microscopic model of interacting point processes that approximates correctly in large population the traveling wave solutions to both \eqref{eq:NFE conv} and \eqref{eq:rateNFE_intro}. The main result of the paper will be to look at the diffusive behavior of this particle system w.r.t. the structure of traveling wave profiles on a long time scale.

\subsubsection{Motivation}
The question of the derivation of the macroscopic NFE (as well as stochastic perturbations) by adequate microscopic systems has been a longstanding issue in the mathematical neuroscience literature, mostly from a non-rigorous perspective. We refer to Section~\ref{sec:literature} below for a detailed exposition of the literature on the subject. The present model originates from a previous work by Chevallier, Duarte, L\"ocherbach and Ost \cite{CHEVALLIER20191}, which has been the first to derive rigorously NFE-type dynamics from the large population limit of point processes of Hawkes type. Hawkes point processes have been introduced by Hawkes and Oakes in \cite{MR0378093} (originally motivated by the modeling of earthquakes activity) and have met recently a growing interest in mathematical neuroscience, since this framework allows to capture time-dependence in the distribution of spikes of interacting neurons. In order to motivate our model, let us mention briefly the result of \cite{CHEVALLIER20191}: the authors introduce a family of $N$ mean-field interacting Hawkes processes, defined through their counting processes $(Z_{ 1}, \ldots, Z_{ N})$, where $Z_{ i,t}$ accounts for the number of spikes of neuron $i$ in the time interval $[0,t]$. Each neuron $i$ is associated to a position $x_{ i}\in \mathbb{ R}^{ d}$ and the dynamics of each $Z_{ i}$ is defined through its intensity process defined as $\lambda_{ i,t}= f \left(U_{ i, t-}\right)$ where the potential $U_{ i}$ satisfies (with our notations)
\begin{equation}
\label{eq:lambdai_CDLO}
U_{ i,t}= e^{ - t} u_{ 0}(x_{ i}) + \frac{ 1}{ N} \sum_{ j=1}^{ N} W(x_{ i}- x_{ j}) \int_{ 0}^{ t} e^{ -(t-s)}{\rm d}Z_{ j, s}.
\end{equation}
Hence, \eqref{eq:lambdai_CDLO} defines a mean-field family of interacting Hawkes processes with spatial extension. An important hypothesis in \cite{CHEVALLIER20191} is the convergence of the empirical measure of the positions $ \frac{ 1}{ N}\sum_{ j=1}^{ N} \delta_{ x_{ j}}$ towards some probability profile $ \rho({\rm d}x)$ on $ \mathbb{ R}^{ d}$. Under some further mild assumptions on the coefficients, the main result of \cite{CHEVALLIER20191} concerns the convergence of the system as $N\to\infty$ to independent inhomogeneous copies of a nonlinear point process $ \left(\bar Z_{ x,t}\right)_{ x\in \mathbb{ R}}$ with intensity $ \lambda_{ t}(x)=f \left(u_{ t}(x)\right)$ solving the NFE-type equation
\begin{equation}
\label{eq:NFE_CDLO}
\partial_{ t} u_t(x) = - u_t(x) + \int_{ \mathbb{ R}^{ d}} W(x-y) f(u_t(y)) \rho({\rm d}y).
\end{equation}
The only but crucial difference with the original NFE \eqref{eq:NFE conv} lies in the integration w.r.t. the spatial variable: $ \rho$ in \eqref{eq:NFE_CDLO} is a probability measure whereas one integrates w.r.t. the Lebesgue measure in \eqref{eq:NFE conv}. From the perspective of traveling wave solutions for these equations, this changes dramatically the nature of the dynamics: Theorem~\ref{th:EML} uses in an essential way the translation invariance of the convolution in \eqref{eq:NFE conv}, which is broken by the presence of $ \rho$ in \eqref{eq:NFE_CDLO}. In words, there is in general no such thing as traveling waves for \eqref{eq:NFE_CDLO}.

\subsubsection{Our model}
The previous remark is the starting point of the present work. The aim of the paper is to address the following questions:
\begin{enumerate}
\item Can we construct a Hawkes dynamics that approximates correctly the NFE \eqref{eq:NFE conv} and reproduces faithfully its traveling wave structure?
\item Can we take advantage of the structural stability of these traveling waves to infer the influence of the noise within the microscopic system on the macroscopic properties of these traveling waves? The question of the influence of noise on structured dynamics such as traveling waves is a longstanding issue in the literature, which has been mostly addressed so far from a macroscopic perspective (that is from a SPDE point of view, i.e. adding directly a macroscopic noise to the NFE \eqref{eq:NFE conv}). We refer in details to the SPDE approach in Section~\ref{sec:literature} below.
\end{enumerate}

We are now in position to define our model. We follow here the classical construction of Hawkes processes via a thinning procedure w.r.t. independent Poisson random measures \cite{Ogata1981,MR546120}. For the rest of the paper, $ \varepsilon>0$ is a small parameter and $f$ and $W$ satisfy Assumptions~\ref{ass:Wf},~\ref{ass:W_exp} and ~\ref{ass:neutral}.
\begin{definition}
\label{def:model}
Consider an infinite family of neurons, indexed by $ \mathbb{ Z}$, where the neuron $i\in \mathbb{ Z}$ is attached to position $x_{ i}= i \varepsilon\in \varepsilon\mathbb{ Z}$. Denote by $ t \mapsto Z_{ i,t}^{ (\varepsilon)}$ the counting process associated to neuron $i\in \mathbb{ Z}$: for $t>0$, $Z_{ i,t}^{ (\varepsilon)}$ equals the number of spikes of neuron $i$ on $[0,t]$. The family $ \left(Z_{ i,t}^{ (\varepsilon)}\right)_{ i\in \mathbb{ Z}, t\geq0}$ is defined as follows: let $ \pi_{ i} \left({\rm d}s, {\rm d}z\right), i\in \mathbb{ Z}$ be a collection of i.i.d. Poisson random measures on $[0, +\infty) \times [0, +\infty)$ with intensity $ {\rm d}z {\rm d}t$. Then define, for $i\in \mathbb{ Z}$,
\begin{equation}
\label{eq:Zi_def}
Z_{ i,t}^{ (\varepsilon)}= \int_{ 0}^{t} \int_{ 0}^{+\infty} \mathbf{ 1}_{ z\leq \lambda_{ i,s}^{ ( \varepsilon)}} \pi_{ i}({\rm d}s, {\rm d}z),\ t\geq0,
\end{equation}
where the conditional intensity $\lambda_{ i,t}^{ ( \varepsilon)}$ of $Z_{i}^{ (\varepsilon)}$ is given by
\begin{equation}
\label{eq:lambda_i}
\lambda_{ i,t}^{ ( \varepsilon)}= f \left(U_{ i,t^-}^{ (\varepsilon)}\right),\ t \geq0
\end{equation}
with potentials $U_{i, t}^{ (\varepsilon)}$ defined as
\begin{align}
\label{eq:Ui_mod}
U_{i,t}^{ (\varepsilon)}&= e^{ -t} u_{ 0}^{ (\varepsilon)}(x_{ i}) + \varepsilon\sum_{ j\in \mathbb{Z}} W(x_{ i}- x_{ j}) \int_{ 0}^{t} e^{ -(t-s)} {\rm d}Z_{ j,s}^{ (\varepsilon)}, \quad i=- N_{ \varepsilon}, \ldots, N_{ \varepsilon},\ t\geq0
\end{align}
and
\begin{align}
\label{eq:Ui_outside}
U_{i,t}^{ (\varepsilon)}=\begin{cases}
a_{ 1} & \text{ for } i < -N_{ \varepsilon},\\
a_{ 2}& \text{ for } i > N_{ \varepsilon}
\end{cases}
\quad ,\ t\geq0
\end{align}
for 
\begin{equation}
\label{eq:N_ell_eps}
N_\varepsilon=\lfloor \ell_\varepsilon /\varepsilon\rfloor\ \text{ and }\ \ell_{ \varepsilon}:= \frac{ 1}{ \varepsilon^{ \beta}},
\end{equation}
where for the rest of the paper, the parameter $ \beta>0$ is a small parameter, chosen to verify
\begin{equation}
\label{eq:beta_small}
\beta\in \left(0, \change{\frac{ 1}{ 24}}\right).
\end{equation}
\end{definition}
In words, we are considering an infinite family of interacting neurons, whose dynamics depends on their positions: the neurons inside the bulk (i.e. of index $ i\in \left\lbrace -N_{ \varepsilon}, \ldots, N_{ \varepsilon}\right\rbrace$) are of Hawkes type, interacting with the whole population through \eqref{eq:Ui_mod}. They feel in particular the influence of boundary condition in $\pm \infty$, in the presence of independent Poisson processes outside the bulk, with intensities that match the boundary conditions $a_{ 1}$ and $ a_{ 2}$ for the traveling wave profile $ \hat{ u}$.

\medskip

In addition to Assumptions~\ref{ass:Wf},~\ref{ass:W_exp}, and~\ref{ass:neutral}, we suppose also
\begin{assumption}
\label{ass:u0}
The initial condition $u_{ 0}^{ (\varepsilon)}$ is in the affine space $L^{ 2}_{ \hat{ u}_{ 0}}= L^{ 2}+ \hat{ u}_{ 0}$, uniformly in $ \varepsilon>0$: $\sup_{ \varepsilon>0} \left\Vert u_{ 0}^{ (\varepsilon)}- \hat{ u}_{ 0} \right\Vert_{ L^{ 2}}< \infty$. Moreover, the family $ \left(u_{ 0}^{ (\varepsilon)}\right)_{ \varepsilon>0}$ is uniformly Lipschitz continuous on $ \mathbb{ R}$: there exists some $C>0$ such that for all $ \varepsilon>0$
\begin{equation*}
\left\vert u_{ 0}^{ (\varepsilon)}(x)  - u_{ 0}^{ (\varepsilon)}(y)\right\vert \leq C \left\vert x-y \right\vert,\ x,y\in \mathbb{ R}.
\end{equation*}
\end{assumption}
The first result concerns the well-posedness of the system given by Definition~\ref{def:model}:
\begin{proposition}
\label{prop:model}
Under Assumptions~\ref{ass:Wf} and \ref{ass:u0}, there exists a pathwise unique system $ \left(Z_{ i}^{ (\varepsilon)}\right)_{ i\in \mathbb{ Z}}$ satisfying Definition~\ref{def:model} such that $ t \mapsto \sup_{ i\in \mathbb{ Z}} \mathbb{ E} \left(Z_{ i, t}^{ (\varepsilon)}\right)$ is locally bounded.
\end{proposition}
Proof of Proposition~\ref{prop:model} relies on the following remark: note that by assumption, the processes $ Z_{ i}^{ (\varepsilon)}$ for $ \left\vert i \right\vert> N_{ \varepsilon}$ are nothing else than independent homogeneous Poisson processes (with constant rates $a_{ 1}$ or $a_{ 2}$). Hence the only difficulty of Proposition~\ref{prop:model} is to give a construction of the Hawkes processes within $i=-N_{ \varepsilon},\ldots, N_{ \varepsilon}$ (and there is a finite number of them). The proof relies then on a classical Picard iteration whose proof can be found in \cite[Th.~6]{MR3449317}.

\subsection{The voltage profile and main result}
\subsubsection{The voltage profile}
The main object of interest of the paper is the voltage profile associated to the neurons given by Definition~\ref{def:model}, that is defined in the following way: consider the intervals 
\begin{equation}
\label{eq:def_I_i}
I_{i}^{ (\varepsilon)} := \left[x_i- \frac{ \varepsilon}{ 2},x_{i}+ \frac{ \varepsilon}{ 2}\right), \ i \in \mathbb{ Z},
\end{equation} so that $ \mathbb{ R}$ is partitioned into $ \mathbb{ R}= \mathcal{ D}_{ -}^{ (\varepsilon)} \cup \mathcal{ D}_{ 0}^{ (\varepsilon)} \cup \mathcal{ D}_{+}^{ (\varepsilon)}$ where
\begin{equation}
\label{eq:partition_R}
\begin{cases}
\mathcal{ D}_{ -}^{ (\varepsilon)}&:= \left( -\infty, -\varepsilon N_\varepsilon- \frac{ \varepsilon}{ 2}\right),\\
\mathcal{ D}_{ 0}^{ (\varepsilon)}&:= \bigcup_{ i=-N_{ \varepsilon}}^{N_{ \varepsilon}} I_{ i}^{ (\varepsilon)}= \left[  -\varepsilon N_\varepsilon- \frac{ \varepsilon}{ 2},  \varepsilon N_\varepsilon+ \frac{ \varepsilon}{ 2} \right).\\
\mathcal{ D}_{ +}^{ (\varepsilon)}&:= \left[ \varepsilon N_\varepsilon+ \frac{ \varepsilon}{ 2}, +\infty\right).
\end{cases}
\end{equation} 
Define the microscopic voltage profile $\left(U_t^{ (\varepsilon)} \right)_{ t\geq0}$ associated to the particle system $ \left(Z_{ i}^{ (\varepsilon)}\right)_{ i\in \mathbb{ Z}}$ as
\begin{equation}
\label{eq:profileU}
U_t^{ (\varepsilon)}(x):= U_{ t}(x) =a_1{\bf 1}_{ \mathcal{ D}_-^{ (\varepsilon)}}(x) + \sum_{ i=-N_\varepsilon }^{N_\varepsilon} U_{i,t}^{ (\varepsilon)} {\bf 1}_{I_{i}^{ (\varepsilon)}}(x)+a_2{\bf 1}_{ \mathcal{ D}_+^{ (\varepsilon)}}(x),\ x\in \mathbb{ R},
\end{equation}
where $ U^{ (\varepsilon)}_{ i, t}$ is given by \eqref{eq:Ui_mod}.
\subsubsection{Main result}

The main result of the paper is then
\begin{theorem}\label{th:main}
Grant Assumptions~\ref{ass:Wf},~\ref{ass:W_exp},~\ref{ass:neutral}, and~\ref{ass:u0}. Let $ \beta$ verify \eqref{eq:beta_small}. Let $ U^{ (\varepsilon)}$ be the spatial profile process given by \eqref{eq:profileU}, associated to the family of interacting neurons $ \left(Z_{ i}^{ (\varepsilon)}\right)_{ i\in \mathbb{ Z}}$ given by Definition~\ref{def:model} and let $ \hat{ u}$ be the traveling wave profile given by Theorem~\ref{th:EML}.  Then, there exists some $ \beta_{ 0}>0$ such that, if as $ \varepsilon\to0$
\begin{equation}
\label{eq:u0_close_to_M}
\left\Vert u_{ 0}^{ (\varepsilon)} - \hat{ u}_{ 0}\right\Vert_{ L^{ 2}} \leq C \varepsilon^{ \beta_{ 0}}
\end{equation} then the following holds: 
\begin{enumerate}
\item Convergence on a bounded time interval: for any fixed $T_{ 0}>0$, all $ \eta>0$
\begin{equation}
\label{eq:conv_main_T0}
\mathbb{P}\left(\sup_{0\leq t\leq T_{ 0} } \left\Vert U_{t}^{ (\varepsilon)}-\hat u\right\Vert_{L^2} \geq \eta\right) \underset{\varepsilon\rightarrow 0}{\longrightarrow} 0,
\end{equation}
\item Brownian wandering along $ \mathcal{ M}$ on a diffusive time scale: for any $t_f>0$, there exists some càdlàg process $ \psi^{ (\varepsilon)}:= \left( \psi^{ (\varepsilon)}_{ u}\right)_{ u\in [0, t_{ f}]}$ that converges in distribution as $ \varepsilon\to 0$ to a standard Brownian motion $ B$ on $[0, t_{ f}]$ \change{as well as some $\bar v\in \mathbb{ R}$} such that for all $\eta>0$,
\begin{equation}
\label{eq:conv_main}
\mathbb{P}\left(\sup_{0\leq u\leq t_f } \left\Vert U_{u\varepsilon^{-1}}^{ (\varepsilon)}-\hat u_{  \change{u \bar v} + \varrho\psi_{u}^{ (\varepsilon)}}\right\Vert_{L^2} \geq \eta\right) \underset{\varepsilon\rightarrow 0}{\longrightarrow} 0,
\end{equation}
where the constant $ \varrho>0$ is given by
\begin{equation}
\label{eq:sigma}
\varrho^{ 2}:= \frac{\int_{ \mathbb{ R}}\left( \int_{ \mathbb{ R}} W(x-y)\hat u'_{ 0}(y) f^{ \prime}( \hat{ u}_{ 0}(y)) {\rm d}y\right)^{ 2}f \left( \hat{ u}_{ 0}\right)(x) {\rm d}x}{ \change{ \left(\int_{ \mathbb{ R}} \left( \hat{ u}_{ 0}^{ \prime}\right)^{ 2}(x) f^{ \prime} \left( \hat{ u}_{ 0}(x)\right) {\rm d}x\right)^{ 2}}}.
\end{equation}
\end{enumerate}
\end{theorem}
The informal conclusion of Theorem~\ref{th:main} is the following: on a diffusive time scale of order $ \varepsilon^{ -1}$, the process $U^{ (\varepsilon)}$ wanders about along the manifold $ \mathcal{ M}$ of traveling wave profiles and the phase of the process along $ \mathcal{ M}$ is essentially Brownian, see Figure~\ref{fig:TW}.
\change{\begin{remark}
\label{rem:v0}
Note that the induced drift $ \bar v$ in the statement of Theorem~\ref{th:main} may be written as
\begin{equation}
\bar v:= \frac{ 1}{ 2} \int_{ \mathbb{ R}} {\rm D}^{ 2} \Theta[W_{ y}, W_{ y}] f \left( \hat{ u}_{ 0}(y)\right) {\rm d}y
\end{equation}
where $ \Theta$ is the isochron related to the deterministic dynamics of \eqref{eq:NFE_intro} in a neighborhood of $ \mathcal{ M}$ (see Proposition~\ref{prop:regularity isochron} for a precise definition) and $W_{ y}= W \left(\cdot-y\right)$.
\end{remark}}
\begin{figure}[h]
\centering
\includegraphics[width=\textwidth]{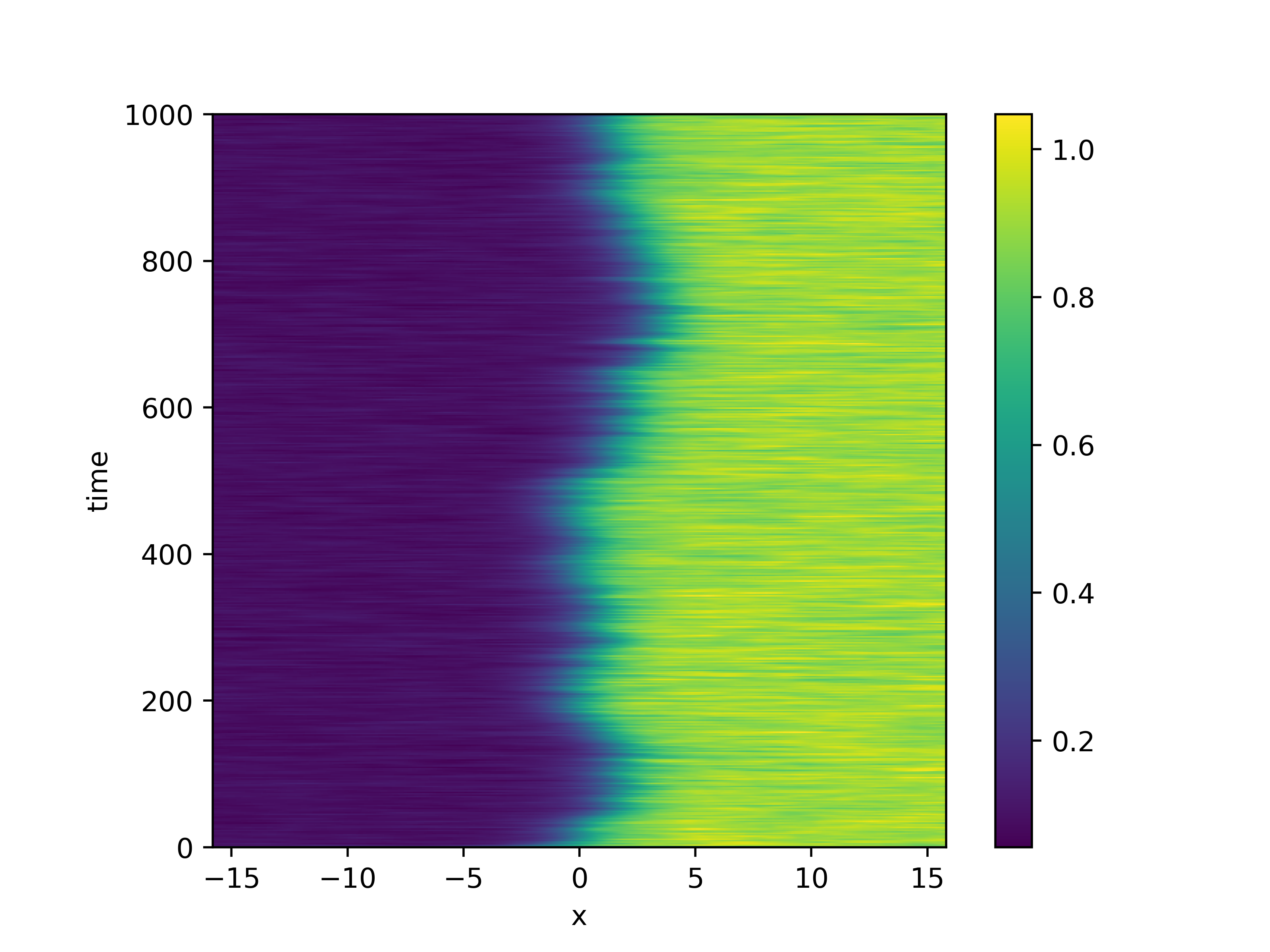}
\caption{Simulation of the voltage profile $U^{ (\varepsilon)}$ for $ \varepsilon=0.01$, $ \beta= \frac{ 3}{ 4}$, $t_{ f}=10$, $f(x)= \frac{ 1}{ 2} + \frac{ 1}{ \pi}\arctan \left(8(x- \frac{ 1}{ 2})\right)$, $W$ given by \eqref{eq:Wexp} for $ \sigma=1$ and initial condition $ u_{ 0}^{ (\varepsilon)}(x)= a_{ 1} \mathbf{ 1}_{ x< - \ell_{ \varepsilon}} + \left(a_{ 1}+ \frac{ a_{ 2}-a_{ 1}}{ 2\ell_{ \varepsilon}}(x+\ell_{ \varepsilon}) \right)\mathbf{ 1}_{ - \ell_{ \varepsilon} \leq x \leq \ell_{ \varepsilon}} + a_{ 2} \mathbf{ 1}_{ x> \ell_{ \varepsilon}}$.}
\label{fig:TW}
\end{figure}
\subsubsection{About the initial condition}
We have assumed in \eqref{eq:u0_close_to_M} that the initial condition $ u_{ 0}^{ (\varepsilon)}$ gets eventually close, as $ \varepsilon\to0$ to the invariant manifold $ \mathcal{ M}$. This choice has been made mostly for simplicity of exposition and can be easily generalized as follows: one can replace condition \eqref{eq:u0_close_to_M} by the more general assumption of convergence of $u_{ 0}^{ (\varepsilon)}$ towards some element $ \tilde{ u}_{ 0}$ in a macroscopic neighborhood of $ \mathcal{ M}$ (but not necessarily on $ \mathcal{ M}$). Indeed, supposing now that 
\begin{equation}
\label{eq:u0_not_tooclose_to_M}
\left\Vert u_{ 0}^{ (\varepsilon)} - \tilde{ u}_{ 0}\right\Vert_{ L^{ 2}} \leq C \varepsilon^{ \beta_{ 0}}
\end{equation} 
where $ \tilde{ u}_{ 0} \in L^{ 2}_{ \hat{ u}_{ 0}}$ verifies $ {\rm dist}_{ L^{ 2}}( \tilde{ u}_{ 0}, \mathcal{ M})\leq \eta_{ 0}$ for some sufficiently small $ \eta_{ 0}>0$, then the result of Theorem~\ref{th:main} remains true, up to the following slight modifications: the convergence \eqref{eq:conv_main} now becomes
\begin{equation}
\label{eq:conv_main_gen}
\mathbb{P}\left(\sup_{ c \varepsilon \log \left\vert \varepsilon \right\vert\leq u\leq t_f } \left\Vert U_{u\varepsilon^{-1}}^{ (\varepsilon)}-\hat u_{ \theta_{ 0}+  \change{u \bar v} + \varrho\psi_{u}^{ (\varepsilon)}}\right\Vert_{L^2} \geq \eta\right) \underset{\varepsilon\rightarrow 0}{\longrightarrow} 0.
\end{equation}
This logarithmic time in $c \log \left\vert \varepsilon \right\vert$ represents the initial time that is necessary for the profile $ U^{ (\varepsilon)}$, starting from $u_{ 0}^{ (\varepsilon)}\approx \tilde{ u}_{ 0}$ to reach an $ \varepsilon$-neighborhood of $ \mathcal{ M}$ (and once $U^{ (\varepsilon)}$ has reached such a microscopic neighborhood, we are back in the hypotheses of Theorem~\ref{th:main} and the rest of the proof follows). Starting from $ \tilde{ u}_{ 0}$ to an $ \varepsilon$-neighborhood of $ \mathcal{ M}$ is indeed logarithmic in $ \varepsilon$, as the system mostly follows the deterministic flow given by the NFE \eqref{eq:NFE conv} that is locally exponentially attractive (see Theorem~\ref{th:stability} below). Doing so, the deterministic flow induces a slight dephasing expressed in \eqref{eq:conv_main_gen} by the presence of the initial $ \theta_{ 0}$. The treatment of this initial approach to the manifold $ \mathcal{ M}$ follows a discretization procedure that is in itself very similar to the one that we describe in the following for the proof of Theorem~\ref{th:main}. As this initial procedure has been already written in details in previous papers (in different context of phase oscillators, see e.g. \cite{Bertini2014,Giacomin:2015ab,2021arXiv210702473L}), we have chosen to keep to hypothesis \eqref{eq:u0_close_to_M} for simplicity of exposition and refer the reader to the aforementioned references for further details.

\subsection{About the rate-based Neural Field Equation}
\label{sec:RNFE}

\subsubsection{Definitions and traveling wave solutions}
The above results concern the Voltage-based NFE \eqref{eq:NFE conv}. We are interested in this paragraph in the companion Rate-based NFE \eqref{eq:rateNFE_intro}. Surprisingly, very few rigorous results exist about the latter (we are aware at least of \cite{Riedler2013} where well-posedness and  approximation results of \eqref{eq:rateNFE_intro} in terms of Markov processes on a bounded domain are provided). Again, write \eqref{eq:rateNFE_intro} in a convolution form, for $d=1$:
\begin{equation}
\label{eq:rateNFE}
\partial_{ t} v_{ t}(x) = -v_{ t}(x) + f \left(\int_{ \mathbb{ R}}W(x-y) v_{ t}(y) {\rm d}y\right), \ v_{ t}(x)_{ \vert_{ t=0}}=v_{ 0}(x).
\end{equation}
It appears that a very simple correspondence can be drawn between the VNFE \eqref{eq:NFE_intro} and the RNFE \eqref{eq:rateNFE_intro}, which allows in particular to prove the existence of traveling wave solutions to \eqref{eq:rateNFE}. Surprisingly, despite the simplicity of Proposition~\ref{prop:correspondanceNFE}, we have have not been able to find a similar result in the existing literature. In this sense, to the best of our knowledge, Proposition~\ref{prop:correspondanceNFE} appears to be new and may have an interest of its own. Proof of Proposition~\ref{prop:correspondanceNFE} is given in Appendix~\ref{sec:app rate based NFE}.
\begin{proposition}
\label{prop:correspondanceNFE}
Suppose Assumption~\ref{ass:Wf}. Then, the following holds:
\begin{enumerate}
\item \label{it:v} Well-posedness for \eqref{eq:rateNFE}: for any initial condition $v_{ 0}$ that is nonnegative, continuous and bounded on $ \mathbb{ R}$, there is a unique solution $ t \mapsto v_{ t}$ to \eqref{eq:rateNFE} that is continuous on $ \mathbb{ R}$. Moreover, such a solution remains nonnegative and bounded by $\max \left( \left\Vert v_{ 0} \right\Vert_{ \infty}, \left\Vert f \right\Vert_{ \infty}\right)$.
\item \label{it:utov} Correspondence between solutions to the VNFE \eqref{eq:NFE conv} and the RNFE \eqref{eq:rateNFE}:
\begin{enumerate}
\item  \label{subit:vtfromut}Let $u_{ 0}$ be continuous and nonnegative on $ \mathbb{ R}$. Let $(u_{ t})_{ t\geq0}$ be the solution to  the VNFE \eqref{eq:NFE conv} with initial condition $u_{ 0}$. Let $v_{ 0}$ such that
\begin{equation}
\label{eq:u0Wv0}
u_{ 0}= W\ast v_{ 0}.
\end{equation}
Then 
\begin{equation}
\label{eq:vtfromut}
v_{ t}:= \int_{ 0}^{t} e^{ -(t-s)} f(u_{ s}) {\rm d}s+ v_{ 0} e^{ -t}
\end{equation}
solves the RNFE \eqref{eq:rateNFE} with initial condition $v_{ 0}$ (we refer to Remark~\ref{rem:u0Wv0} for more comments on the identity \eqref{eq:u0Wv0}).
\item \label{subit:utfromvt}Conversely, let $v_{ 0}$ nonnegative, continuous and bounded and let $(v_{ t})_{ t\geq0}$ solution to the RNFE \eqref{eq:rateNFE} with initial condition $v_{ 0}$. Then 
\begin{equation}
\label{eq:utWvt}
t \mapsto u_{ t}:= W\ast v_{ t}
\end{equation}
is well-defined, bounded and continuous on $ \mathbb{ R}$ and solves the VNFE \eqref{eq:NFE conv} with initial condition $u_{ 0}=W\ast v_{ 0}$.
\end{enumerate} 
\item \label{it:TW} Traveling waves solution to  the RNFE \eqref{eq:rateNFE}: under Assumption~\ref{ass:Wf}, there exists one and (modulo translation) only one monotonic traveling wave solution $ (x,t) \mapsto \hat{ v}(x-c^{ \prime}t)$ to the RNFE \eqref{eq:rateNFE} such that $\lim_{-\infty} \hat{ v}= a_{ 1}$ and $ \lim_{ +\infty} \hat{ v}=a_{ 2}$. Moreover, the speed $ c^{ \prime}$ is equal to the speed $c$ of the traveling wave profile $ \hat{ u}(x-ct)$ associated to the VNFE \eqref{eq:NFE conv} and one has the following correspondence between the two profiles: 
\begin{enumerate}
\item Let $ \hat{ u}$ be a traveling wave profile to the VNFE \eqref{eq:NFE conv}. Then, any traveling profile $ \hat{ v}$ to the RNFE \eqref{eq:rateNFE} is a translation of
\begin{equation}
\label{eq:hatv}
\hat{ v}(x):= \begin{cases}
f \left(\hat{ u}(x)\right)& \text{ if } c=0,\\
\change{\frac{1}{ c} \int_{ 0}^{+\infty} e^{ - \frac{ z}{ \left\vert c \right\vert}} f \left( \hat{ u}(x+ \frac{ c}{ \left\vert c \right\vert}z)\right) {\rm d}z}, & \text{ if }c\neq 0.
\end{cases}
\end{equation}
\item Let $ \hat{ v}$ a traveling wave profile to the RNFE \eqref{eq:rateNFE}. Then, any traveling profile $ \hat{ u}$ to the VNFE \eqref{eq:NFE conv} is necessarily a translation of 
\begin{equation}
\label{eq:hatuVSv}
\hat{ u} = W \ast \hat{ v}.
\end{equation}
\end{enumerate}
We define as
\begin{equation}
\label{eq:Mrate}
\mathcal{ M}_{ rate}:= \left\lbrace \hat{ v}(\cdot- \psi),\ \psi\in \mathbb{ R}\right\rbrace
\end{equation}
as the manifold of traveling wave profiles for \eqref{eq:rateNFE}.
\end{enumerate}
\end{proposition}

\begin{remark}
\begin{enumerate}
\label{rem:u0Wv0}
\item Although written in convolution form for $d=1$ for simplicity, Items~\ref{it:v} and~\ref{it:utov} of Proposition~\ref{prop:correspondanceNFE} are actually valid at the level of generality of \eqref{eq:NFE_intro} and \eqref{eq:rateNFE_intro} (in this case, \eqref{eq:u0Wv0}, resp. \eqref{eq:utWvt}, must be written as $u_{ 0} = \int_{ \mathbb{ R}^{ d}} W(\cdot, y) v_{ 0}(y) {\rm d}y$, resp. $u_{ t} = \int_{ \mathbb{ R}^{ d}} W(\cdot, y) v_{ t}(y) {\rm d}y$).
\item The passage from a solution $u_{ t}$ to \eqref{eq:NFE conv} to a solution $v_{ t}$ to \eqref{eq:rateNFE} requires to solve the compatibility condition for the initial conditions \eqref{eq:u0Wv0}. Solving \eqref{eq:u0Wv0} in $v_{ 0}$ for a given $u_{ 0}$ corresponds to a generic deconvolution problem, but in a rather unconventional setting: we are generically looking at traveling wave solutions for $u_{ 0}$ and $v_{ 0}$, so that a direct application of Fourier techniques in e.g. some $L^{ 2} \left(\mathbb{ R}\right)$-framework is not directly applicable. Note however that one is able to solve this problem in the particular case of  \eqref{eq:Wexp}: provided $u_{ 0}\in \mathcal{ C}^{ 2}(\mathbb{ R})$, there is indeed a unique $v_{ 0}$ satisfying \eqref{eq:u0Wv0} given by
\begin{equation}
\label{eq:u0Wv0exp}
v_{ 0}= u_{ 0}- \sigma^{ 2} u_{ 0}^{ \prime \prime}.
\end{equation}
\item A last remark is that the compatibility condition \eqref{eq:u0Wv0} concerning the traveling wave profiles $ \hat{ u}= W\ast \hat{ v}$ is true by construction, without any further assumption on $W$ and not at all specific to the case \eqref{eq:Wexp}. Hence, the correspondence between the traveling solutions between \eqref{eq:NFE conv} and \eqref{eq:rateNFE} in Item~\eqref{it:TW} of Proposition~\ref{prop:correspondanceNFE} is true for general $W$ verifying Assumption~\ref{ass:Wf} only.
\end{enumerate}
\end{remark}

\subsubsection{Microscopic interpretation of the rate-based NFE and stability of traveling wave solutions}
The point of the present section is to transpose the macroscopic correspondence of Proposition~\ref{prop:correspondanceNFE} between the VNFE \eqref{eq:NFE conv} and the RNFE \eqref{eq:rateNFE} at a microscopic level, i.e. at the level of the Hawkes particle system that has been defined in Definition~\ref{def:model}. We  take advantage of the construction of Proposition~\ref{prop:correspondanceNFE} and the stability result of Theorem~\ref{th:main} concerning the VNFE \eqref{eq:NFE conv} to derive a similar stability result for the traveling wave profile of the RNFE \eqref{eq:rateNFE}. Note that we give in this way a novel and rigorous microscopic interpretation of the RNFE \eqref{eq:rateNFE} as the macroscopic limit of the rates of interacting Hawkes processes.

More precisely, recall the definition of the intensity $ \lambda_{ i}^{ (\varepsilon)}$ of the Hawkes process $Z_{ i}^{ (\varepsilon)}$ in \eqref{eq:lambda_i} and the voltage profile $U^{ (\varepsilon)}$ in \eqref{eq:profileU}. A natural transposition of the macroscopic correspondence formula \eqref{eq:vtfromut} provides a first simple microscopic candidate for the rate-based NFE \eqref{eq:rateNFE} in terms of the intensity of the Hawkes particle system:
\begin{definition}[Convoluted rate profile]
\label{def:conv_rate_profile}
For any $i\in \mathbb{ Z}$, $ \varepsilon>0$, introduce the process
\begin{equation}
\label{eq:conv_rate_i}
\change{V_{ i,t}}^{ (\varepsilon)}:= \int_{ 0}^{t} e^{ - (t-s)} \lambda_{ i,s}^{ (\varepsilon)} {\rm d}s,\ t\geq0.
\end{equation}
In words, $ \change{V_{ i,t}}^{ (\varepsilon)}$ represents the instantaneous jump rate of the neuron $i$, that has been convoluted along the whole history of the process from $0$ to $t$. Secondly, define the microscopic convoluted rate profile as
\begin{equation}
\label{eq:conv_rate_profile}
V_{ t}^{ (\varepsilon)}(x):= \int_{ 0}^{t} e^{ -(t-s)} f \left(U_{ s}^{ (\varepsilon)}(x)\right) {\rm d}s + v_{ 0}^{ (\varepsilon)}(x) e^{ -t},\ x\in \mathbb{ R}
\end{equation}
where the voltage profile $U_{ t}^{ (\varepsilon)}$ is given in \eqref{eq:profileU} and for all $ \varepsilon>0$, $v_{ 0}^{ (\varepsilon)}$ is a continuous profile on $ \mathbb{ R}$.
\end{definition} 
\begin{remark}
Recalling the definition of $I_{ i}^{ (\varepsilon)}$ in \eqref{eq:def_I_i}, it is immediate to see that, on each $I_{ i}^{ (\varepsilon)}$, $ V_{ t}^{ (\varepsilon)}$ in \eqref{eq:conv_rate_profile} coincides with $V_{ i, t}^{ (\varepsilon)}+ v_{ 0}^{ (\varepsilon)} e^{ -t}$ in \eqref{eq:conv_rate_i}.
\end{remark}
A second observation is that in the neutral case $c=0$, the traveling profile of \eqref{eq:rateNFE} is simply written as $ \hat{ v}= f \left( \hat{ u}\right)$ (recall \eqref{eq:hatv}). Hence a second natural microscopic candidate is directly the rate profile itself
\begin{definition}[Rate profile]
\label{def:rate_profile}
For any $ \varepsilon>0$, introduce the microscopic rate profile of the Hawkes system \eqref{eq:Zi_def} as
\begin{equation}
\label{eq:Lambda}
\Lambda_{ t}^{ (\varepsilon)}= \sum_{ j\in \mathbb{ Z}} \lambda_{ j, t}^{ (\varepsilon)} \mathbf{ 1}_{ I_{ j}^{ (\varepsilon)}} = f \left(U_{ t-}^{ (\varepsilon)}\right),\ t\geq0.
\end{equation}
\end{definition}

The main result of the present section is the following:
\begin{theorem}
\label{th:main_rate}
Place ourselves under the same assumptions as for Theorem~\ref{th:main}. Suppose in addition that the initial profile $ v_{ 0}^{ (\varepsilon)}$ in \eqref{eq:conv_rate_profile} satisfies
\begin{equation}
\label{eq:v0eps}
\left\Vert v_{ 0}^{ (\varepsilon)} - \hat{ v}\right\Vert_{ L^{ 2}} \to 0,\ \text{ as }\varepsilon\to0,
\end{equation}
where $ \hat{ v}$ is the traveling wave profile to the RNFE \eqref{eq:rateNFE} defined by \eqref{eq:hatv}.
Then the rate profile $ \Lambda^{ (\varepsilon)}$ in Definition~\ref{def:rate_profile} and the convoluted rate profile $V^{ (\varepsilon)}$ in Definition~\ref{def:conv_rate_profile} satisfy: 
\begin{enumerate}
\item Convergence on a bounded time interval: for any $T_{ 0}>0$, $ \eta>0$,
\begin{align}
\mathbb{P}\left(\sup_{0\leq t\leq T_{ 0}} \left\Vert \Lambda_{t}^{ (\varepsilon)}- \hat{ v}\right\Vert_{L^2} \geq \eta\right)& \underset{\varepsilon\rightarrow 0}{\longrightarrow} 0, \label{eq:Lambda_convT0}\\
\mathbb{P}\left(\sup_{0\leq t\leq T_{ 0}} \left\Vert V_{t}^{ (\varepsilon)}- \hat{ v}\right\Vert_{L^2} \geq \eta\right)& \underset{\varepsilon\rightarrow 0}{\longrightarrow} 0.\label{eq:V_convT0}
\end{align}
\item Brownian wandering along $ \mathcal{ M}_{ rate}$ (recall \eqref{eq:Mrate}) on a diffusive time scale: for any $t_f>0$, for the same càdlàg process $ \psi^{ (\varepsilon)}$ as in Theorem~\ref{th:main}, for the same constant $ \varrho>0$ given by \eqref{eq:sigma} \change{and the same drift $\bar v$}, for all $\eta>0$,
\begin{align}
\label{eq:conv_main_Lambda}
\mathbb{P}\left(\sup_{0\leq u\leq t_f } \left\Vert \Lambda_{u\varepsilon^{-1}}^{ (\varepsilon)}- \hat{ v} \left(\cdot -   \change{u \bar v} - \varrho\psi_{u}^{ (\varepsilon)}\right)\right\Vert_{L^2} \geq \eta\right) &\underset{\varepsilon\rightarrow 0}{\longrightarrow} 0.
\end{align}
\end{enumerate}
 \end{theorem}
 \begin{remark}
The convoluted rate profile $ V^{ (\varepsilon)}$ is a less natural candidate for the approximation of $ \mathcal{ M}_{ rate}$ on a long time scale: following the same lines as in the proof of Theorem~\ref{th:main_rate}, one can however prove that
 \begin{equation*}
 \mathbb{P}\left(\sup_{0\leq u\leq t_f } \left\Vert V_{u\varepsilon^{-1}}^{ (\varepsilon)}- \hat{ V}_{ u \varepsilon^{ -1}}^{ (\varepsilon)}\right\Vert_{L^2} \geq \eta\right) \underset{\varepsilon\rightarrow 0}{\longrightarrow} 0,
 \end{equation*}
 for the process 
 \begin{equation*}
 \hat{ V}_{ t}^{ (\varepsilon)}= \int_{ 0}^{t} e^{ -(t-s)} \hat{ v} \left(\cdot +  \change{ \varepsilon s \bar v} + \varrho\psi_{ \varepsilon s}^{ (\varepsilon)}\right) {\rm d}s+ e^{ -t} \hat{ v}.
 \end{equation*}
The simplicity of the expression of $ \hat{V}^{ (\varepsilon)}$ allows however for a similar approach as we have done for the voltage profile $U^{ (\varepsilon)}$ itself: using the same arguments, it should be easy to prove in a same way that $ \sup_{ u\in [0, t_{ f}]}{\rm dist}_{ L^{ 2}} \left(\hat{ V}_{ u \varepsilon^{ -1}}, \mathcal{ M}_{ rate}\right)\to 0$ in probability.
 \end{remark}
 \begin{proof}[Proof of Theorem~\ref{th:main_rate}]
 \begin{align*}
V_{t}^{ ( \varepsilon)}(x) - \hat{ v}(x)&= \int_{ 0}^{t} e^{ - \left(t -s\right)} \left\lbrace f \left( U_{ s}^{ (\varepsilon)}(x)\right) - f \left( \hat{ u}(x)\right)\right\rbrace {\rm d}s+ \left(v_{ 0}^{ (\varepsilon)}(x)- \hat{ v}(x)\right) e^{ -t}
\end{align*}
and the result follows again from the Lipschitz continuity of $f$, Theorem~\ref{th:main} and \eqref{eq:v0eps}.
 \end{proof}

\subsection{Strategy of proof}
The proof of the long-time result of Theorem~\ref{th:main} relies on the observation that on a bounded interval $[0,T]$, the dynamics of the voltage profile $ U^{ (\varepsilon)}$ is essentially driven by the deterministic dynamics of the NFE: locally around $ \mathcal{ M}$, one can write (see Proposition~\ref{prop:mart_Y} for a precise statement) $ {\rm d} \left(U_{ t}^{ (\varepsilon)}- \hat{ u}\right) = \mathcal{ L} \left(U_{ t}^{ (\varepsilon)} - \hat{ u}\right){\rm d}t + R^{ (\varepsilon)}_{ t} {\rm d}t + M_{ t}^{ (\varepsilon)}$, where $ \mathcal{ L}$ is the linearized operator around $ \hat{ u}\in \mathcal{ M}$ (see \eqref{eq:Lvarphi} for a definition), $R^{ (\varepsilon)}$ gathers quadratic remainder terms and $M^{ (\varepsilon)}$ is the jump-noise part. By the stability result of Theorem~\ref{th:stability}, the approximate dynamics $ {\rm d} \left(U_{ t}^{ (\varepsilon)}- \hat{ u}\right) \approx \mathcal{ L}   \left(U_{ t}^{ (\varepsilon)}- \hat{ u}\right) {\rm d}t $ is attractive. The point is then to carefully control the noise (this is done in Section~\ref{sec:control_noise}) and the remainder parts and to proceed by iteration on intervals of order $1$, $[kT, (k+1)T]$: once, with high probability, $U^{ (\varepsilon)}$ is close to $ \mathcal{ M}$ at time $t=kT$ (at a distance at most $ \varepsilon^{ \alpha}$ for some $ \alpha>0$), the deterministic dynamics will keep $ U^{ (\varepsilon)}$ at a slightly larger distance from $ \mathcal{ M}$ (at most $ \varepsilon^{ \alpha^{ \prime}}$ say, for some $ \alpha^{ \prime}< \alpha$) on $[kT, (k+1)T]$. Choosing then $T$ sufficiently large (see the condition \eqref{hyp:T}) so that we are sure that $U^{ (\varepsilon)}$ goes back to a neighborhood of size $ \varepsilon^{ \alpha}$ at time $t=(k+1)T$, we can proceed with a recursion on $k$.

A second part of the analysis consists in deriving the dynamics of the phase of the process $U^{ (\varepsilon)}$ along $ \mathcal{ M}$ and to check that this phase is approximately Brownian on a time scale of $ \varepsilon^{ -1}$. This is done through the use of the isochron $ \Theta$ associated to \eqref{eq:NFE conv} (see Section~\ref{sec: stab, isochron}) and an application of Ito's formula to $ \Theta \left(U_{ t}^{ (\varepsilon)}\right)$.

This program is somehow similar to previous longtime analysis, mostly in the context of mean-field diffusions, see Section~\ref{sec:literature} for further details. Note that the fact that we are not dealing here with a mean-field system (but with an infinite number of particles on the real line) causes certain disturbances to the above program. The control of the distance of $ U^{ (\varepsilon)}$ to $ \mathcal{ M}$ requires itself the a priori knowledge that the phase \change{$ \Theta( U^{ (\varepsilon)})$} does not derive too much from the bulk $ \mathcal{ D}_{ 0}^{ (\varepsilon)}$ so that the two previous parts are actually quite intertwined. This together with the control of boundary conditions for the particle system entails some significant technical difficulties that were not present in previous papers.

\subsection{Position of the result w.r.t. the literature}
\label{sec:literature}
\subsubsection{The Neural Field Equation as a prototype for the emergence of structured patterns and sensitivity to noise of traveling waves}
Structured patterns are ubiquitous in cortical areas, see e.g. \cite{ERMENTROUT200133}, \cite[Fig.~4]{MR2871421} or \cite{Huang2010}. Since the seminal works of Wilson and Cowan \cite{Wilson:1972aa} and Amari \cite{Amari:1977:DPF:2731211.2731248}, Neural Field Equations have proven to reproduce a wide variety of such patterns, such as stationary states \cite{Faugeras2008}, traveling waves \cite{ermentrout_mcleod_1993}, localized patterns (bumps, pulses) \cite{101137130918721,doi:10.1137/18M1165797,Faye2013,Pinto2001,Faugeras2008,Veltz2010} and spiral waves \cite{Laing2005}. 

The version that is studied in the present paper corresponds to a simple instance of the NFE, in dimension $d=1$ within an excitatory network (i.e. the interaction kernel $W$ is nonnegative). Although this setting would seem simplistic from a biological perspective, this is the case which gathers most of the existing rigorous mathematical results. As already mentioned, we strongly base our analysis on two seminal works: firstly, the existence of traveling waves for the NFE established by Ermentrout and McLeod \cite{ermentrout_mcleod_1993} (see also \cite{chen1997} for related results) and secondly, the linear stability of these traveling waves proven by Lang and Stannat \cite{2016JDE...261.4275L}. Note also that the existence of traveling waves can be made rigorous and explicit in the particular case where the function $f$ is Heaviside, see e.g. \cite{Bressloff2015} for details. Extensions of such results (even the mere existence of traveling waves) to higher dimensions and to situations where inhibition is present remain poorly understood from a rigorous mathematical perspective (in this direction, see \cite{Pinto2001} for an analysis of an Excitatory-Inhibitory network).

Another simple model corresponds to the so-called Ring model, that has applications to the visual cortex \cite{Kilpatrick2013,Veltz2010,AgatheNerine2025}: in this case, the integration domain in \eqref{eq:NFE conv} is the one-dimensional circle and the integration kernel is $2\pi$-periodic, a simple instance being $W(x-y)=\cos(x-y)$. Here, periodicity does not allow for traveling waves, one obtains instead wandering bumps, whose existence and stability have been proven in \cite{Kilpatrick2013} when $f$ is Heaviside and in \cite{AgatheNerine2025} when $f$ is sigmoid. The long term stability of Hawkes approximation for the Ring model has been proven by Agathe-Nerine in \cite{AgatheNerine2025} and one can see the present paper as a continuation of \cite{AgatheNerine2025}. Note that the Ring model combines spatial translation invariance (as in our  case) and contrary to our case, a spatial domain with finite Lebesgue-measure. From the perspective of the Hawkes approximation, this last feature makes the analysis in \cite{AgatheNerine2025} considerably simpler than the present one, as the original approach of Duarte et al. \cite{CHEVALLIER20191} by mean-field Hawkes processes is sufficient for \cite{AgatheNerine2025}. 

\subsubsection{Stability of structured patterns under noise}
There is biological evidence of stochasticity in neuronal activity \cite{Luo2020}. A natural question is the influence of noise on the stability of structured patterns for the NFE \eqref{eq:NFE conv}. A longstanding approach in this direction has been to consider macroscopic perturbations to the NFE, i.e. Stochastic Neural Field Equations (SNFE), generically written as
\begin{equation}
\label{eq:SNFE}
\partial_{ t} U_t(x) = - U_t(x) + \int_{ \mathbb{ R}} W(x-y) f(U_t(y)) {\rm d}y + \varepsilon\sigma(U_{ t}(x)) {\rm d} \change{ \mathcal{ W}}(t,x)
\end{equation}
where \change{$ \mathcal{ W}$} is an appropriate $Q$-Wiener process in some Hilbert space $H$ (we refer to \cite{MR3367676} for a rigorous functional setting associated to \eqref{eq:SNFE}). The point is to question the influence of noise on the shape and the position of the traveling wave, in an asymptotic regime where the noise is small $ \varepsilon\to 0$, on a bounded time interval $[0,T]$. Namely, writing (with our notations) the solution $U=U^{ (\varepsilon)}$ to \eqref{eq:SNFE} as
\begin{equation}
\label{eq:asympt_exp_Ueps}
U^{ (\varepsilon)}(t,x) \approx \hat{ u}(x- C_{t}^{ (\varepsilon)}t) + \sqrt{ \varepsilon} \Phi^{ (\varepsilon)} \left(x- C_{ t}^{ (\varepsilon)},t\right),\ \text{ as } \varepsilon\to 0,
\end{equation}
the point is to derive asymptotic expansion in $ \varepsilon \to 0$ of the speed of traveling wave $ C^{ (\varepsilon)}$ and the shape profile $ \Phi^{ (\varepsilon)}$. This program has been carried out in a heuristic way by Bressloff and Webber \cite{Bressloff2012}, Bressloff and Kilpatrick in \cite{Bressloff2015} and by Kilpatrick and Ermentrout for the Ring model in \cite{Kilpatrick2013}. If one considers $U^{ (\varepsilon)}$ as a noisy version of the traveling wave profile $ \hat{ u}$ in \eqref{eq:asympt_exp_Ueps}, the difficulty is to make rigorous sense of the position of the noisy traveling wave $C_{ t}^{ (\varepsilon)}$. A common approach has been simultaneously followed by Inglis and MacLaurin \cite{10113715M102856X}, Kr\"uger and Stannat \cite{10113713095094X,Kruger:2017aa} and Lang \cite{10113715M1033927} (see also \cite[Chap.~5]{https://doi.org/10.14279/depositonce-5019}). Although their results are nonintersecting, all of these works rely on a similar identification of the wave phase $C^{ (\varepsilon)}$ as a solution to a minimisation problem 
\begin{equation} 
\label{eq:minim_Ceps}
C^{ (\varepsilon)}_{ t} \approx \inf_{ \phi\in \mathbb{ R}} \left\lbrace \left\Vert U^{ (\varepsilon)}_{ t} - \hat{ u}(\cdot + \phi)\right\Vert^{ 2}_{ H} \right\rbrace
\end{equation} 
for some functional space $H$ (see the aforementioned references for details). Note however that this approach seems to be intrinsically valid for bounded time intervals $[0,T]$ but maybe not suitable for a long time analysis.

Our approach significantly differs from the previous setting. Whereas the previous analysis is made by adding directly noise to the macroscopic NFE on a bounded time interval, the present approach analyses the influence of noise from a microscopic point of view (that is at the level of the particle system of Definition~\ref{def:model}) and noise only becomes macroscopic on a time scale that is unbounded w.r.t. the size of the population. Note that our approach for the identification of the phase does not rely on any $L^{ 2}$-minimisation procedure as for \eqref{eq:minim_Ceps}, but on a direct application of the isochron associated to \eqref{eq:NFE conv}. The possibility of drawing rigorous connections between the SPDE approach and our results  seems unclear to us.

\subsubsection{NFE as mean-field limits of Hawkes processes}
Neural Field Equations have been originally derived as informal limits of piecewise deterministic Markov processes (see eg. \cite{MR2871421} for a review). Rigorous approaches in this direction have been made by Riedler and Buckwar \cite{Riedler2013} (for NFEs on a bounded domain) and Lang and Stannat \cite{Lang2017}. We are here concerned with the approximation of the NFEs in terms of Hawkes processes, \cite{MR0378093}. The mathematical derivation of (homogeneous) mean-field Hawkes processes goes back to the seminal work of Delattre, Fournier and Hoffmann \cite{MR3449317} (the mean-field limit being then the Nonlinear Renewal Equation, that can be seen as the non-spatial version of the NFEs). As already mentioned earlier, Chevallier, Duarte, L\"ocherbach and Ost in \cite{CHEVALLIER20191} have been the first to address rigorously the approximation on a bounded time frame of a voltage based NFE \eqref{eq:NFE_CDLO} (recall here the comments made in Section~\ref{sec:micro_model} concerning the difficulties of this approach for the study of traveling waves). Note here that we are not aware of any previous results addressing a similar approach for the rate based NFE \eqref{eq:rateNFE}. In this sense, the result of Theorem~\ref{th:main_rate} seems to be new, even on a bounded time horizon. These propagation of chaos results for Hawkes processes have been generalized in many directions, e.g. in a multi-populations setting \cite{DITLEVSEN20171840,Locherbach2017}. Going beyond finite time propagation of chaos is a difficult task for mean-field models in general. A first attempt in this direction is to provide fluctuations results, see e.g.  \cite{2019arXiv190906401C,Heesen2021}.  As already mentioned, one can see the present work as a natural continuation of the works of Agathe-Nerine who has proven long time results for mean-field Hawkes models with a unique equilibrium \cite{agathenerine22} and with rotation invariance \cite{AgatheNerine2025}. Although the techniques used here share some similarities with \cite{AgatheNerine2025}, passing from the circle to the real line presents significant technical difficulties (one of them being the control of the boundary conditions), all linked to the fact that our present model is essentially no longer mean-field, contrary to the mentioned references.

\subsubsection{Long-time behaviors of stable normally hyperbolic manifolds}
The present paper borrows techniques used for the longtime dynamics of interacting particle systems under the presence of a stable macroscopic manifold. By the local stability result of Theorem~\ref{th:stability}, the manifold $ \mathcal{ M}$ of traveling wave profiles defined in \eqref{eq:M} is a simple prototype of a Stable Normally Hyperbolic Manifold (SNHM) \cite{Bates1998,sell2013dynamics}. The question is then to understand the behavior of a particle system subject to noise, in the vicinity of such a SNHM. Such an analysis, originally initiated in the context of Ginzburg-Landau equations, see eg. \cite{MR1340032,MR1337253,MR2448321,MR3023410}, has been in recent years successfully applied to mean-field diffusions in the presence of a stable limit cycle (eg. phase oscillators \cite{BGP,GPP2012,doi:10.1137/110846452,MR3336876,Giacomin:2015ab,MR3689966} and FitzHugh-Nagumo systems \cite{Lucon2019,2021arXiv210702468L,Lucon2021}). The use of the isochron \cite{Guckenheimer1975,Adams2025} in this context was initiated in \cite{Giacomin:2015ab} for small noise behavior of finite dimensional diffusions. 

\subsection{ Organisation of the paper}
\label{sec:organisation}
The rest of the paper is organised as follows: Section~\ref{sec: stab, isochron} gathers estimates concerning the stability of the manifold $ \mathcal{ M}$ of traveling wave profiles together with regularity estimates of the isochron around $ \mathcal{ M}$. Section~\ref{sec:iter scheme} introduces the semimartingale decompositions of the voltage profile and the isochron around $ \mathcal{ M}$ and the main iterative scheme at the basis of the stability result for the voltage profile of Theorem~\ref{th:main}. Section~\ref{sec:control_noise} is devoted to the control of the noise terms in these decompositions. The proximity of the process w.r.t. $ \mathcal{ M}$ is proven in Section~\ref{sec:proximity_M}. Section~\ref{sec:proof_main} is devoted to the proof of the main Theorem~\ref{th:main}. Section~\ref{sec: boundary} contains the technical boundary estimates that are necessary in the previous sections.  \change{Section~\ref{sec:extensions} gives some perspectives to the present work.} Appendix~\ref{sec:app_auxiliary} contains several auxiliary technical estimates. Appendix~\ref{sec:app rate based NFE} contains the proof of Proposition~\ref{prop:correspondanceNFE}. Appendix~\ref{sec:app regularity isochron} contains the proof of Proposition~\ref{prop:regularity isochron}.

\change{In the following, we may drop without further notice the exponent $(\varepsilon)$ in variables (e.g. writing  $ \lambda_{ i,t}$ (resp. $U_{ t}$) in place of $ \lambda_{ i,t}^{ (\varepsilon)}$ (resp. $U_{ t}^{ (\varepsilon)}$)) in absence of ambiguity, for simplicity of notations.}

\section{Stability and isochron map.}\label{sec: stab, isochron}

In this Section we describe in more details the stability results for the NFE \eqref{eq:NFE conv} and the properties of the isochron map we will extensively use in the proof of our result. Let us define, for $ \varphi\in \mathbb{ R}$, the operator $\mathcal{L}_{ \varphi}$ of the linearized dynamics of \eqref{eq:NFE conv} in the neighborhood of $\hat u_\varphi$
\begin{equation}
\label{eq:Lvarphi}
\mathcal{L}_{ \varphi}v= -v + W\ast \left(f^{ \prime} \left( \hat{ u}_{ \varphi}\right) v \right).
\end{equation}
It is easy to check that $\mathcal{L}_\phi \hat u'_\phi=0$ and that this operator is bounded and self-adjoint in the weighted space $L^{ 2}_{ m_\varphi}$ endowed with the scalar product
\begin{equation}
\label{eq:L2_weight}
\left\langle v,w\right\rangle_{ m_\varphi}:= \int_{ \mathbb{ R}} v(x) w(x) m_{\varphi}(x) {\rm d}x,\qquad  m_{\varphi}(x):= f^{ \prime}( \hat{ u}_{\varphi}(x))>0.
\end{equation}
Since $f^{ \prime}$ is bounded above and below \change{on $[a_{ 1}, a_{ 2}]$, by Assumptions~\ref{ass:Wf} and~\ref{ass:neutral}}, $L^{ 2}_{ m_{ \varphi}}$ is equivalent to the usual $L^{ 2}$ space, uniformly on $ \varphi$: we have, 
\begin{equation}
\label{eq:norm_equiv}
\sqrt{ f^{ \prime}(a_{ 1})}\left\Vert v \right\Vert_{ L^{ 2}} \leq \left\Vert v \right\Vert_{ m_{ \varphi}} \leq \sqrt{ f^{ \prime}( \change{a})}\left\Vert v \right\Vert_{ L^{ 2}}.
\end{equation} 
Let us denote the projections, for $\phi\in\mathbb{R}$ and $v\in L^2$,
\begin{equation}
\label{eq:projectionsL2}
P_\phi v = \frac{\left\langle v,\hat u'_\phi\right\rangle_{m_\phi}}{\left\langle \hat u'_\phi,\hat u'_\phi\right\rangle_{m_\phi}} \change{\hat{u}'}_\phi, \quad P^\perp_\phi v = v - P_\phi v.
\end{equation}
The following theorem corresponds to Theorems~4 and~8 in \cite{2016JDE...261.4275L} (see also Theorem 4.1.2 and Corollary 4.2.4 in \cite{https://doi.org/10.14279/depositonce-5019}).
\begin{theorem}[Lang and Stannat \cite{2016JDE...261.4275L}]
\label{th:stability}
There exists $\kappa>0$ such that for all $\varphi\in \mathbb{R}$ and all $v\in L^2$
\begin{equation}
\label{eq:spectralgap_LangStannat}
\left\langle \mathcal{L}_{\varphi} v, v\right\rangle_{ m_\varphi} \leq -\kappa\left(\Vert v\Vert_{m_\varphi}^2-\left\langle v\, ,\, \hat u'_\varphi\right\rangle_{ m_\varphi}^{ 2}\right),
\end{equation}
and thus there exists a constant $C_\mathcal{L}>0$ such that for any $\varphi\in \mathbb{R}$ and any $v\in L^2$ satisfying $P_\varphi v=0$ we have for all $t\geq 0$
\begin{equation}
\label{eq:contract_ecL}
\left\Vert e^{t \mathcal{L}_\varphi} v\right\Vert_{L^2} \leq C_\mathcal{L} e^{-\kappa t}\left\Vert v\right\Vert_{L^2}.
\end{equation}
Moreover there exist constants $c, \mu>0$ such that, for all $z \in \mathbb{ R}$,
\begin{equation}
\label{eq:uprime_exp}
0\leq \hat{ u}^{ \prime}_0(z)\leq c e^{ - \mu \left\vert z \right\vert}.
\end{equation}
\end{theorem}
It is well known (see e.g. \cite[Th.~4.6]{GPP2012}, \cite[Cor.~3.6]{AgatheNerine2025} or \cite[Th.~3.1]{Adams2025} for similar results that can be easily adapted to the present case) that the linear stability property stated in this Theorem ensures the local stability of $\mathcal{M}$ for \eqref{eq:NFE conv}. More precisely there exists an isochron map $\Theta$, defined in a tubular neighborhood $\mathcal{N}$ of $\mathcal{M}$,
\begin{equation}
\label{eq:mcN}
\mathcal{N} = \{v:\, \exists \varphi\in \mathbb{R},\, \Vert v-\hat u_\varphi\Vert_{L^2} \leq C_0\},
\end{equation}
such that for $0<\lambda<\kappa$ there exists a constant $C_\lambda>0$ such that for any initial condition $u_0\in \mathcal{N}$ we have
\begin{equation}\label{eq:convergence ut}
\left\Vert u_t-\hat u_{\Theta(u_0)}\right\Vert_{L^2}\leq C_\lambda e^{-\lambda t} \left\Vert u_0-\hat u_{\Theta(u_0)}\right\Vert_{L^2}.
\end{equation}
In particular, taking $u_{ 0}= \hat{ u}_{ \varphi}$ for some $ \varphi\in \mathbb{ R}$, one has $ u_{ t}= \hat{ u}_{ \varphi}$ for all $t\geq0$ so that $ \Theta( \hat{ u}_{ \varphi})= \varphi$ for all $ \varphi\in \mathbb{ R}$.
The isochron map satisfies the following properties (the proof of this Proposition in given in Appendix \ref{sec:app regularity isochron}).
\begin{proposition}\label{prop:regularity isochron}
The map $u\mapsto \Theta(u)$ is $C^2$ from $\mathcal{N}$ to $\mathbb{R}$.
Moreover, for any $u\in \mathcal{N}$, \change{$h^{ 1}, h^{ 2}\in L^2$} and $h^3\in L^2\cap L^\infty$, the map $u\mapsto {\rm D}^2\Theta(u)[h^2,h^3]$ is Frechet differentiable and ${\rm D}^3 \Theta(u)[h^1,h^2,h^3]:={\rm D} {\rm D}^2\Theta(u)[h^2,h^3][h^1]$ satisfies
\begin{equation}\label{eq:bound D3Theta}
\left\Vert {\rm D}^3 \Theta(u)[h^1,h^2,h^3]\right\Vert_{2}\leq C_\Theta \Vert h^1\Vert_{L^2}\Vert h^2\Vert_{L^2}(\Vert h^3\Vert_{L^2}+\Vert h^3\Vert_{L^\infty}).
\end{equation}
Moreover,
\begin{equation}
\label{eq:uniform_bound_DTheta}
\sup_{ u_{ 0}\in \mathcal{ N}} \left\Vert {\rm D} \Theta(u_{ 0}) \right\Vert< +\infty\text{ and }\sup_{ u_{ 0}\in \mathcal{ N}} \left\Vert {\rm D}^{ 2} \Theta(u_{ 0}) \right\Vert< +\infty,
\end{equation}
and thus (modifying the constant $C_{ \Theta}$ if necessary), for all $v\in \mathcal{ N}$, 
\begin{equation}
\label{eq:Theta_dist}
\left\Vert v - \hat{ u}_{ \Theta(v)} \right\Vert_{ L^{ 2}} \leq C_{ \Theta} {\rm dist}_{ L^{ 2}} \left(v, \mathcal{ M}\right).
\end{equation}
Secondly, for any $\varphi\in \mathbb{R}$ and any $v,w\in L^2$ we have
\begin{equation}
\label{eq:DTheta}
{\rm D}\Theta(\hat u_\phi)[v] = - \frac{\left\langle v,\hat u'_\phi\right\rangle_{m_\phi}}{\left\langle \hat u'_\phi,\hat u'_\phi\right\rangle_{m_\phi}},
\end{equation}
and
\begin{align}
\label{eq:D2Theta}
{\rm D}^2\Theta(\hat u_\phi)[v,w]=&-\frac{1}{\left\langle \hat u'_\phi,\hat u'_\phi\right\rangle_{m_\phi}} \int_0^\infty \int_\mathbb{R} f''(\hat u_\phi)\hat u'_\phi  e^{s\mathcal{L}_\phi}P^\perp_\phi v  \, e^{s\mathcal{L}_\phi}w \nonumber\\
&-\frac{1}{\left\langle \hat u'_\phi,\hat u'_\phi\right\rangle_{m_\phi}} \int_0^\infty \int_\mathbb{R} f''(\hat u_\phi)\hat u'_\phi e^{s\mathcal{L}_\phi} v  \, e^{s\mathcal{L}_\phi} P^\perp_\phi w\\
& \change{+}\frac{1}{\left\langle \hat u'_\phi,\hat u'_\phi\right\rangle_{m_\phi}} \int_0^\infty \int_\mathbb{R} f''(\hat u_\phi)\hat u'_\phi e^{s\mathcal{L}_\phi}P^\perp_\phi v  \, e^{s\mathcal{L}_\phi} P^\perp_\phi w.\nonumber
\end{align}
\end{proposition}

\medskip

In the following we will also make use of the truncated operator $L_{ \phi}$ defined as follows: set $L^{ 2}(\mathcal{ D}_{ 0})=L^{ 2}_{ m_{ \varphi}}(\mathcal{ D}_{ 0})$ the $L^{ 2}$-space of test functions with compact support in $ \mathcal{ D}_{ 0}$ \change{(recall \eqref{eq:partition_R})} endowed with the scalar product 
\begin{equation}
\label{eq:scalarproduct_D0}
\left\langle u\, ,\, v\right\rangle_{ m_\phi, \mathcal{ D}_{ 0}}:= \int_{ \mathcal{ D}_{ 0}} u \, v\,  m_\phi,
\end{equation}
and define $ L_{ \varphi}: L^{ 2}(\mathcal{ D}_{ 0}) \to L^{ 2} \left(\mathcal{ D}_{ 0}\right)$ given by
\begin{equation}
\label{eq:L0}
L_{ \phi}(u)(x):= \mathcal{L}_\phi \left(u\right)(x) \mathbf{ 1}_{ \mathcal{ D}_{ 0}}(x)= \mathcal{L}_\phi \left(u \mathbf{ 1}_{ \mathcal{ D}_{ 0}}\right)(x) \mathbf{ 1}_{ \mathcal{ D}_{ 0}}(x),\ u\in L^{ 2} \left(\mathcal{ D}_{ 0}\right).
\end{equation}
The truncated operator $L_\varphi$ inherits stability properties from the initial operator $\mathcal{L}_\varphi$, with additional boundary terms, as stated in the following Proposition. The proof of these results are given in Appendix~\ref{app:L0}.

\begin{proposition}
\label{prop:L0}
For any $\varphi\in \mathbb{R}$ the operator $L_{\phi}$ is bounded from $L^{ 2}( \mathcal{ D}_{ 0})$ into itself and selfadjoint for the scalar product \eqref{eq:scalarproduct_D0}. Moreover, $L_{ \varphi}$ is dissipative and hence generates a $C_{ 0}$-semigroup of contraction $ e^{ tL_{ \varphi}}$ in the following sense
\begin{align}
\label{eq:Dirform_Lphi}
\left\langle L_{ \varphi} v\, ,\, v\right\rangle_{ m_{ \varphi}, \mathcal{ D}_{ 0}} & \leq 0,\ v\in L^{ 2} \left(\mathcal{ D}_{ 0}\right),\\
\label{eq:bound_eLphi}
\left\Vert e^{ t L_{ \varphi}} v\right\Vert_{ m_{ \varphi}, \mathcal{ D}_{ 0}} &\leq \left\Vert v \right\Vert_{ m_{ \varphi}, \mathcal{ D}_{ 0}},\ t\geq0, \ v\in L^{ 2} \left(\mathcal{ D}_{ 0}\right).
\end{align} 
Secondly, $L_{ \varphi}$ satisfies the following spectral gap inequality, for the same constant $\kappa$ as in eq. \eqref{eq:spectralgap_LangStannat} of Theorem~\ref{th:stability},
\begin{equation}
\label{eq:SG_L0}
\left\langle L_{\phi}v , v\right\rangle_{m_\phi, \mathcal{ D}_{ 0}}\leq -\kappa\left( \left\Vert v \right\Vert_{m_\phi, \mathcal{ D}_{ 0}}^{ 2} -\left\langle v , \mathbf{ 1}_{ \mathcal{ D}_{ 0}} \hat{ u}'_\phi\right\rangle_{m_\phi, \mathcal{ D}_{ 0}}^{ 2}\right).
\end{equation}
Moreover if $v= \change{\tilde{ v} \mathbf{ 1}_{ \mathcal{ D}_{ 0}}}\in L^2(\mathcal{D}_0)$ \change{for some $ \tilde{ v}\in L^{ 2}$ such that $P_\phi \tilde{ v}=0$}, then 
\begin{equation}
\label{eq:etL0exp}
\change{\left\Vert e^{tL_\phi} v\right\Vert_{m_\phi,\mathcal{D}_0}\leq e^{-\kappa t}\Vert v\Vert_{m_\phi,\mathcal{D}_0}+ \left\Vert \hat u'_\phi {\bf 1}_{\mathcal{D}_0^c}\right\Vert_{m_\phi} \left(\left\Vert \tilde{ v} \mathbf{ 1}_{ \mathcal{ D}_{ 0}^{ c}}\right\Vert_{ m_{ \varphi}}  + c_{\mathcal L} t \left\Vert v \right\Vert_{m_\phi, \mathcal{ D}_{ 0}}\right)}.
\end{equation}
\end{proposition}

To define our iterative scheme in the following Section we will rely on the following projection defined on the tubular neighborhood $\mathcal{N}$ given by \eqref{eq:mcN} (reducing the radius of $\mathcal{N}$ if needed), whose proof is a consequence of the implicit function Theorem (see for example \cite[page 501]{sell2013dynamics} for a proof). 

\begin{lemma}\label{lem:proj}
There exists a positive constant $C_{\mathcal{M}}$ such that for any $u\in \mathcal{N}$, there exists a unique $\psi:={\rm proj}_{\mathcal{M}}(u)$ such that $P_\psi \left( \left(u-\hat u_\psi\right)\right)=0$, and moreover 
\begin{equation}
\label{eq:proj_dist}
\left\Vert u-\hat u_{\psi}\right\Vert_{m_\psi}\leq C_{\mathcal{M}}{\rm dist}_{L^2}(u,\mathcal{M}).
\end{equation} 
\end{lemma}

\section{Iterative scheme and evolution of the phase}\label{sec:iter scheme}

Our aim is to define an iterative scheme that will allow us to prove that the process $\left(U_t^{ (\varepsilon)} \right)_{ t\geq0}$ given by \eqref{eq:profileU} stays close to $\mathcal{M}$ for long times with high probability. Let us consider $\tau=\tau_\varepsilon$ the stopping time defined as (recall here that the constant $C_{ 0}$ below defines the neighborhood $ \mathcal{ N}$ of $ \mathcal{ M}$, see \eqref{eq:mcN})
\begin{equation}
\label{eq:tau}
\tau_{ \varepsilon} = \inf\left\{t\geq 0 : {\rm dist}_{L^2}(U_t^{ (\varepsilon)},\mathcal{M})\geq C_0\right\},
\end{equation}
and consider $T>0$ satisfying (recall the definitions of constants $C_\mathcal{M}$ in Lemma~\ref{lem:proj} and $ \kappa$ in Theorem~\ref{th:stability})
\begin{equation}\label{hyp:T}
 \frac{ C_{ \mathcal{ M}}}{ \sqrt{ f^{ \prime}(a_{ 1})}} e^{- \kappa T} \leq \frac{ 1}{ 3}.
\end{equation}
The parameter $T>0$ defined through \eqref{hyp:T} has to be thought as a typical relaxation time (concerning the dynamics driven by the linear operator $L_{  \varphi}$ given by \eqref{eq:L0}) towards the manifold of stationary profiles $ \mathcal{ M}$. The necessity of the above definition \eqref{hyp:T} for $T>0$ will become clear later in the proof of Proposition~\ref{prop:closeness M} below. The only thing that is needed here is to take $T$ sufficiently large but finite. Fix $t_{ f}>0$ as in the statement of Theorem~\ref{th:main} and define
\begin{equation}
\label{eq:n_varepsilon}
n_\varepsilon := \lfloor t_f\varepsilon^{-1}/T\rfloor,
\end{equation}
and $T_k=kT$ for $k\in \{0,1,\ldots, n_\varepsilon\}$. Then we consider, for $ {\rm proj}_{  \mathcal{ M}}$ given by Lemma~\ref{lem:proj}, for $k\in \{0,1,\ldots,  \lfloor \tau/T\rfloor\}$,
\begin{equation*}
\psi_k := {\rm proj}_\mathcal{M}(U_{T_{k}}) \quad \text{if} \quad  T_k\leq \tau, \quad \text{and} \quad \psi_k := {\rm proj}_\mathcal{M}(U_{\tau}) \quad \text{if} \quad  T_k>\tau.
\end{equation*}
Moreover, for each $k\in \{1,2,\ldots, n_\varepsilon\}$ and $t\in [0,T]$, if $T_{k-1}+t\leq \tau$,
\begin{equation}
\label{eq:Yk}
Y_{k,t}^{ (\varepsilon)}=Y_{k,t} := \left(U_{T_{k-1}+t}^{ (\varepsilon)}-\hat u_{\psi_{k-1}}\right){\bf 1}_{\mathcal{D}_0},
\end{equation}
and if $T_{k-1}+t\geq \tau$,
\begin{equation*}
Y_{k,t}^{ (\varepsilon)}=Y_{k,t} := \left(U_\tau^{ (\varepsilon)}-\hat u_{\psi_{k-1}}\right){\bf 1}_{\mathcal{D}_0}.
\end{equation*}

\begin{proposition}
\label{prop:mart_Y}
One has the following semimartingale decomposition for $Y_{k, t}$: for any $t\in [0,T]$ such that $T_{k-1}+t\leq \tau$,
\begin{align}
{\rm d} Y_{ k,t}(x)=& L_{\psi_{k-1}}Y_{ k,t}(x)  {\rm d}t + W\ast g_{k,t}(x){\bf 1}_{ \mathcal{ D}_{ 0}}(x) {\rm d}t \nonumber \\ &+\sum_{ i=-N_\varepsilon }^{N_\varepsilon}\Bigg(\varepsilon  \sum_{ j\in \mathbb{ Z}}W(x_i-x_j)f(U_{j,T_{k-1}+t^-}) -W* f \left( U_{T_{k-1}+t}\right)(x) \Bigg){\bf 1}_{I_i}(x){\rm d} t \nonumber\\
&+{\rm d} M_{T_{k-1}+t}(x),\ x\in \mathbb{ R}\label{eq:semimart_Y}
\end{align}
where we recall the definition of $L_{ \psi}$ in \eqref{eq:L0} and define
\begin{equation}
\label{eq:gpsi}
g_{k,t}(x):= Y_{k,t}^{ 2}(x) \int_{ 0}^{1} (1-r) f^{ \prime\prime} \left( \hat{ u}_{\psi_{k-1}}(x) + r Y_{k,t}(x)\right){\rm d}r,\ x\in \mathbb{ R}
\end{equation}
and
\begin{equation*}
M_t(x)=\varepsilon\sum_{ i=-N_\varepsilon }^{N_\varepsilon}\sum_{ j\in \mathbb{ Z}} W(x_i-x_j){\bf 1}_{I_{i}}(x) \change{\int_0^t \big({\rm d}Z_{j,s}-f(U_{j,s^-}){\rm d}s\big)},\ x\in \mathbb{ R}.
\end{equation*}
\end{proposition}

\begin{proof}[Proof of Proposition~\ref{prop:mart_Y}]
Note first that on $ \mathcal{ D}_{ 0}^{ c}$, ${\rm d} Y_{k, t}\equiv 0$, $ M_{ t} \equiv 0$,  $ \mathbf{ 1}_{ I_{ i}}\equiv 0$ for $i=-N_{ \varepsilon},\ldots, N_{ \varepsilon}$ and $L_{\psi_{k-1}}Y_{k,t} \equiv 0$, so that \eqref{eq:semimart_Y} is satisfied.
Concentrate now on $x\in \mathcal{ D}_{ 0}$. Since we have $0= - \hat{ u}_{\psi_{k-1}} + W\ast f \left( \hat{ u}_{\psi_{k-1}}\right)$, Itô's Lemma gives, for $x\in \mathcal{ D}_{ 0}$,
\begin{align}
{\rm d} Y_{k,t}(x)
&=- Y_{k, t}(x) {\rm d}t- W\ast f(\hat{ u}_{\psi_{k-1}})(x) {\rm d}t \nonumber\\
&\quad +\sum_{ i=-N_\varepsilon }^{N_\varepsilon}\Bigg(\varepsilon  \sum_{ j\in \mathbb{ Z}}W(x_i-x_j)f(U_{j,T_{k-1}+t^-})\Bigg){\bf 1}_{I_i}(x){\rm d} t +{\rm d} M_{T_{k-1}+t}(x). \label{eq:Yfirstdecomp}
\end{align}
Write the Taylor decomposition of $f$ around $\hat{ u}_{\psi_{k-1}}$ as
\begin{equation*}
f(U_{T_{k-1}+t}(y)) = f( \hat{ u}_{\psi_{k-1}}(y)) + f^{ \prime}( \hat{ u}_{\psi_{k-1}}(y)) Y_{k,t} + g_{k,t}(y),\ y\in \mathbb{ R}
\end{equation*} 
where the remainder term $g$ is given by \eqref{eq:gpsi}. Taking the convolution w.r.t. $W$, we obtain
\begin{equation}
\label{eq:Taylor_UN}
-W* f \left( \hat{ u}_{\psi_{k-1}}\right) =- W* f \left( U_{T_{k-1}+t}\right)+ W*\left( f'\left(\hat{ u}_{\psi_{k-1}}\right) Y_{ k,t}\right)+ W* g_{k,t}.
\end{equation}
Going back to \eqref{eq:Yfirstdecomp}, recalling that, for $x\in \mathcal{D}_0$ and $v\in L^{ 2} \left(\mathcal{ D}_{ 0}\right)$, $L_{\psi_{k-1}} v(x)=-v(x)+W*(f'(\hat{ u}_{\psi_{k-1}}) v)(x)$, we get
\begin{align*}
{\rm d} Y_{k,t}(x)=& L_{\psi_{k-1}} Y_{k, t}(x) {\rm d}t -W* f \left( U_{T_{k-1}+t}\right)(x) {\rm d}t + W* g_{k,t}(x) {\rm d}t \\
&+\sum_{ i=-N_\varepsilon }^{N_\varepsilon}\Bigg(\varepsilon  \sum_{ j\in \mathbb{ Z}}W(x_i-x_j)f(U_{j,T_{k-1}+t^-})\Bigg){\bf 1}_{I_i}(x){\rm d} t +{\rm d} M_{T_{k-1}+t}(x).
\end{align*}
The results follows then from the decomposition of $W* f \left( U_{T_{k-1}+t}\right)(x)$ onto $ \mathcal{ D}_{ 0}= \cup_{ i}I_{ i}$.
\end{proof}
As a consequence of Proposition~\ref{prop:mart_Y}, we obtain the mild formulation (see \cite[Lemma~3.2]{Zhu2017}) for $t\in [0,T]$ such that  $T_{k-1}+t\leq \tau$,
\begin{multline}
Y_{k,t} = \,  e^{t L_{\psi_{k-1}}}Y_{k,0} +\int_0^t  e^{ (t-s)L_{\psi_{k-1}}} W\ast g_{k,s}{\bf 1}_{ \mathcal{ D}_{ 0}} {\rm d}s\\
 +\sum_{ i=-N_\varepsilon }^{N_\varepsilon}\int_0^t e^{ (t-s)L_{\psi_{k-1}}}\Bigg(\varepsilon  \sum_{ j\in \mathbb{ Z}}W(x_i-x_j)f(U_{j,T_{k-1}+s^-}) -W* f \left( U_{T_{k-1}+s}\right) \Bigg){\bf 1}_{I_i}{\rm d} s\\
+\zeta_{k,t},\label{eq:Ykt mild}
\end{multline}
where the noise term in \eqref{eq:Ykt mild} is given by
\begin{equation}\label{eq:def zetakt}
\zeta_{k,t}^{ (\varepsilon)}=\zeta_{k,t}:= \sum_{ j\in \mathbb{Z}}\int_{ T_{k-1}}^{T_{k-1}+t} \int_{ 0}^{+\infty}e^{ (T_{k-1}+t-s)L_{\psi_{k-1}}} \varpi_{ j}^{ (\varepsilon)}\mathbf{ 1}_{ z\leq \lambda_{ j, s}} \tilde{ \pi}_{ j}({\rm d}s, {\rm d}z),
\end{equation}
for the following functions $ \varpi^{ (\varepsilon)}_{ j}$, that will appear repeatedly in the following
\begin{equation}
\label{eq:varpij}
\varpi_{j}^{ (\varepsilon)}=\varpi_{j}:= \varepsilon \sum_{ i=-N_{ \varepsilon}}^{ N_{ \varepsilon}} W(x_{ i}-x_{ j}) \mathbf{ 1}_{ I_{ i}^{ (\varepsilon)}} \in L^{ 2} \left(\mathcal{ D}_{ 0}\right) .
\end{equation}

On the other hand, a direct application of Ito's Lemma gives the following time evolution for the phase $\Theta(U_t)$, for $t\leq \tau$,
\begin{align}
\Theta(U_t)  = &\, \Theta(U_0) -\int_0^t {\rm D}\Theta(U_s)U_s  {\bf 1}_{\mathcal{D}_0} {\rm d} s \label{eq:Ito to Theta}\\
&+ \sum_{j\in\mathbb{Z}}\int_0^t \int_0^\infty\left(\Theta(U_s+\varpi_{j}\mathbf{ 1}_{ z\leq \lambda_{ j, s}} )-\Theta(U_s)\right){\rm d} s{\rm d} z+ \Xi_t,\nonumber
\end{align}
where
\begin{equation}
\label{eq:Nt}
\Xi_t^{ (\varepsilon)}= \Xi_t := \sum_{j\in\mathbb{Z}}\int_0^{t\wedge \tau}\int_0^\infty \left(\Theta(U_s+\varpi_{j}\mathbf{ 1}_{ z\leq \lambda_{ j, s}} )-\Theta(U_s)\right)\tilde \pi_j({\rm d} s,{\rm d}z).
\end{equation}

\section{Control of the noise}
\label{sec:control_noise}
The aim of this section is to establish a-priori bounds on the noise terms $ \zeta_{ k}^{ (\varepsilon)}$ in \eqref{eq:def zetakt} and $ \Xi^{ (\varepsilon)}$ in \eqref{eq:Nt}. A first technical tool in this direction is the following result concerning the control of  $ \varpi_j^{ (\varepsilon)}$ in \eqref{eq:varpij}. Lemma~\ref{lem:bound_sum_varpi} will be of constant use in the rest of the paper.
Recall that $ \beta>0$ is the parameter in the definition of $\ell_{ \varepsilon}$ in \eqref{eq:N_ell_eps} satisfying \eqref{eq:beta_small}.
\begin{lemma}
\label{lem:bound_sum_varpi}
We have the following bounds: 
\begin{enumerate}
\item There exist  $ C>0$ and $ \varepsilon_{ 0}>0$ only depending on $W$ such that for all $ \varepsilon\in (0, \varepsilon_{ 0})$, for all $j\in \mathbb{ Z}$,
\begin{equation}
\label{eq:estim_varpi_j}
\Vert \varpi_{j}^{ (\varepsilon)}\Vert_{L^2} \leq
\begin{cases}
C \varepsilon W(x_j-x_{-N_\varepsilon}) & \text{if } j<-N_\varepsilon,\\
C \varepsilon & \text{if } -N_\varepsilon\leq j\leq N_\varepsilon,\\
C \varepsilon W(x_j-x_{N_\varepsilon}) & \text{if } j>N_\varepsilon.
\end{cases}
\end{equation}
In particular, there exists a constant $C>0$ such that
\begin{equation}
\label{eq:bound_supvarpi}
\sup_{ j\in \mathbb{ Z}} \left\Vert \varpi_{j}^{ (\varepsilon)} \right\Vert_{L^2} \leq C \varepsilon.
\end{equation}
\item For any integer $k\geq 1$, there exist $C>0$ and $ \varepsilon_{ 0}>0$ depending only on $W$ and $k$ such that for all $ \varepsilon\in (0, \varepsilon_{ 0})$
\begin{align}
\sum_{ j\in \mathbb{Z}} \Vert \varpi_{j}^{ (\varepsilon)}\Vert_{L^2}^k &\leq C \varepsilon^{ k-1- \beta}. \label{eq:bound sum vert pi}
\end{align}
\end{enumerate}
\end{lemma}
The proof of Lemma~\ref{lem:bound_sum_varpi} is postponed to Appendix~\ref{sec:estimates_W_varpi}.
Another technical tool that we will use in this Section is the following standard concentration result (see e.g. \cite[Corollary~5]{PSMIR_1995___2_A3_0}):
\begin{proposition}[Concentration for martingale]
\label{prop:concentration_mart}
Let $(M_{ t})_{ t\geq0}$ a real square integrable local martingale. Then, for all $ \eta>0$, $a>0$
\begin{equation}
\label{eq:concentration_M}
\mathbb{ P} \left( \sup_{ s\leq t} \left\vert M_{ s} \right\vert\geq \eta\right) \leq 2 \exp \left(- \Psi \left( \frac{ \eta}{ a}\right)\right) + 2 \mathbb{ P} \left( \left\langle M\right\rangle_{ t} > a^{ 2}\right) + \mathbb{ P} \left( \sup_{ s\leq t} \left\vert \Delta M_{s}\right\vert> a\right),
\end{equation}
for 
\begin{equation}
\label{eq:Psi}
\Psi(x):= (x+1) \ln \left(x+1\right)-x.
\end{equation}
\end{proposition}

\subsection{Control of the noise term \texorpdfstring{$\Xi^{ (\varepsilon)}$}{Xieps}}
We first begin with the a-priori control on $\Xi^{ (\varepsilon)}$ defined in \eqref{eq:Nt}:
\begin{proposition}\label{prop:bound N, R}
The noise term $ \Xi^{ (\varepsilon)}$ in \eqref{eq:Nt} satisfies the following: there exists some $C,c>0$ depending only on $W$, $f$ and $t_f$ such that for $ \delta >0$, the event $ \mathcal{ A}_{ \varepsilon, \delta}$ given by
\begin{equation}
\label{eq:event_mcA}
\mathcal{A}_{\varepsilon, \delta} = \left\{\sup_{ 0\leq t \leq \varepsilon^{-1}t_f} \left\vert \Xi_t^{ (\varepsilon)}\right\vert\leq C \varepsilon^{-\frac{\beta}{2}- \delta}\right\},
\end{equation}
we have, for \change{sufficiently small} $\varepsilon >0$,
\begin{align}
\mathbb{ P} \left(\mathcal{A}_{ \varepsilon, \delta}\right) \geq 1- e^{-c \varepsilon^{- \delta}} \label{eq:concentration_Nt}.
\end{align}
\end{proposition}
\begin{proof}[Proof of Proposition~\ref{prop:bound N, R}]
We drop the dependence in $ \varepsilon$ for readability. We have
\begin{align}
\langle \Xi\rangle_t &= \sum_{j\in\mathbb{Z}}\int_0^{t\wedge \tau}\int_0^\infty  \left(\Theta(U_s+\varpi_{j}\mathbf{ 1}_{ z\leq \lambda_{ j, s}} )-\Theta(U_s)\right)^2 {\rm d} z{\rm d } s \nonumber\\
&\leq  2\sum_{j\in\mathbb{Z}}\int_0^{t\wedge \tau} {\rm D}\Theta (U_s) \left[\varpi_j\right]^2 {\rm d} s \nonumber\\
&\quad +2\sum_{j\in\mathbb{Z}}\int_0^{t\wedge \tau}\int_0^\infty \left(\Theta(U_s+\varpi_{j}\mathbf{ 1}_{ z\leq \lambda_{ j, s}} )-\Theta(U_s)-{\rm D} \Theta(U_s)[\varpi_{j}\mathbf{ 1}_{ z\leq \lambda_{ j, s}} ]\right)^2{\rm d} z{\rm d } s, \label{eq:cro_Xi}
\end{align}
where we have used that $ \lambda_{ j,s}\leq 1$ in the first term of \eqref{eq:cro_Xi}. Concentrate on the first term above: by Riesz representation theorem, there exists some $ \nabla \Theta \left(U_{ s}\right)\in L^{ 2}$ such that for all $h\in L^{ 2}$, $ {\rm D} \Theta \left(U_{ s}\right)[h]= \left\langle \nabla \Theta \left(U_{ s}\right)\, ,\, h\right\rangle_{ L^{ 2}}$. Therefore, recalling the definition of $ \varpi_{ j}$ in \eqref{eq:varpij}, we have for all $j\in \mathbb{ Z}$, by Cauchy-Schwarz inequality
\begin{align*}
{\rm D}\Theta (U_s) \left[\varpi_j\right]^2 &= \left( \varepsilon \sum_{ i=-N_{ \varepsilon}}^{ N_{ \varepsilon}} W \left(x_{ i}- x_{ j}\right) \left\langle \nabla \Theta(U_{ s})\, ,\, \mathbf{ 1}_{ I_{ i}}\right\rangle \right)^{ 2},\\
&\leq \varepsilon^{ 2} \left(\sum_{ i=-N_{ \varepsilon}}^{ N_{ \varepsilon}} W\left(x_{ i}- x_{ j}\right)^{ 2}\right) \left( \sum_{ i=-N_{ \varepsilon}}^{ N_{ \varepsilon}}\left\langle \nabla \Theta(U_{ s})\, ,\, \mathbf{ 1}_{ I_{ i}}\right\rangle^{ 2}\right).
\end{align*}
Note that, using Jensen's inequality $ \langle \nabla \Theta (U_s), \mathbf{1}_{I_i}\rangle^2 \leq \varepsilon \int_{ I_{ i}} \nabla \Theta \left(U_{ s}\right)^{ 2}$. Therefore,
\begin{align*}
{\rm D}\Theta (U_s) \left[\varpi_j\right]^2 &\leq \varepsilon^{ 3} \left\Vert \nabla \Theta \left(U_{ s}\right) \right\Vert_{ L^{ 2}}^{ 2}\left(\sum_{ i=-N_{ \varepsilon}}^{ N_{ \varepsilon}} W\left(x_{ i}- x_{ j}\right)^{ 2}\right)\leq C \varepsilon^{ 3}\left(\sum_{ i=-N_{ \varepsilon}}^{ N_{ \varepsilon}} W\left(x_{ i}- x_{ j}\right)^{ 2}\right),
\end{align*}
using the uniform bound on ${\rm D} \Theta$ on $ \mathcal{ N}$ on the last term (recall Proposition~\ref{prop:regularity isochron}). Hence, we can bound the first term in \eqref{eq:cro_Xi} by
\begin{align}
2\sum_{j\in\mathbb{Z}}\int_0^{t\wedge \tau} {\rm D}\Theta (U_s) \left[\varpi_j\right]^2 {\rm d} s&\leq C (t\wedge \tau) \varepsilon^{ 3} \sum_{ j \in \mathbb{ Z}} \sum_{ i=-N_{ \varepsilon}}^{ N_{ \varepsilon}}W\left(x_{ i}- x_{ j}\right)^{ 2} \leq C \left(t\wedge \tau\right) \varepsilon^{ 1- \beta},
\end{align}
where we used the estimate \eqref{eq:W_doublesum} for the last bound.
Turn now to the second term of \eqref{eq:cro_Xi}: using now the uniform bound on $ \left\Vert {\rm D}^{ 2} \Theta(U) \right\Vert$ along $ \mathcal{ N}$ (Proposition~\ref{prop:regularity isochron}), we have, using again that $ \lambda_{ j}\leq 1$,
\begin{align*}
 \sum_{j\in\mathbb{Z}}&\int_0^{t\wedge \tau}\int_0^\infty \left(\Theta(U_s+\varpi_{j}\mathbf{ 1}_{ z\leq \lambda_{ j, s}} )-\Theta(U_s)-{\rm D} \Theta(U_s)[\varpi_{j}\mathbf{ 1}_{ z\leq \lambda_{ j, s}} ]\right)^2{\rm d} z{\rm d } s\\
&\leq  C (t\wedge \tau)\sum_{j\in\mathbb{Z}} \Vert \varpi_j\Vert^2_{L^2}\leq C (t\wedge \tau) \varepsilon^{ 1- \beta}
\end{align*}
where we used Lemma~\ref{lem:bound_sum_varpi} on the last line.
This, together with the fact that $t\leq \varepsilon^{ -1} t_{ f}$ gives, for a constant $C>0$ that only depends on $W$, $f$ and $t_{ f}$
\begin{equation}
\label{eq:bound_cro_Xi}
\left\langle \Xi\right\rangle_{ t} \leq C \varepsilon^{ - \beta}.
\end{equation}
Moreover we have almost surely
\[
\sup_{0\leq t\leq \varepsilon^{-1}t_f} \left|\Delta \Xi_t\right|\leq \sup_{0\leq t\leq \varepsilon^{-1}t_f}\sup_{j\in \mathbb{Z}} \left|\Theta(U_t+\varpi_{j}\mathbf{ 1}_{ z\leq \lambda_{ j, s}} )-\Theta(U_t)\right| \leq \sup_{j\in \mathbb{Z}} \Vert \varpi_j\Vert_{L^2} \leq C \varepsilon,
\]
where we used \eqref{eq:bound_supvarpi} of Lemma~\ref{lem:bound_sum_varpi} in the last estimate. Noting that (recall Proposition~\ref{prop:concentration_mart}) $ \Psi(x) \geq x$ as long as $x\geq c$ for some $c>0$ so that $2 \exp \left(- \Psi \left( \frac{ \eta}{ a}\right)\right)\leq 2 \exp \left(- \frac{ \eta}{ a}\right)$ as long as $ a\leq \frac{ \eta}{ c}$, choosing $a$ in \eqref{eq:concentration_M} of order $\varepsilon^{-\frac{\beta}{2}}$ and $\eta$ of order $\varepsilon^{-\frac{\beta}{2}- \delta}$ we obtain the result.
\end{proof}

\subsection{Control on the noise term \texorpdfstring{$\zeta^{ (\varepsilon)}$}{zetaeps}}
We now turn to the control of $ \zeta^{ (\varepsilon)}$ in \eqref{eq:def zetakt}. The result is the following
\begin{proposition}\label{prop:bound zeta}
The noise term $ \zeta^{ (\varepsilon)}$ in \eqref{eq:def zetakt} satisfies the following: there exist some $ \varepsilon_{ 0}>0$ and some $C,c>0$ depending only on $W$ and $f$ such that the event given as
\begin{equation}
\label{eq:event_mcB}
\mathcal{B}_\varepsilon = \left\{\sup_{k\in \{1,2,\ldots,n_\varepsilon\}}\sup_{ 0\leq t \leq T} \left\Vert \zeta_{ k,t}^{ (\varepsilon)}\right\Vert_{m_{\psi_{k-1}}, \mathcal{ D}_{ 0}}\leq C \change{T^{ 5/8}\varepsilon^{\frac38-\beta}}\right\},
\end{equation}
satisfies for $\varepsilon \in(0, \varepsilon_{ 0})$,
\begin{align}
\mathbb{ P} \left(\mathcal{B}_\varepsilon\right) \geq 1-  n_\varepsilon^{ 2} e^{-c \varepsilon^{-\frac{\beta}{2}}} \label{eq:concentration_zeta}.
\end{align}
\end{proposition}
The proof of Proposition~\ref{prop:bound zeta} is more involved, since $\zeta_{k,t}^{ (\varepsilon)}$ in \eqref{eq:def zetakt} is not a martingale. We first consider the family of processes $(\rho_{k, q, t})_{k\geq 1, q\in \mathbb{Z}}$ defined as
\begin{equation}\label{eq:def_rho}
\rho_{k, q, t}:= \left\langle \zeta_{ k,t}^{ (\varepsilon)}\, ,\, \varpi_{ q}\right\rangle_{ m_{\psi_{k-1}}},\ t\geq0,\ k\geq1,\ q\in \mathbb{ Z}
\end{equation}
and prove the following preliminary Lemma.

\begin{lemma}
\label{lem:rho_k}
There exist constants $ \change{c_{ \rho}}, c_2>0$ such that
\begin{align}
\label{eq:bound_rho_k}
\mathbb{ P} \left(\sup_{ k=1, \ldots, n_{ \varepsilon}}\sup_{ q\in \mathbb{Z}} \sup_{ 0\leq t \leq T} \frac{\left\vert \rho_{k,q,t}\right\vert }{\Vert \varpi_{ q}\Vert_{m_{\psi_{k-1}}}}> \change{c_{ \rho}} T  \varepsilon^{-\beta} \right) &\leq n_{ \varepsilon}e^{-c_2 \varepsilon^{-(1+\beta)}  },
\end{align}
where we recall the definition of $n_{ \varepsilon}$ in \eqref{eq:n_varepsilon}.
\end{lemma}

\begin{proof}[Proof of Lemma~\ref{lem:rho_k}]
Using the definition \eqref{eq:def zetakt} of $ \zeta_{k,t}$, the fact that $\lambda_{j,s}\in [0,1]$ and the a priori bound \eqref{eq:bound_eLphi} on the semigroup $e^{ tL_{ \psi_{ k-1}}}$, we get the rough bound
\begin{multline*}
\left| \rho_{ k,q,t}\right|  \leq \sum_{ j\in \mathbb{Z}} \int_{T_{k-1}}^{T_{k-1}+t} \int_0^{1} \left|\left\langle e^{ (T_{k-1}+t-s)L_{\psi_{k-1}}} \varpi_{  j}, \varpi_{q}\right\rangle_{ m_{\psi_{k-1}}}\right| \left({\rm d} s {\rm d} z + \pi_{j}({\rm d} s, {\rm d} z )\right)\\
 \leq  \left(\sum_{ j\in \mathbb{Z}} \Vert \varpi_{j}\Vert_{ m_{\psi_{k-1}}}  \right) \Vert \varpi_{q}\Vert_{ m_{\psi_{k-1}}} t+ \Vert \varpi_{q}\Vert_{ m_{\psi_{k-1}}}  \sum_{ j\in \mathbb{Z}} \Vert \varpi_{ j}\Vert_{ m_{\psi_{k-1}}} \int_{T_{k-1}}^{T_{k-1}+t}  \int_0^{1} \pi_{j}({\rm d} s, {\rm d} z ).
\end{multline*}
Applying Lemma~\ref{lem:bound_sum_varpi}, we deduce from \eqref{eq:bound sum vert pi} that
\begin{align}
\sup_{ q\in \mathbb{ Z}}\sup_{0\leq t\leq T} \frac{ \left| \rho_{k,q,t}\right|}{ \Vert \varpi_{q}\Vert_{ m_{\psi_{k-1}}}} & \leq C \left(\varepsilon^{-\beta} T +  \sum_{ j\in \mathbb{Z}} \Vert \varpi_{  j}\Vert_{ \change{L^{ 2}}} \tilde Z_{k,j}\right),\label{eq:bound sup rho}
\end{align}
where $(\tilde Z_{k,j})_{j\in \mathbb{Z}}$ are IID random variables with Poisson distribution of parameter $T$. Now, for any $\eta,\alpha>0$ we have
\begin{align*}
\mathbb{P}\left(\sum_{ j\in \mathbb{Z}} \Vert \varpi_{ j}\Vert_{ \change{L^{ 2}}} \tilde Z_{k,j} >\eta \right)
& \leq e^{-\alpha \eta }\mathbb{E}\left[\exp\left(\alpha\sum_{ j\in \mathbb{Z}} \Vert \varpi_{ j}\Vert_{ \change{L^{ 2}}} \tilde Z_{k,j}\right)\right]\\
&\leq e^{-\alpha \eta}\prod_{j\in \mathbb{Z}}\exp\left(T\left(e^{\alpha\Vert \varpi_{ j}\Vert_{ \change{L^{ 2}}} }-1\right)\right).
\end{align*}
Recall from Lemma~\ref{lem:bound_sum_varpi} that there is a constant $C>0$ such that $\sup_{ j\in \mathbb{ Z}} \left\Vert \varpi_{ j} \right\Vert_{ \change{L^{ 2}}} \leq C \varepsilon$. Hence, using that $e^x-1\leq 2x$ for $x\in [0,1]$, for $\alpha_\varepsilon=c/\varepsilon$ with $c$ small enough such that $cC\leq 1$, we obtain that $\alpha \Vert \varpi_{ j}\Vert_{ \change{L^{ 2}}}\leq 1$ for all $j$ and $\psi$, so that, using \eqref{eq:bound sum vert pi}
\begin{align}\label{eq:bound weighted sum Poisson}
\mathbb{P}\left(\sum_{ j\in \mathbb{Z}} \Vert \varpi_{ j}\Vert_{ \change{L^{ 2}}} \tilde Z_{k,j} >\eta \right) &\leq   e^{- \frac{c \eta }{\varepsilon}} \exp\left( \frac{2cT}{\varepsilon}  \sum_{j\in \mathbb{Z}}\Vert \varpi_{  j}\Vert_{ \change{L^{ 2}}}\right)\leq e^{-c\left(\eta \varepsilon^{-1} -2T\varepsilon^{-(1+\beta)}\right)},
\end{align} 
which for the choice of $ \eta= 2 (T+1) \varepsilon^{ - \beta}$, together with \eqref{eq:bound sup rho} and a union bound on $k=1, \ldots, n_{ \varepsilon}$ implies the result \eqref{eq:bound_rho_k}.
\end{proof}

We are now able to prove Proposition~\ref{prop:bound zeta}.

\begin{proof}[Proof of Proposition~\ref{prop:bound zeta}]
Applying Ito's formula to $ \phi(\zeta_{ k,t})$ for $ \phi(\zeta):= \left\Vert \zeta \right\Vert^{ 2}_{ m_{\psi_{k-1}},\mathcal{D}_0}$ (see \cite{Hausenblas2006,Zhu2017}), we obtain
\begin{align}
&\phi(\zeta_{ k,t}) = \int_{T_{k-1}}^{T_{k-1}+ t} {\rm D} \phi(\zeta_{ k,s})(L_{\psi_{k-1}}\zeta_{k,s}) {\rm d}s \nonumber\\
&+ \sum_{ j\in \mathbb{Z}} \int_{T_{k-1}}^{T_{k-1}+ t}\int_{ 0}^{+\infty} {\rm D} \phi(\zeta_{ k,s^{ -}})(  \varpi_{ j}) \mathbf{ 1}_{ z\leq \lambda_{ j, s}}\tilde{ \pi}_{ j}({\rm d}s, {\rm d}z)\label{eq:Ito_zeta}\\
&+ \sum_{ j\in \mathbb{Z}} \int_{T_{k-1}}^{T_{k-1}+ t} \int_{ 0}^{+\infty} \Big[ \phi(\zeta_{k,s^{ -}} +  \varpi_{j}\mathbf{ 1}_{ z\leq \lambda_{ j, s}}) - \phi(\zeta_{k,s^{ -}}) - {\rm D} \phi( \zeta_{ k,s^{ -}})(  \varpi_{ j}\mathbf{ 1}_{ z\leq \lambda_{ j, s}})\Big] \pi_{ j}({\rm d}s, {\rm d}z).\nonumber
\end{align}
Since ${\rm D}\phi(\zeta)(h)= 2 \left\langle \zeta\, ,\, h\right\rangle_{ m_{\psi_{k-1}},\mathcal{D}_0}$, we obtain in particular that the first term in \eqref{eq:Ito_zeta} is \change{nonpositive}: by \eqref{eq:Dirform_Lphi},
\[
\int_{T_{k-1}}^{T_{k-1}+ t} {\rm D} \phi(\zeta_{ k,t})(L_{\psi_{k-1}}\zeta_{k,s}) {\rm d}s
= 2\int_{T_{k-1}}^{T_{k-1}+ t} \left\langle \zeta_{ k,s}\, ,\, L_{\psi_{k-1}}\zeta_{ k,s}\right\rangle_{ m_{\psi_{k-1}},\mathcal{D}_0} {\rm d}s\leq 0.
\]

\paragraph{ \change{Step 1}} \change{We prove the following intermediate result: the event
\begin{equation}
\label{eq:event_mcB1}
\mathcal{B}_\varepsilon^{ 1} = \left\{\sup_{k\in \{1,2,\ldots,n_\varepsilon\}}\sup_{ 0\leq t \leq T} \left\Vert \zeta_{ k,t}^{ (\varepsilon)}\right\Vert_{m_{\psi_{k-1}}, \mathcal{ D}_{ 0}}\leq C \change{T^{ 3/4}\varepsilon^{\frac14-\beta}}\right\},
\end{equation}
satisfies for $\varepsilon \in(0, \varepsilon_{ 0})$,
\begin{align}
\mathbb{ P} \left(\mathcal{B}_\varepsilon^{ 1}\right) \geq 1-  n_\varepsilon e^{-c \varepsilon^{-\frac{\beta}{2}}} \label{eq:concentration_zeta_1}.
\end{align}}
Denoting $I_{ k,t}$ the second term of the decomposition \eqref{eq:Ito_zeta}, we get, using again that ${\rm D}\phi(\zeta)(h)= 2 \left\langle \zeta\, ,\, h\right\rangle_{ m_{\psi_{k-1}},\mathcal{D}_0}$,
\begin{align*}
I_{k,t}&:= 2\sum_{ j\in \mathbb{Z}} \int_{T_{k-1}}^{T_{k-1}+ t} \int_{ 0}^{+\infty} \left\langle \zeta_{k,s^{ -}}\, ,\,  \varpi_{ j}\right\rangle_{ m_{\psi_{k-1}},\mathcal{D}_0} \mathbf{ 1}_{ z\leq \lambda_{ j, s}} \tilde{ \pi}_{ j}({\rm d}s, {\rm d}z),
\end{align*}
and
\begin{align}
\left\langle I_{ k,\cdot}\right\rangle_{ T}&= \change{4}\sum_{ j\in \mathbb{Z}} \int_{ T_{k-1}}^{T_k} \int_{ 0}^{+\infty} \left\langle \zeta_{ k,s^{ -}}, \varpi_{j} \right\rangle_{ m_{\psi_{k-1}},\mathcal{D}_0}^{ 2} \mathbf{ 1}_{ z\leq \lambda_{ j, s}} {\rm d}s{\rm d}z \label{eq:bracket_I}\\
&\leq \change{4}\sum_{ j\in \mathbb{Z}}\int_{T_{k-1}}^{T_k} \left\langle \zeta_{ k,s^{ -}},\varpi_{j} \right\rangle_{ m_{\psi_{k-1}},\mathcal{D}_0}^{ 2} {\rm d}s \leq \change{4}T \sum_{ j\in \mathbb{Z}} \sup_{0\leq t\leq T} \vert \rho_{k,j,t}\vert^2, \nonumber
\end{align}
as $ \lambda_{ j, s}\leq 1$ a.s. Now let us place ourselves on the event (recall Lemma~\ref{lem:rho_k})
\begin{equation}
\label{eq:event_Ak}
A_{ \varepsilon, T} :=\left\{ \sup_{ k=1, \ldots, n_{ \varepsilon}}\sup_{ q\in \mathbb{Z}} \sup_{ 0\leq t \leq T} \frac{\left\vert \rho_{ k,q,t}\right\vert}{\Vert \varpi_{  q}\Vert_{ m_{\psi_{k-1}}} } \leq \change{c_{ \rho}} T  \varepsilon^{-\beta}  \right\}.
\end{equation}
On $A_{ \varepsilon, T}$, we deduce from Lemma~\ref{lem:bound_sum_varpi}, eq.  \eqref{eq:bound sum vert pi} with $k=2$ that
\[
\left\langle I_{ k,\cdot}\right\rangle_{ T} \leq \change{4} \change{c_{ \rho}}^2 T^{ 3}  \varepsilon^{-2\beta} \sum_{j\in \mathbb{Z}} \Vert \varpi_{ j}\Vert_{ m_{\psi_{k-1}}}^2 \leq  C_1 T^{ 3}\varepsilon^{1-3\beta}.
\]
Secondly, we have that a.s. for all $0\leq t\leq T$, $ \left\vert \Delta I_{ k,t} \right\vert \leq \sup_{ q} \sup_{ 0\leq t \leq T} \left\vert \rho_{k,q,t} \right\vert$, hence we deduce, still on the event $A_{ \varepsilon, T}$, recalling the estimate \eqref{eq:bound_supvarpi} in Lemma~\ref{lem:bound_sum_varpi},
\begin{align}
\label{eq:bound_supI}
\sup_{0\leq t\leq T}\left\vert \Delta I_{ k,t}\right\vert \leq  C_2 T \varepsilon^{1-\beta}.
\end{align}
Therefore, for $C_{ 3}=\max \left( \sqrt{ C_{ 1}}, C_{ 2}\right)>0$, choosing $a:= C_3 T^{ 3/2} \varepsilon^{\frac{1-3\beta}{2}}$, then on the event $A_{\varepsilon,T}$ we have,
\[
\sup_{0\leq t\leq T}\left\vert \Delta I_{ k,t}\right\vert \leq  a \quad \text{and} \quad \left\langle I_{k,\cdot}\right\rangle_{ T}\leq a^2.
\]
Choosing $\eta= C a \varepsilon^{-\beta/2}$ with $ C$ large enough and applying Proposition~\ref{prop:concentration_mart} we obtain, for some $C_4>0$ and $C_{ 5}>0$
\begin{align*}
\mathbb{P}\left(\sup_{0\leq t\leq T} \left|I_{k,t}\right|\geq C_{ 4}  T^{ 3/2} \varepsilon^{\frac{1}{2}-2\beta}\right) &\leq \mathbb{P}\left(\sup_{0\leq t\leq T} \left|I_{k,t}\right|\geq C_{ 4}  T^{ 3/2} \varepsilon^{\frac{1}{2}-2\beta}\Big\vert A_{ \varepsilon, T}\right) + \mathbb{ P} \left(A_{ \varepsilon, T}^{ c}\right) \leq  e^{- C_{ 5} \varepsilon^{-\frac{\beta}{2}}}.
\end{align*}
Finally denote by $J_{k,t}$ as the last term in the Ito decomposition \eqref{eq:Ito_zeta}:
\begin{multline*}
J_{k,t}:= \sum_{ j\in \mathbb{Z}} \int_{T_{k-1}}^{T_{k-1}+ t} \int_{ 0}^{+\infty} \Big[ \phi(\zeta_{k,s^{ -}} +  \varpi_{j}\mathbf{ 1}_{ z\leq \lambda_{ j, s}}) - \phi(\zeta_{k,s^{ -}}))\\
 - {\rm D} \phi( \zeta_{ k,s^{ -}})(  \varpi_{ j}\mathbf{ 1}_{ z\leq \lambda_{ j, s}})\Big] \pi_{ j}({\rm d}s, {\rm d}z).
\end{multline*}
Since again ${\rm D}\phi(\zeta)(h)= 2 \left\langle \zeta\, ,\, h\right\rangle_{ m}$, simply expanding the previous expression gives
\begin{equation*}
J_{k,t}=\sum_{ j\in \mathbb{Z}} \int_{T_{k-1}}^{T_{k-1}+ t}\int_{ 0}^{+\infty} \left\Vert \varpi_{ j} \right\Vert^{ 2}_{ m_{\psi_{k-1}}} \mathbf{ 1}_{ z\leq \lambda_{ j, s}}\pi_{ j}({\rm d}s, {\rm d}z),
\end{equation*}
and using again the bound $ \lambda_{ j, s}\leq 1$, we obtain that 
\begin{equation*}
\sup_{ 0\leq t \leq T} \left\vert J_{k,t} \right\vert= J_{ k,T} \leq C  \sum_{ j\in \mathbb{Z}} \left\Vert \varpi_{ j} \right\Vert^{ 2}_{L^{ 2}} \tilde{ Z}_{k, j}
\end{equation*}
for the i.i.d. Poisson processes with intensity $T$ defined as in the proof of Lemma~\ref{lem:rho_k}. With calculations similar to the ones leading to \eqref{eq:bound weighted sum Poisson}, choosing $\alpha = \frac{ c}{ \varepsilon^{ 2}}$ for sufficiently small $c>0$ such that $0\leq \alpha \Vert \varpi_{ j}\Vert_{L^{ 2}}^2\leq 1$ for all $j$ (recall \eqref{eq:bound_supvarpi}), we obtain, 
\begin{align*}
\mathbb{ P} \left( \sup_{ 0\leq t \leq T} \left\vert J_{ k,t} \right\vert> \eta\right)
&\leq e^{-c \eta/ \varepsilon^{ 2}}\prod_{j\in \mathbb{Z}} \exp\left(T\left(e^{ \frac{ c}{ \varepsilon^{ 2}} \left\Vert \varpi_{j} \right\Vert^{ 2}_{L^{ 2} }}-1\right)\right)\\
&\leq e^{-c\eta/ \varepsilon^{ 2} } \exp\left(2 \frac{ cT}{ \varepsilon^{ 2}}\sum_{j\in \mathbb{Z}}\left\Vert \varpi_{ \varepsilon, j} \right\Vert^{ 2}_{L^{ 2}}\right)\leq e^{- \frac{ c}{ \varepsilon^{ 2}} (\eta -2T \varepsilon^{1-\beta})},
\end{align*}
where we have used that $\sum_{j\in \mathbb{Z}}\left\Vert \varpi_{ j} \right\Vert^{ 2}_{ L^{ 2}} \leq C \varepsilon^{1-\beta}$ (Lemma~\ref{lem:bound_sum_varpi}). Hence, choosing $ \eta= (2T+1) \varepsilon^{ 1- \beta}$, we deduce that there exists $C_6>0$ such that
\begin{equation}
\label{eq:bound_J}
\mathbb{ P} \left( \sup_{ 0\leq t \leq T} \left\vert J_{ k,t} \right\vert> (2T+1) \varepsilon^{1-\beta}\right) \leq e^{- C_{ 6}  \varepsilon^{-(1+\beta)}}.
\end{equation}
Noting that there is no loss of generality in assuming that $T\geq1$ (recall the definition of $T$ in \eqref{hyp:T}), we have $CT^{ 3/2} \varepsilon^{ 1/2-2 \beta} \geq T \varepsilon^{ 1- \beta}$ as long as $ \varepsilon\in (0, \varepsilon_{ 0})$ for some sufficiently small $ \varepsilon_{ 0}>0$ depending only on $W$, $f$ and $t_{ f}$. Estimation \eqref{eq:concentration_zeta_1} is then a consequence of the triangle inequality and of a union bound over $k$.

\paragraph{ \change{Step 2: final estimate}} \change{Let us place ourselves on the event $ \mathcal{ B}_{ \varepsilon}^{ 1}$ given by \eqref{eq:event_mcB1} and go back to the estimate of $ \left\langle I_{ k, \cdot}\right\rangle_{ T}$ in \eqref{eq:bracket_I}: we have, for some $C_{ 7}>0$, on the event $ \mathcal{ B}_{ \varepsilon}^{ 1}$, as again $ \lambda_{ j, s}\leq 1$,
\begin{align*}
\left\langle I_{ k,\cdot}\right\rangle_{ T}&\leq 2\sum_{ j\in \mathbb{Z}}\int_{T_{k-1}}^{T_k} \left\langle \zeta_{ k,s^{ -}},\varpi_{j} \right\rangle_{ m_{\psi_{k-1}},\mathcal{D}_0}^{ 2} {\rm d}s \leq 2 T \sup_{ s\in [0, T]} \left\Vert \zeta_{ k,s} \right\Vert_{ m_{ \psi_{ k-1}}, \mathcal{ D}_{ 0}}^{ 2}\sum_{ j\in \mathbb{Z}} \left\Vert \varpi_{j} \right\Vert_{ m_{\psi_{k-1}},\mathcal{D}_0}^{ 2},\\
&\leq C_{ 7}T^{ 5/2}\varepsilon^{\frac32-3\beta},
\end{align*}
where we have used again Lemma~\ref{lem:bound_sum_varpi} for the last estimate. Keeping the same estimate \eqref{eq:bound_supI} for $\sup_{ 0\leq t \leq T} \left\vert \Delta I_{ k, t} \right\vert$, we obtain, for $C_{ 8}=\max \left( \sqrt{ C_{ 7}}, C_{ 2}\right)>0$, choosing $a:= C_8 T^{ 5/4} \varepsilon^{ \frac{ 3}{ 4} - \frac{ 3 \beta}{ 2}}$, then on the event $ \mathcal{ B}_{ \varepsilon}^{ 1}$ we have,
\[
\sup_{0\leq t\leq T}\left\vert \Delta I_{ k,t}\right\vert \leq  a \quad \text{and} \quad \left\langle I_{k,\cdot}\right\rangle_{ T}\leq a^2.
\]
Choosing $\eta= C a \varepsilon^{-\beta/2}$ with $ C$ large enough and applying Proposition~\ref{prop:concentration_mart} we obtain, for some $C_9>0$ and $C_{ 10}>0$
\begin{align*}
\mathbb{P}\left(\sup_{0\leq t\leq T} \left|I_{k,t}\right|\geq C_{ 9}  T^{ 5/4} \varepsilon^{\frac{3}{4}-2\beta}\right) &\leq \mathbb{P}\left(\sup_{0\leq t\leq T} \left|I_{k,t}\right|\geq C_{ 9}  T^{ 5/4} \varepsilon^{\frac{3}{4}-2\beta}\Big\vert \mathcal{ B}_{ \varepsilon}^{ 1}\right) + \mathbb{ P} \left( \mathcal{ B}_{ \varepsilon}^{1, c}\right) \leq  n_{ \varepsilon}e^{- C_{ 10} \varepsilon^{-\frac{\beta}{2}}}.
\end{align*}
Keeping the same estimate \eqref{eq:bound_J} concerning $\sup_{ 0\leq t \leq T} \left\vert J_{ k, t} \right\vert$, we obtain the desired estimate \eqref{eq:event_mcB}, by the triangle inequality and a union bound on $k$.}
\end{proof}
\change{ \begin{remark}
\label{rem:order_zeta}
It is apparent from the above proof that one could very well reiterate the previous argument to improve the previous estimate on $ \zeta^{ (\varepsilon)}$: $ \left\Vert \zeta_{ k, t}^{ (\varepsilon)} \right\Vert$ is actually of order $ \varepsilon^{ - \alpha}$ for $ \alpha < \frac{ 1}{ 2}$ (arbitrarily close to $ \frac{ 1}{ 2}$). However, choosing $ \alpha= \frac{ 3}{ 8}- \beta - \delta$ is sufficient for our purpose.
\end{remark}}

\section{Proximity to the manifold \texorpdfstring{$\mathcal{M}$}{M}}
\label{sec:proximity_M}
Recall the definitions of the events $ \mathcal{ A}_{ \varepsilon, \delta}$ and $ \mathcal{ B}_{ \varepsilon}$ in \eqref{eq:event_mcA} and \eqref{eq:event_mcB} respectively. Recall also \eqref{eq:beta_small}: the parameter $ \beta$ in \eqref{eq:N_ell_eps} satisfies $ \change{\beta< \frac{ 1}{ 24}}$. Throughout the rest of the proof, we fix some parameter $\delta>0$ satisfying
\begin{equation}
\label{eq:delta_small}
\delta\in \left(0,  \min \left( \change{\frac{ 1}{ 24}}-\beta, \frac{ \beta}{ 4}\right)\right). 
\end{equation}
The precise reasons for this choice of parameters will become visible in the proofs of Propositions~\ref{prop: bound Theta} and~\ref{prop:closeness M}. We refer in particular to Remark~\ref{rem:beta_delta_small} below for more insights on \eqref{eq:beta_small} and \eqref{eq:delta_small}. For this choice of parameters, let us define the event
\begin{equation}
\label{eq:event_Cepsilon}
\mathcal{C}_{ \varepsilon, \delta} := \left\{ \left\Vert U_{ 0}-\hat u_0\right\Vert_{L^2} \leq \change{\varepsilon^{\frac38- \beta- \frac{ \delta}{ 2}}}\right\} \cap \mathcal{A}_{\varepsilon,\delta}\cap \mathcal{B}_\varepsilon.
\end{equation}
Recall the definition of the exit-time $ \tau= \tau_{ \varepsilon}$ in \eqref{eq:tau}. The aim of this section is to prove the following Proposition.
\begin{proposition}\label{prop:closeness M}
Let $ \beta, \delta$ satisfy \eqref{eq:beta_small} and \eqref{eq:delta_small}. On the event $\mathcal{C}_{ \varepsilon, \delta}$ given by \eqref{eq:event_Cepsilon}, we have, for $\varepsilon>0$ small enough, $\tau\geq t_f\varepsilon^{-1}$ and
\begin{equation*}
\sup_{0\leq t\leq t_f\varepsilon^{-1}} {\rm dist}_{L^2}\left(U_t,\mathcal{M}\right) \leq  \change{\varepsilon^{\frac38-\beta-\delta}}.
\end{equation*}
\end{proposition}
To prove this result, we first show that if the noise terms are not too large (that is on the event $\mathcal{A}_{\varepsilon,\delta}$) and the process $U_t$ stays sufficiently close to the manifold $\mathcal{M}$, then the phase $\Theta(U_t)$ cannot be too large. This is the aim of the next Subsection, in which we show that the dynamics of the phase does not drift too much on the time-scale $\varepsilon^{-1}t$ (we show that $|\Theta(U_t)|\leq \varepsilon^{-\frac{3\beta}{4}}$). This result allows then us to bound the boundary terms that appear in the iterative scheme defined in Section~\ref{sec:iter scheme}, which we use to prove Proposition~\ref{prop:closeness M}.

\subsection{Control on the isochron}
We prove in this paragraph the following proposition: recall the definition of the event $ \mathcal{ A}_{ \varepsilon, \delta}$ in \eqref{eq:event_mcA}.

\begin{proposition}\label{prop: bound Theta}
Let $ \beta, \delta$ satisfy \eqref{eq:beta_small} and \eqref{eq:delta_small}. Let $t_1\in [0,t_f]$ and consider the following event:
\begin{equation}
\label{eq:event_Oeps}
\mathcal{ O}_{ \varepsilon, \delta, t_{ 1}}:= \mathcal{A}_{\varepsilon,\delta}\cap \left\{\sup_{0\leq t\leq t_1\varepsilon^{-1}} {\rm dist}_{L^2}\left(U_t,\mathcal{M}\right) \leq  \change{\varepsilon^{\frac38-\beta-\delta}}\right\}.
\end{equation}
Then, on $ \mathcal{ O}_{ \varepsilon, \delta, t_{ 1}}$,
\begin{equation}
\label{eq:Theta_3b4}
\sup_{0\leq t\leq t_1\varepsilon^{-1}} \vert \Theta(U_t) \vert\leq \varepsilon^{-\frac{3\beta}{4}}.
\end{equation}
\end{proposition}

Before turning to the proof of this result, let us first present some auxillary results that will be needed. We present first the following tail bound for the profile $\hat u_0$, to which we provide a elementary proof.
\begin{lemma}
\label{lem:int_tail_hatu}
There exists some constant $ C>0$ such that, for $\ell>0$,
\begin{equation}
\label{eq:int_tail_hatu}
\max \left(\int_{ -\infty}^{- \ell} \left(a_{ 1} - \hat{ u}_{ 0}(x)\right)^{ 2} {\rm d}x, \int_{ \ell}^{+\infty} \left(a_{ 2} - \hat{u}_{ 0}(x)\right)^{ 2} {\rm d}x\right) \leq C e^{ - 2\mu\ell},
\end{equation}
where $ \mu$ is the constant appearing in \eqref{eq:uprime_exp}.
\end{lemma}
\begin{proof}[Proof of Lemma~\ref{lem:int_tail_hatu}]
This an immediate consequence of \eqref{eq:uprime_exp}:
\begin{multline*}
\int_{ \ell}^{+\infty} \left(a_{ 2}- \hat{ u}_{ 0}(x)\right)^{ 2} {\rm d}x = \int_{ \ell}^{ +\infty} \left( \int_{ x}^{+\infty} \hat{ u}^{ \prime}_{ 0}(z) {\rm d}z\right)^{ 2} {\rm d}x \leq c_{ 0}^{ 2} \int_{ \ell}^{ +\infty} \left( \int_{ x}^{+\infty} e^{ - \mu z} {\rm d}z\right)^{ 2} {\rm d}x \\= \frac{ c_{ 0}^{ 2}}{ 2 c_{ 1}^{ 3}} e^{ -2 \mu \ell},
 \end{multline*}
the second estimate of \eqref{eq:int_tail_hatu} being treated similarly.
\end{proof}

To control the boundary terms in the proof of Proposition~\ref{prop: bound Theta} we will rely on the following Proposition which shows that, close to $\mathcal{M}$, the isochron map depends very little on the shape of the profile outside $\mathcal{D}_0$ (recall \eqref{eq:partition_R} for the definition of the partition). The proof of this Proposition will be given in Section~\ref{sec: boundary}.
\begin{proposition}
\label{prop:D2Theta}
Let $0 < \beta^{ \prime}< \beta$. There is a constant $C>0$ such that, for all $v, w\in L^{ 2}$, $ \varphi$ such that $ \left\vert \varphi \right\vert\leq c \varepsilon^{ - \beta}$ with $c\in(0,1)$,  we have
\begin{equation}\label{eq:DThetaD+-}
\left\vert {\rm D}\Theta( \hat{ u}_{ \varphi}) \left[ v \left(\mathbf{ 1}_{  \mathcal{ D}_{ -}}+\mathbf{ 1}_{  \mathcal{ D}_{ +}}\right)\right] \right\vert \leq C e^{ - \varepsilon^{ - \beta^{ \prime}}} \left\Vert v \right\Vert,
\end{equation}
and
\begin{equation}
\label{eq:D2ThetaD+-}
\left\vert {\rm D}^{ 2}\Theta( \hat{ u}_{ \varphi}) \left[ v, w \left(\mathbf{ 1}_{  \mathcal{ D}_{ -}}+\mathbf{ 1}_{  \mathcal{ D}_{ +}}\right)\right] \right\vert \leq C e^{ - \varepsilon^{ - \frac{ \beta^{ \prime}}{ 2}}} \left\Vert v \right\Vert \left\Vert w \right\Vert.
\end{equation}
\end{proposition}

The proof of Proposition~\ref{prop: bound Theta} relies on a precise estimation of errors made by approaching the continuous NFE model \eqref{eq:NFE conv} with the discrete model. More precisely, let us define the following notations.
\begin{definition}
\label{def:Deltas}
For all $s\geq 0$, $ \varepsilon>0$, $i=-N_{ \varepsilon}, \ldots, N_{ \varepsilon}$, $x\in \mathbb{ R}$, define
\begin{align}
\Delta_{ s, \varepsilon, i}^{ (0)}(x)&:= \varepsilon \sum_{ j=-N_{ \varepsilon}}^{ N_{ \varepsilon}} W(x_i-x_j)f(U_{j,s^-}) - \int_{ \mathcal{ D}_{ 0}} W(x-y) f \left(U_{s}(y)\right) {\rm d}y, \label{eq:Delta0_def}\\
\Delta_{ s, \varepsilon, i}^{ (+)}(x)&:= \varepsilon \sum_{ j>N_{ \varepsilon}}W(x_i-x_j)f(U_{j,s^-}) - \int_{ \mathcal{ D}_{ +}} W(x-y) f \left(U_{s}(y)\right) {\rm d}y,\label{eq:Delta+_def}\\
\Delta_{ s, \varepsilon, i}^{ (-)}(x)&:= \varepsilon \sum_{ j<-N_{ \varepsilon}} W(x_i-x_j)f(U_{j,s^-}) - \int_{ \mathcal{ D}_{ -}} W(x-y) f \left(U_{s}(y)\right) {\rm d}y\label{eq:Delta-_def}.
\end{align}
\end{definition}

The proof of the following approximation results, which will be used in the proof of Proposition~\ref{prop: bound Theta}, will be given in Appendix~\ref{sec:discrete continuous}.

\begin{proposition}
\label{prop:bound_approx}
There exist some $C>0$ and $ \varepsilon_{ 0}>0$ only depending on $W,f$ such that for all $s\geq 0$, $ \varepsilon\in \left(0, \varepsilon_{ 0}\right)$, $i=-N_{ \varepsilon}, \ldots, N_{ \varepsilon}$, we have
\begin{equation}
\left\Vert \sum_{ i=- N_{ \varepsilon}}^{ N_{ \varepsilon}} \Delta_{ s, \varepsilon, i}^{ (0)} \mathbf{ 1}_{ I_{ i}}\right\Vert_{ L^{ 2}} \leq C \varepsilon^{ 1- \frac{ \beta}{ 2}},\label{eq:Norm_Delta0}
\end{equation}
and
\begin{equation}
\max \left(\left\Vert \sum_{ i=- N_{ \varepsilon}}^{ N_{ \varepsilon}} \Delta_{ s, \varepsilon, i}^{ (+)} \mathbf{ 1}_{ I_{ i}}\right\Vert_{ L^{ 2}} , \left\Vert \sum_{ i=- N_{ \varepsilon}}^{ N_{ \varepsilon}} \Delta_{ s, \varepsilon, i}^{ (-)} \mathbf{ 1}_{ I_{ i}}\right\Vert_{ L^{ 2}}\right)\leq C \varepsilon.
\label{eq:Norm_Delta+}
\end{equation}
\end{proposition}

We can now turn to the proof of Proposition~\ref{prop: bound Theta}.

\begin{proof}[Proof of Proposition~\ref{prop: bound Theta}]
Note that by \eqref{eq:Theta_dist} we have, on $ \mathcal{ O}_{ \varepsilon, \delta, t_{ 1}}$,
\[\sup_{ s\leq t_{ 1} \varepsilon^{ -1}} \left\Vert U_{ s} - \hat{ u}_{ \Theta(U_{ s})}\right\Vert_{ L^{ 2}} \leq \mathcal{ C}_{ \Theta} \change{\varepsilon^{ \frac{ 3}{8}- \beta- \delta}}.
\]
Recalling the Ito decomposition \eqref{eq:Ito to Theta}, we have, for $t\in [0, t_1\varepsilon^{-1}]$,
\begin{align*}
\Theta(U_t)  &=  \Theta(U_0) -\int_0^t {\rm D}\Theta(U_s)U_s  {\bf 1}_{\mathcal{D}_0} {\rm d} s\\
&\quad + \sum_{j\in\mathbb{Z}}\int_0^t \int_0^\infty\left(\Theta(U_s+\varpi_{j}\mathbf{ 1}_{ z\leq \lambda_{ j, s}} )-\Theta(U_s)\right){\rm d} s{\rm d} z + \Xi_t,
\end{align*}
where $ \Xi_{ t}$ in defined in \eqref{eq:Nt}. We expand this decomposition into several contributions as follows:
\begin{align}
\Theta(U_t) &=\Theta(U_0) + \int_0^t {\rm D}\Theta(U_s)\left( -U_s+ W*f(U_s) \right) {\rm d} s \nonumber\\
& \quad + \int_0^t {\rm D}\Theta(U_s)\left(\left(U_s - W*f(U_s) \right) \left({\bf 1}_{\mathcal{D}_-} +{\bf 1}_{\mathcal{D}_+} \right)\right) {\rm d} s \nonumber\\
&\quad +  \sum_{j\in\mathbb{Z}}\int_0^t \int_0^\infty {\rm D}\Theta(U_s)\varpi_{j}\mathbf{ 1}_{ z\leq \lambda_{ j, s}}{\rm d} s{\rm d} z -\int_0^t {\rm D}\Theta(U_s)\left( W*f(U_s) {\bf 1}_{\mathcal{D}_0} \right) {\rm d} s \nonumber \\
&\quad + \sum_{j\in\mathbb{Z}}\int_0^t \int_0^\infty\left(\Theta(U_s+\varpi_{j}\mathbf{ 1}_{ z\leq \lambda_{ j, s}} )-\Theta(U_s)- {\rm D}\Theta(U_s)\varpi_{j}\mathbf{ 1}_{ z\leq \lambda_{ j, s}} \right){\rm d} z{\rm d} s + \Xi_t.\label{eq:decomp Theta}
\end{align}
Let us define
\begin{equation}
\label{eq:t_ast}
t_* =\inf \left\{t\in  [0, t_1\varepsilon^{-1}]:\, \vert \Theta(U_t) \vert>  \varepsilon^{- \frac{3\beta}{4}}\right\}.
\end{equation}
Our aim is to show that $t_*=t_1 \varepsilon^{ -1}$ by finding adequate upper bounds for the different terms of the right-hand side of \eqref{eq:decomp Theta}.

\medskip
\noindent
{\it Step $1$: first line of \eqref{eq:decomp Theta}.} Let us first tackle the first line of the right-hand side of \eqref{eq:decomp Theta}. This is the simplest part of the proof. Indeed, if for all $ v\in \mathcal{ N}$, $ \left(T_{ t}(v)\right)_{ t\geq0}$ is the flow of the dynamics of \eqref{eq:NFE conv} with $ T_{ 0}(v)=v$, we have, by construction of the isochron map $\Theta$, that $ \Theta \left(T_{ t}(v)\right)= \Theta \left(T_{ 0}(v)\right)= \Theta \left(v\right)$. Hence, by differentiation w.r.t. $t$ and taking $t=0$ gives ${\rm D}\Theta(v)\left( -v+ W*f(v)\right) = 0$ for all $v\in \mathcal{ N}$. Hence, the first integral term in \eqref{eq:decomp Theta} is simply $0$.

\medskip
\noindent
{\it Step $2$: second line of \eqref{eq:decomp Theta}.} Let us now work on the second line of the right-hand side of \eqref{eq:decomp Theta}. For simplicity we concentrate on what happens on $\mathcal{D}_+$, the estimates for $\mathcal{D}_-$ being similar. We have, for $t\leq t_{ \ast}$,
\begin{align}\label{eq:replace Us by proj with isochron}
\int_{ 0  }^{t}{\rm D}\Theta(U_s)&\left(\left(U_s - W*f(U_s) \right) {\bf 1}_{\mathcal{D}_+}\right) {\rm d}s= \int_{ 0}^{t}{\rm D}\Theta(U_s)\left(\left(a_2 - \hat u_{\Theta(U_s)} \right) {\bf 1}_{\mathcal{D}_+}\right) {\rm d}s\\
&\quad + \int_{ 0}^{t}{\rm D}\Theta(U_s)\left(W*\left(f\left(\hat u_{\Theta(U_s)}\right) - f(U_s) \right) {\bf 1}_{\mathcal{D}_+}\right) {\rm d}s.\nonumber
\end{align}
Firstly,
\begin{align*}
\int_{ 0}^{t}\left\vert {\rm D}\Theta(U_s)\left(\left(a_2 - \hat u_{\Theta(U_s)} \right) {\bf 1}_{\mathcal{D}_+}\right)\right\vert {\rm d}s &\leq C \int_{ 0}^{t}\left\Vert \left(a_2 - \hat u_{\Theta(U_s)} \right) {\bf 1}_{\mathcal{D}_+}\right\Vert_{L^2} {\rm d}s\\
&= C \int_{ 0}^{t}\left(\int_{\ell_\varepsilon-\Theta(U_s)}^\infty (a_2-\hat u_0(x))^2{\rm d} x\right)^\frac12 {\rm d}s\\
&\leq C \int_{ 0}^{t}e^{-\mu\left(\ell_\varepsilon-\Theta(U_s)\right)} {\rm d}s\leq \frac{ C}{ \varepsilon} e^{ - \mu \left(\varepsilon^{ - \beta}-\varepsilon^{ - \frac{3\beta}{4}}\right)},
\end{align*}
where we have used Lemma~\ref{lem:int_tail_hatu} in the last estimate, together with the definition of $t_{ \ast}$. Concentrate now on the second term in the righthand side of \eqref{eq:replace Us by proj with isochron}: for $t\leq t_{ \ast}$,
\begin{align}
\int_{ 0}^{t}{\rm D}\Theta(U_s)&\left(W*\left(f\left(\hat u_{\Theta(U_s)}\right) - f(U_s) \right) {\bf 1}_{\mathcal{D}_+}\right) {\rm d}s \nonumber\\
&= \int_{ 0}^{t}\left({\rm D}\Theta(U_s) - {\rm D} \Theta( \hat{ u}_{ \Theta \left(U_{ s}\right)})\right)\left(W*\left(f\left(\hat u_{\Theta(U_s)}\right) - f(U_s) \right) {\bf 1}_{\mathcal{D}_+}\right) {\rm d}s \nonumber\\
&\quad + \int_{ 0}^{t}{\rm D} \Theta( \hat{ u}_{ \Theta \left(U_{ s}\right)})\left(W*\left(f\left(\hat u_{\Theta(U_s)}\right) - f(U_s) \right) {\bf 1}_{\mathcal{D}_+}\right) {\rm d}s \label{aux:DThetaAB}\\
&= A_{ t}+B_{ t}\nonumber.
\end{align}
Let us denote $Y_s=U_s-\hat u_{\Theta(U_s)}$ and decompose the first term $A_{ t}$ in \eqref{aux:DThetaAB} into
\begin{align}
&A_{ t}= \int_{ 0}^{t}{\rm D}^{ 2} \Theta \left(\hat{ u}_{ \Theta(U_{ s})}\right) \left[ Y_s, W*\left(f\left(\hat u_{\Theta(U_s)}\right) - f(U_s) \right) {\bf 1}_{\mathcal{D}_+}\right] {\rm d}s\nonumber\\
&\quad \quad + \frac{ 1}{ 2} \int_{ 0}^{t}\int_{ 0}^{1} (1-r){\rm D}^{ 3} \Theta \left(\hat{ u}_{ \Theta \left(U_{ s}\right)} + r Y_s\right) \left[Y_s, Y_s, W*\left(f\left(\hat u_{\Theta(U_s)}\right) - f(U_s) \right) {\bf 1}_{\mathcal{D}_+}\right] {\rm d}r {\rm d}s\nonumber\\
&\quad =A_{ 1, t} + A_{ 2, t}.\label{aux:At}
\end{align}
To control of $A_{ 1, t}$ in \eqref{aux:At} we apply Proposition~\ref{prop:D2Theta}, for any $ \beta^{ \prime}< \beta$ and $t\leq t_{ \ast}$,
\begin{align*}
\left\vert A_{ 1, t} \right\vert &\leq C e^{ - \varepsilon^{ -\frac{ \beta^{ \prime}}{ 2}}}\int_{ 0}^{t} \left\Vert Y_s \right\Vert_{L^2} \left\Vert W*\left(f\left(\hat u_{\Theta(U_s)}\right) - f(U_s) \right) \right\Vert_{L^2}{\rm d}s,\\
& \leq C e^{ - \varepsilon^{ -\frac{ \beta^{ \prime}}{ 2}}}\int_{ 0}^{t} \left\Vert U_{ s} - \hat{ u}_{ \Theta \left(U_{ s}\right)} \right\Vert_{L^2}^{ 2}{\rm d}s,
\end{align*}
by Young's inequality for convolution, Lipschitz continuity of $f$ and since $W\in L^{ 1}$. Hence, for $t\leq t_{ \ast}$, on the event $ \mathcal{ O}_{ \varepsilon, \delta, t_{ 1}}$,
\begin{equation}
\label{eq:boundA1t}
\left\vert A_{ 1, t} \right\vert  \leq C \varepsilon^{-4\beta-4\delta} e^{ - \varepsilon^{ -\frac{ \beta^{ \prime}}{ 2}}}.
\end{equation}
For the control of $A_{ 2, t}$ in \eqref{aux:At}, remark that
\begin{align*}
\left\Vert W*\left(f\left(\hat u_{\Theta(U_s)}\right) - f(U_s) \right) {\bf 1}_{\mathcal{D}_+} \right\Vert_{ L^{ 2}} \leq C \left\Vert W \right\Vert_{ L^{ 1}} \left\Vert U_{ s} - \hat{ u}_{ \Theta \left(U_{ s}\right)}\right\Vert_{ L^{ 2}},
\end{align*}
and
\begin{align*}
\left\Vert W*\left(f\left(\hat u_{\Theta(U_s)}\right) - f(U_s) \right) {\bf 1}_{\mathcal{D}_+} \right\Vert_{ L^{ \infty}} \leq C \left\Vert W \right\Vert_{ L^{ 2}} \left\Vert U_{ s} - \hat{ u}_{ \Theta \left(U_{ s}\right)}\right\Vert_{ L^{ 2}},
\end{align*}
so that, applying \eqref{eq:bound D3Theta}, on the event $ \mathcal{ O}_{ \varepsilon, \delta, t_{ 1}}$ (recall \eqref{eq:event_Oeps}), for $t\leq t_{ \ast}$,
\begin{align*}
\left\vert A_{ 2, t} \right\vert \leq C \int_{ 0}^{t}\left\Vert U_{ s} - \hat{ u}_{ \Theta \left(U_{ s}\right)}\right\Vert_{ L^{ 2}}^{ 3} {\rm d}s \leq \change{ C \varepsilon^{ -1}\varepsilon^{ \frac{ 9}{ 8} - 3 \beta - 3 \delta}= C \varepsilon^{ \frac{ 1}{ 8} - 3 \beta - 3 \delta}}.
\end{align*}
Concentrate now on the term $B_{ t}$ in \eqref{aux:DThetaAB}: relying again on Proposition~\ref{prop:D2Theta} we obtain, for any $\beta'<\beta$ and $t\leq t_*$,
\begin{align*}
\left\vert B_{ t} \right\vert &\leq Ce^{-\varepsilon^{\beta'}}\int_0^t \left\Vert W*\left(f\left(\hat u_{\Theta(U_s)}\right) - f(U_s) \right) \right\Vert_{L^2}{\rm d}s\leq Ce^{-\varepsilon^{\beta'}}\int_0^t \left\Vert \hat u_{\Theta(U_s)} - U_s\right\Vert_{ L^{ 2}} {\rm d}s,
\end{align*}
by Lipschitz-continuity of $f$ and Young's inequality. So, for $t\leq t_{ \ast}$, on the event $\mathcal{O}_{\varepsilon,\delta,t_1}$,
\begin{align*}
\left\vert B_{ t} \right\vert \leq \change{C \varepsilon^{ -\frac{ 5}{ 8} - \beta - \delta} e^{ - \varepsilon^{ - \beta'}}}.
\end{align*}

\medskip

\noindent
{\it Step $3$: third line of \eqref{eq:decomp Theta}.} Concentrate now on the third term in the decomposition \eqref{eq:decomp Theta}: for reasons that will become obvious in the following (see the term $C_{ 4,t}$ in \eqref{aux:Ct2}), we need to introduce the discrete $ \varepsilon$-approximation of the phase
\begin{equation}
\label{eq:barTheta}
\bar \Theta_s^{ (\varepsilon)} := \bar \Theta_s =\varepsilon \lfloor \Theta(U_s)/\varepsilon\rfloor,\ s\geq0.
\end{equation}
With this notation at hand, let
\begin{align}
C_{ t}&:= \sum_{j\in\mathbb{Z}}\int_0^t \int_0^\infty {\rm D}\Theta(U_s)\varpi_{j}\mathbf{ 1}_{ z\leq \lambda_{ j, s}}{\rm d} s{\rm d} z -\int_0^t {\rm D}\Theta(U_s)\left( W*f(U_s) {\bf 1}_{\mathcal{D}_0} \right) {\rm d} s, \nonumber\\
&= \int_{ 0}^{t} {\rm D} \Theta(U_{ s}) \left( \sum_{ i=-N_{ \varepsilon}}^{ N_{ \varepsilon}} \left(\sum_{ j\in \mathbb{ Z}} \varepsilon W(x_{ i}-x_{ j})  f \left(U_{ j, s-}\right)- W\ast f \left(U_{ s}\right)\right)  \mathbf{ 1}_{ I_{ i}}\right){\rm d}s, \nonumber\\
&= \int_{ 0}^{t} \left({\rm D} \Theta(U_{ s})- {\rm D} \Theta(\hat{ u}_{ \bar \Theta_s})\right) \left( \sum_{ i=-N_{ \varepsilon}}^{ N_{ \varepsilon}} \left(\Delta_{ s, \varepsilon, i}^{ (-)} + \Delta_{ s, \varepsilon, i}^{ (0)} + \Delta_{ s, \varepsilon, i}^{ (+)}\right) \mathbf{ 1}_{ I_{ i}}\right){\rm d}s \nonumber\\
&\quad + \int_{ 0}^{t} {\rm D} \Theta(\hat{ u}_{\bar \Theta_s})  \sum_{ i=-N_{ \varepsilon}}^{ N_{ \varepsilon}}\sum_{ j=-N_{ \varepsilon}}^{ N_{ \varepsilon}}\varepsilon W(x_i-x_j)\left(f(U_{j,s^-})-f(U_{j,s})\right)\mathbf{ 1}_{ I_{ i}}{\rm d}s \nonumber\\
&\quad + \int_{ 0}^{t} {\rm D} \Theta(\hat{ u}_{ \bar \Theta_s})  \left( \sum_{ i=-N_{ \varepsilon}}^{ N_{ \varepsilon}} \left(\sum_{ j\in \mathbb{ Z}} \varepsilon W(x_{ i}-x_{ j})  f \left(U_{ j, s}\right)- W\ast f \left(U_{ s}\right)\right)  \mathbf{ 1}_{ I_{ i}}\right){\rm d}s \nonumber\\
&= C_{1,t}+C_{2,t}+C_{3,t}. \label{aux:Ct}
\end{align}
Remark that $C_{2,t}=0$ in \eqref{aux:Ct} since there is almost surely a countable number of jumps before time $t_*$. Concerning $C_{1,t}$ we get, by Proposition~\ref{prop:regularity isochron} and Proposition~\ref{prop:bound_approx}, for $t\leq t_{ \ast}$,
\begin{align*}
\vert C_{1,t}\vert&\leq C \int_0^t \left\Vert \hat u_{\bar \Theta_s} - U_s\right\Vert_{ L^{ 2}}\left\Vert  \sum_{ i=-N_{ \varepsilon}}^{ N_{ \varepsilon}} \left(\Delta_{ s, \varepsilon, i}^{ (-)} + \Delta_{ s, \varepsilon, i}^{ (0)} + \Delta_{ s, \varepsilon, i}^{ (+)}\right) \mathbf{ 1}_{ I_{ i}}\right\Vert_{L^2} {\rm d} s\leq \change{C \varepsilon^{ \frac{ 3}{ 8}- \frac{ 3 \beta}{ 2}-\delta}},
\end{align*}
where we used the fact that
\[
\left\Vert \hat u_{\bar \Theta_s} - U_s\right\Vert_{ L^{ 2}}\leq \left\Vert \hat u_{\bar \Theta_s} - \hat u_{\Theta(U_s)}\right\Vert_{ L^{ 2}}+\left\Vert \hat u_{\Theta(U_s)} - U_s\right\Vert_{ L^{ 2}}\leq \change{C\varepsilon^{ \frac{ 3}{ 8}-\beta-\delta}}.
\]
Let us now focus on $C_{3,t}$ in \eqref{aux:Ct}. Define
\begin{equation}
\bar Y_s=U_s-\hat u_{\bar \Theta_s}
\end{equation}
as well as
\begin{equation}
h_s(y) = \int_0^1 f'\left(\hat u_{\bar \Theta_s}(y)+rY_s(y)\right){\rm d} r.
\end{equation}
With this at hand, let us make another further decomposition of this term
\begin{align}
C_{3,t} &= \int_0^t \sum_{ i\in \mathbb{Z}}\sum_{ j\in \mathbb{Z}}\int_{I_i}\int_{I_j}\left(W(x_i-x_j)-W(x-y)\right)f(\hat u_{\bar \Theta_s})(y)\hat u'_{\bar \Theta_s}(x) m_{\bar \Theta_s}(x){\rm d}y{\rm d}x{\rm d}s \nonumber\\
&\quad +\int_0^t \sum_{ i\in \mathbb{Z}}\sum_{ j\in \mathbb{Z}}\int_{I_i}\int_{I_j}\left(W(x_i-x_j)-W(x-y)\right)\bar Y_s(y) h_s(y)\hat u'_{\bar \Theta_s}(x) m_{\bar \Theta_s}(x){\rm d}y{\rm d}x{\rm d}s \nonumber\\
&\quad +\int_0^t \sum_{ i > N_\varepsilon}\int_{I_i}\left(\sum_{ j\in \mathbb{Z}}\varepsilon W(x_i-x_j)f(U_{j,s})-W*f(U_s)(x)\right)\hat u'_{\bar \Theta_s}(x) m_{\bar \Theta_s}(x){\rm d}x{\rm d}s \nonumber\\
&\quad +\int_0^t \sum_{ i < -N_\varepsilon}\int_{I_i}\left(\sum_{ j\in \mathbb{Z}}\varepsilon W(x_i-x_j)f(U_{j,s})-W*f(U_s)(x)\right)\hat u'_{\bar \Theta_s}(x) m_{\bar \Theta_s}(x){\rm d}x{\rm d}s \nonumber\\
&=C_{4,t}+C_{5,t}+C_{6,t}+C_{7,t}. \label{aux:Ct2}
\end{align}
Let us first focus on $C_{4,t}$ in \eqref{aux:Ct2}. With the two changes of variables $x \to x- \bar \Theta_s$ and $y \to y- \bar \Theta_s$ we get
\begin{align*}
C_{4,t} &= \int_0^t \sum_{ i\in \mathbb{Z}}\sum_{ j\in \mathbb{Z}}\int_{I_i}\int_{I_j}\left(W(x_i-x_j)-W(x-y)\right)f(\hat u_{0})(y)\hat u'_{0}(x) m_{0}(x){\rm d}y{\rm d}x{\rm d}s.
\end{align*}
This is where we see the necessity of using the discrete phase $ \bar \Theta_{ s} \in \varepsilon \mathbb{ Z}$ defined in \eqref{eq:barTheta}: these two changes of variables do not induce any changes in the discrete sums on $i$ and $j$, as, doing so, we remain on the discrete lattice $ \varepsilon \mathbb{ Z}$. Another symmetry argument gives
\begin{align*}
C_{4,t} 
& =\int_0^t \sum_{ i\in \mathbb{Z}}\sum_{ j\in \mathbb{Z}}\int_{I_i}\int_{I_j}\left(W(x_i-x_j)-W(x-y)\right)f(\hat u_{0})(-y)\hat u'_{0}(-x) m_{0}(-x){\rm d}y{\rm d}x{\rm d}s.
\end{align*}
Since $x\mapsto \hat u'_{0}(x) m_{0}(x)$ is even while $y\mapsto f(\hat u_0)(y)-a$ is odd, we deduce that
\begin{equation*}
C_{4,t} = 2a \int_0^t\sum_{ i\in \mathbb{Z}}\sum_{ j\in \mathbb{Z}}\int_{I_i}\int_{I_j}\left(W(x_i-x_j)-W(x-y)\right)\hat u'_0(x) m_0(x){\rm d}y{\rm d}x{\rm d} s.
\end{equation*}
Remark now that since
\[
\left|\int_{I_i^2}\left(W(x_i-x_i)-W(x-y)\right)\hat u'_0(x) m_0(x){\rm d}y{\rm d}x\right|\leq C \varepsilon^2\int_{I_i}\hat u'_0(x) m_0(x){\rm d}x,
\]
we have, on $t\leq \varepsilon^{ -1} t_{ f}$,
\begin{align*}
\bigg|\int_0^t\sum_{ i\in \mathbb{Z}}\int_{I_i^2}\left(W(x_i-x_i)-W(x-y)\right)\hat u'_0(x) &m_0(x){\rm d}y{\rm d}x{\rm d} s\bigg|\\
&\leq C\varepsilon^2\int_0^t \int \hat u'_0(x) m_0(x){\rm d}x{\rm d} s\leq C\varepsilon,
\end{align*}
so we may neglect the diagonal terms in $C_{ 4,t}$. Moreover, for $i\neq j$, a straightforward computation leads to
\begin{align*}
\int_{I_j}\left(W(x_i-x_j)-W(x-y)\right){\rm d} y &= \varepsilon \left(W(x_i-x_j)- W(x-x_j)\right)\\
&\quad +\left(\varepsilon- \change{2\sigma}\sinh \left( \change{\frac{ \varepsilon}{ 2 \sigma}}\right)\right)W(x-x_j).
\end{align*}
Then, remarking that, for $x\in I_i$,
\begin{align*}
\sum_{ j\neq i} &\left(W(x_i-x_j)-W(x-x_j)\right)\\
&= \left(1-e^{-\frac{x-x_i}{\sigma}}\right)\sum_{j<i}W(x_i-x_j)+\left(1-e^{-\frac{x_i-x}{\sigma}}\right)\sum_{j>i}W(x_i-x_j),
\end{align*}
we obtain the decomposition
\begin{align*}
C_{4,t}&= a \int_0^t\sum_{ i\in \mathbb{Z}}\left(\varepsilon\sum_{j<i}W(x_i-x_j)\right)\int_{I_i}\left(1-e^{-\frac{x-x_i}{\sigma}}\right)\hat u'_0(x) m_0(x){\rm d}x{\rm d} s\\
&\quad + a \int_0^t\sum_{ i\in \mathbb{Z}}\left(\varepsilon\sum_{j>i}W(x_i-x_j)\right)\int_{I_i}\left(1-e^{-\frac{x_i-x}{\sigma}}\right)\hat u'_0(x) m_0(x){\rm d}x{\rm d} s\\
&\quad + a \frac{\varepsilon- \change{2}\sigma\sinh \left(\varepsilon/ \change{2}\sigma\right)}{\varepsilon}\int_0^t\sum_{ i\in \mathbb{Z}}\int_{I_i}\left(\varepsilon\sum_{j\neq i}W(x-x_j) \right)\hat u'_0(x) m_0(x){\rm d}x{\rm d} s + O(\varepsilon)\\
&= C^1_{4,t}+C^2_{4,t}+C^3_{4,t}+ O(\varepsilon).
\end{align*}
Concerning $C^1_{4,t}$, decompose first
\begin{align*}
\int_{I_i} &\left(1-e^{-\frac{x-x_i}{\sigma}}\right)\hat u'_0(x) m_0(x) {\rm d}x\\
&=( \varepsilon- 2\sigma \sinh(\varepsilon/2\sigma))\hat u'_0(x_i) m_0(x_i)+ \int_{I_i}\left(1-e^{-\frac{x-x_i}{\sigma}}\right)\int_{x_i}^x(\hat u'_0 m_0)'(r){\rm d}r{\rm d}x\\
&=( \varepsilon- 2\sigma \sinh(\varepsilon/2\sigma))\hat u'_0(x_i) m_0(x_i)\\
&\quad +\int_{x_i}^{x_i+\varepsilon
/2}\left(2\sigma \left(e^{-\frac{\varepsilon}{2\sigma}}-e^{-\frac{r-x_i}{\sigma}}\right)+\left(x_i-r+\frac{\varepsilon}{2}\right)\right)(\hat u'_0 m_0)'(r){\rm d}r\\
&\quad +\int_{x_i-\varepsilon
/2}^{x_i}\left(2\sigma \left(e^{-\frac{r-x_i}{\sigma}}- e^{\frac{\varepsilon}{2\sigma}}\right)+\left(r-x_i+\frac{\varepsilon}{2}\right)\right)(\hat u'_0 m_0)'(r){\rm d}r.
\end{align*}
Hence, we get, for $t\leq t_{ \ast}$, remarking that $\varepsilon\sum_{j<i}W(x_i-x_j)\leq C$,
\begin{align*}
\left|C^1_{4,t}\right|\leq C \varepsilon^2\int_0^t\varepsilon\sum_{i\in\mathbb{Z}}\hat u'_0(x_i) m_0(x_i){\rm d}s+C\varepsilon^2 \int_0^t \int_{\mathbb{R}}\left|(\hat u'_0 m_0)'\right|(r){\rm d}r\leq C\varepsilon.
\end{align*}
Note here that we have used that $ \hat{ u}_{ 0}^{ \prime \prime}\in L^{ 2}$, see Lemma~\ref{lem:u0H2}. The term $C^2_{4,t}$ can be tackled in a similar way, and
\begin{align*}
\left|C^3_{4,t}\right|\leq C\varepsilon^2\int_0^t \int\hat u'_0(x) m_0(x){\rm d}x{\rm d} s\leq C\varepsilon,
\end{align*}
so that $\left|C_{4,t}\right|\leq C\varepsilon$. Concerning $C_{5,t}$ in \eqref{aux:Ct2}, we have, with two Cauchy-Schwarz inequalities, one on the integral w.r.t. $y$ and the second on the discrete sum w.r.t. $j$, the bound
\begin{multline*}
\left|C_{5,t}\right| \leq \\
\int_0^t\Vert \bar Y_s\Vert_{L^2}\sum_{ i\in \mathbb{Z}}\int_{I_i}\left(\sum_{ j\in \mathbb{Z}}\int_{I_j}\left(\left(W(x_i-x_j)-W(x-y)\right)h_s(y)\right)^2 {\rm d}y\right)^\frac12 \hat u'_{\bar \Theta_s}(x) m_{\bar \Theta_s}(x) {\rm d}x{\rm d}s,
\end{multline*}
and since a simple computation leads to
\[
\int_{I_j}\left(\left(W(x_i-x_j)-W(x-y)\right)h_s(y)\right)^2\leq C\varepsilon^3 W(x_i-x_j)^2,
\]
we obtain, for $t\leq t_{ \ast}$,
\begin{align*}
\left|C_{5,t}\right| &\leq C\int_0^t\Vert \bar Y_s\Vert_{L^2}\sum_{ i\in \mathbb{Z}}\int_{I_i}\left(\sum_{ j\in \mathbb{Z}}\varepsilon^3 W(x_i-x_j)^2\right)^\frac12\hat u'_{\bar \Theta_s}(x) m_{\bar \Theta_s}(x) {\rm d}x{\rm d}s\\
&\leq C\varepsilon \int_0^t\Vert \bar Y_s\Vert_{L^2}\int \hat u'_{\bar \Theta_s} m_{\bar \Theta_s} {\rm d}x{\rm d}s\leq C \change{\varepsilon^{ \frac{ 3}{ 8}-\beta-\delta}},
\end{align*}
where we used \eqref{eq:event_Oeps} and Lemma~\ref{lem:W_IN_IN}. Finally, since $f$ is bounded, we easily obtain, relying on \eqref{eq:uprime_exp}, on $t\leq t_{ \ast}$,
\begin{align*}
\left|C_{6,t}\right|&\leq C\int_0^t\int_{\mathcal{D_+}}\hat u'_{\bar \Theta_s}(x) m_{\bar \Theta_s}(x){\rm d}x{\rm d}s\leq C\int_0^t \int_{\varepsilon^{-\beta}-|\bar \Theta_s|}^\infty \hat u_0(x) m_0(x){\rm d}x{\rm d}s\\
&\leq C\varepsilon^{-1} e^{- \mu\left(\varepsilon^{-\beta}-\varepsilon^{-\frac{3\beta}{4}}\right)},
\end{align*}
and a similar bound can be obtained for $C_{7,t}$.

\medskip

\noindent
{\it Step $4$: fourth line of \eqref{eq:decomp Theta}.} Denote by $D_{ t}$ the fourth term in the decomposition \eqref{eq:decomp Theta}. By a simple Taylor expansion, we directly obtain that 
\begin{align*}
D_{ t}&= \change{ \frac{ 1}{ 2}} \sum_{ j\in \mathbb{ Z}}\int_0^t  {\rm D}^{2}\Theta \left(U_{ s} \right) \left[ \varpi_{ j}, \varpi_{ j}\right]\lambda_{ j,s}{\rm d}s\\
&\quad +  \sum_{ j\in \mathbb{ Z}}\int_{ 0}^{t}\int_0^\infty \int_{ 0}^{1} \frac{(1-r)^2}{2} {\rm D}^{ 3}\Theta \left(U_{ s} + r \left( \varpi_{ j} \mathbf{ 1}_{ z\leq \lambda_{ j, s}}\right)\right) \left[ \varpi_{ j}, \varpi_{ j}, \varpi_{ j}\right]\mathbf{ 1}_{ z\leq \lambda_{ j, s}}{\rm d}r{\rm d}z{\rm d}s
\end{align*}
Here, the strategy is again to use the definition of the microscopic intensity $ \lambda_{ j,s}= f \left(U_{ j,s^{ -}}\right)$ and to make sure that one can replace the random profile $U_{ s}$ by its stationary companion $ \hat{ u}_{ \Theta(U_{ s})}$ whenever it is necessary in the previous expression. This induces the following decomposition of $D_{ t}$:
\begin{align}
D_{ t}& =  \change{ \frac{ 1}{ 2}}\sum_{ j\in \mathbb{ Z}}\int_0^t {\rm D}^{2}\Theta \left(\hat u_{\Theta(U_s)} \right) \left[ \varpi_{ j}, \varpi_{ j}\right]f(\hat u_{\Theta(U_s)} )(x_j){\rm d}s \nonumber\\
&\quad +\change{ \frac{ 1}{ 2}} \sum_{ j\in \mathbb{ Z}}\int_0^t {\rm D}^{2}\Theta \left(\hat u_{\Theta(U_s)} \right) \left[ \varpi_{ j}, \varpi_{ j}\right]\left(f(U_{j,s^-})-f(\hat u_{\Theta(U_s)} )(x_j)\right){\rm d}s \nonumber\\
&\quad + \change{ \frac{ 1}{ 2}}\sum_{ j\in \mathbb{ Z}}\int_{ 0}^{t} \int_{ 0}^{1}  {\rm D}^{ 3}\Theta \left( \hat u_{\Theta(U_s)}+ r \left( U_{ s}-\hat u_{\Theta(U_s)}\right)\right) \left[ U_{ s}-\hat u_{\Theta(U_s)}, \varpi_{ j}, \varpi_{ j}\right]{\rm d}r\lambda_{ j,s}{\rm d}s \nonumber\\
&\quad +  \sum_{ j\in \mathbb{ Z}}\int_{ 0}^{t}\int_0^\infty \int_{ 0}^{1} \frac{(1-r)^2}{2} {\rm D}^{ 3}\Theta \left(U_{ s} + r \left( \varpi_{ j} \mathbf{ 1}_{ z\leq \lambda_{ j, s}}\right)\right) \left[ \varpi_{ j}, \varpi_{ j}, \varpi_{ j}\right]\mathbf{ 1}_{ z\leq \lambda_{ j, s}}{\rm d}r{\rm d}z{\rm d}s \nonumber\\
&= D_{1,t}+D_{2,t}+D_{3,t}+D_{4,t}. \label{aux:Dt}
\end{align}
Let us first focus on $D_{1,t}$. Remarking Lemma~\ref{lem:bound_sum_varpi} induces the bound
\[
\sum_{ j\in \mathbb{ Z}}|{\rm D}^{2}\Theta \left(\hat u_{\Theta(U_s)} \right) \left[ \varpi_{ j}, \varpi_{ j}\right]|\leq C \sum_{ j\in \mathbb{ Z}}\Vert \varpi_{ j}\Vert_{L^2}^2\leq C \varepsilon^{1-\beta},
\] 
and that $\sup_j \sup_{y_j\in I_j}|f(\hat u_{\Theta(U_s)} )(x_j)-f(\hat u_{\Theta(U_s)} )(y_j)|\leq C\varepsilon $ we get, for $t\leq t^*$,
\begin{align*}
D_{1,t} =  \frac{1}{ \change{2}\varepsilon}\sum_{ j\in \mathbb{ Z}}\int_{I_j}\int_0^t {\rm D}^{2}\Theta \left(\hat u_{\Theta(U_s)} \right) \left[ \varpi_{ j}, \varpi_{ j}\right]f(\hat u_{\Theta(U_s)} )(y_j){\rm d} y_j{\rm d}s+ O(\varepsilon^{1-\beta}).
\end{align*}
Now, for any $y_j\in I_j$, defining $W_y(x)=W(x-y)$, we have the decomposition
\begin{align*}
\frac{1}{\varepsilon}{\rm D}^{2}\Theta &\left(\hat u_{\Theta(U_s)} \right) \left[ \varpi_{ j}, \varpi_{ j}\right]-\varepsilon {\rm D}^{2}\Theta \left(\hat u_{\Theta(U_s)} \right) \left[ W_{y_j}\mathbf{1}_{\mathcal{D}_0},W_{y_j}\mathbf{1}_{\mathcal{D}_0}\right]\\
 &= \frac{1}{\varepsilon}{\rm D}^{2}\Theta \left(\hat u_{\Theta(U_s)} \right) \left[ \varpi_{ j}-\varepsilon W_{y_j}\mathbf{1}_{\mathcal{D}_0}, \varpi_{ j}-\varepsilon W_{y_j}\mathbf{1}_{\mathcal{D}_0}\right] \\
 &\quad + 2 {\rm D}^{2}\Theta \left(\hat u_{\Theta(U_s)} \right) \left[ \varpi_{ j}-\varepsilon W_{y_j}\mathbf{1}_{\mathcal{D}_0},  W_{y_j}\mathbf{1}_{\mathcal{D}_0}\right].
\end{align*}
We get the following bound
\begin{align}
\sum_{ j\in \mathbb{ Z}} \Vert \varpi_j-\varepsilon W_{y_j}\mathbf{1}_{\mathcal{D}_0}\Vert^2_{L^2}& =\sum_{j\in \mathbb{Z}} \sum_{i=-N_\varepsilon}^{N_\varepsilon}\varepsilon^2\int_{I_i}(W(x_i-x_j)-W(x-y_j))^2 {\rm d} x\nonumber\\
&\leq \sum_{j\in \mathbb{Z}}\sum_{i=-N_\varepsilon}^{N_\varepsilon}C\varepsilon^5 W(x_i-x_j)^2\nonumber\\
&\leq C\varepsilon^{ \change{3}-\beta},\label{eq:approx_varpi_Wy}
\end{align}
while a straightforward computation leads to
\begin{align*}
\sum_{ j\in \mathbb{ Z}} \Vert W_{y_j}\mathbf{1}_{\mathcal{D}_0}\Vert^2_{L^2}\leq C \varepsilon^{-1-\beta},
\end{align*}
and thus, relying in particular on Cauchy-Schwarz inequality, we can replace the sum by an integral in $D_{1,t}$: for all $t\leq t^*$,
\begin{align*}
D_{1,t} =  \change{ \frac{ \varepsilon}{ 2}}\int_0^t \int_{\mathbb{R}} {\rm D}^{2}\Theta \left(\hat u_{\Theta(U_s)} \right) \left[ W_y \mathbf{1}_{\mathcal{D}_{0}}, W_y\mathbf{1}_{\mathcal{D}_{0}}\right]f(\hat u_{\Theta(U_s)} )(y){\rm d}y{\rm d}s+O(\varepsilon^{1-\beta}).
\end{align*}
Let us define now $\mathcal{D}_{0,\varphi} = \mathcal{D}_{0}\cup (\mathcal{D}_{0}+\varphi)$. Relying on Proposition~\ref{prop:D2Theta} and the fact that $t\leq t_*$, we have
\begin{align}
D_{1,t} =  \change{ \frac{ \varepsilon}{ 2}}\int_0^t \int_{\mathbb{R}} {\rm D}^{2}\Theta \left(\hat u_{\Theta(U_s)} \right) \left[ W_y \mathbf{1}_{\mathcal{D}_{0,\Theta(U_s)}}, W_y\mathbf{1}_{\mathcal{D}_{0,\Theta(U_s)}}\right]f(\hat u_{\Theta(U_s)} )(y)\change{ {\rm d}y}{\rm d}s +O(\varepsilon^{1-\beta}).\label{aux:Dt1_2}
\end{align}
Indeed, if for example $\Theta(U_s)>0$,
\begin{align}
\varepsilon \int_{\mathbb{R}} {\rm D}^{2}\Theta \left(\hat u_{\Theta(U_s)} \right) &\left[ W_y \left(\mathbf{1}_{\mathcal{D}_{0,\Theta(U_s)}}-\mathbf{1}_{\mathcal{D}_0}\right), W_y\mathbf{1}_{\mathcal{D}_0}\right]f(\hat u_{\Theta(U_s)} )(y){\rm d} y \nonumber\\
&=\varepsilon \int_{\mathbb{R}} {\rm D}^{2}\Theta \left(\hat u_{\Theta(U_s)} \right) \left[ W_y \mathbf{1}_{(\ell,\ell+\Theta(U_s)]}, W_y\mathbf{1}_{\mathcal{D}_0}\right]f(\hat u_{\Theta(U_s)} )(y){\rm d} y \nonumber\\
&\leq C \varepsilon e^{-\varepsilon^{-\frac{\beta'}{2}}}\int_{\mathbb{R}} \left\Vert W_y \mathbf{1}_{(\ell,\ell+\Theta(U_s)]}\right\Vert_{L^2} \left\Vert W_y\mathbf{1}_{\mathcal{D}_{0}}\right\Vert_{L^2}{\rm d} y \nonumber\\
&\leq C \varepsilon e^{-\varepsilon^{-\frac{\beta'}{2}}} \left(\int_{\mathbb{R}} \left\Vert W_y \mathbf{1}_{(\ell,\ell+\Theta(U_s)]}\right\Vert_{L^2} ^2{\rm d} y\right)^\frac12 \left(\int_{\mathbb{R}} \left\Vert W_y\mathbf{1}_{\mathcal{D}_{0}}\right\Vert_{L^2}^2{\rm d} y\right)^\frac12, \label{aux:Dt1_1}
\end{align}
and
\begin{align*}
\int_{\mathbb{R}} \left\Vert W_y \mathbf{1}_{(\ell,\ell+\Theta(U_s)]}\right\Vert_{L^2} ^2{\rm d} y = \int_{\ell}^{\ell+\Theta(U_s)}\int_{\mathbb{R}} W^2(x-y){\rm d} y{\rm d}x\leq C\varepsilon^{-\frac{3\beta}{4}},
\end{align*}
while
\begin{align*}
\int_{\mathbb{R}} \left\Vert W_y \mathbf{1}_{\mathcal{D}_{0}}\right\Vert_{L^2} ^2{\rm d} y = \int_{-\ell}^{\ell}\int_{\mathbb{R}} W^2(x-y){\rm d} y{\rm d}x\leq C\varepsilon^{-\beta},
\end{align*}
so that gathering the two previous estimates in front of the $e^{-\varepsilon^{-\frac{\beta'}{2}}}$ term in \eqref{aux:Dt1_1} gives \eqref{aux:Dt1_2}.
Relying now on the translation invariance of the model, we obtain
\begin{align*}
D_{1,t} =  \change{ \frac{ \varepsilon}{ 2}}\int_0^t \int_{\mathbb{R}} {\rm D}^{2}\Theta \left(\hat u_{0} \right) \left[ W_y \mathbf{1}_{\mathcal{D}_{0,\Theta(U_s)}-\Theta(U_s)}, W_y\mathbf{1}_{\mathcal{D}_{0,\Theta(U_s)}-\Theta(U_s)}\right]f(\hat u_{0} )(y){\rm d}s+O(\varepsilon^{1-\beta}),
\end{align*} 
and with similar estimates as before, we may replace $\mathcal{D}_{0,\Theta(U_s)}-\Theta(U_s) =\mathcal{D}_{0}\cup (\mathcal{D}_{0}-\Theta(U_s))$ by $\mathcal{D}_0$:
\begin{align*}
D_{1,t} &=  \varepsilon t \change{ \frac{ 1}{ 2}}\int_{\mathbb{R}} {\rm D}^{2}\Theta \left(\hat u_{0} \right) \left[ W_y \mathbf{1}_{\mathcal{D}_{0}}, W_y\mathbf{1}_{\mathcal{D}_{0}}\right]f(\hat u_{0} )(y) {\rm d}y+O(\varepsilon^{1-\beta}):= \change{ \varepsilon t v_{ \varepsilon}}+O(\varepsilon^{1-\beta}).
\end{align*} 
\change{The next point is to prove that
\begin{multline}
\label{eq:conv_veps}
v_{ \varepsilon}= \frac{ 1}{ 2}\int_{\mathbb{R}} {\rm D}^{2}\Theta \left(\hat u_{0} \right) \left[ W_y \mathbf{1}_{\mathcal{D}_{0}}, W_y\mathbf{1}_{\mathcal{D}_{0}}\right]f(\hat u_{0} )(y) {\rm d}y \\\xrightarrow[ \varepsilon\to 0]{} \bar v:= \frac{ 1}{ 2}\int_{\mathbb{R}} {\rm D}^{2}\Theta \left(\hat u_{0} \right) \left[ W_y , W_y\right]f(\hat u_{0} )(y) {\rm d}y <+\infty.
\end{multline}
By dominated convergence theorem, it suffices to control the integrand $h_{ \varepsilon}(y)$ within the lefthand side of \eqref{eq:conv_veps} uniformly in $ \varepsilon$: we have, by the same calculation as for \eqref{aux:D2ThetaAB}, for some constant $C>0$ independent of $ \varepsilon$, for fixed $y$
\begin{align}
h_{ \varepsilon}(y)&:= {\rm D}^{2}\Theta \left(\hat u_{0} \right) \left[ W_y \mathbf{1}_{\mathcal{D}_{0}}, W_y\mathbf{1}_{\mathcal{D}_{0}}\right]f(\hat u_{0} )(y),\\
& \leq C\left\vert  \int_{ 0}^{+\infty}  \int_{ \mathbb{ R}} f^{ \prime\prime}( \hat{ u}_{0}) \left(\hat{ u}^{ \prime}_{0}\right)^{ 2} \left(e^{ s \mathcal{ L}_{ 0}} P_{ 0}^{ \perp} \left(W_{ y} \mathbf{ 1}_{ \mathcal{ D}_{ 0}}\right)\right)  {\rm d}x {\rm d}s\right\vert \left\vert \frac{ \left\langle W_{ y} \mathbf{ 1}_{ \mathcal{ D}_{ 0}}\, ,\, \hat{ u}_{ 0}^{ \prime}\right\rangle}{ \left\langle \hat{ u}_{ 0}^{ \prime}\, ,\, \hat{ u}_{ 0}^{ \prime}\right\rangle} \right\vert \nonumber\\
  &+C\left\vert  \int_{ 0}^{+\infty}  \int_{ \mathbb{ R}} f^{ \prime\prime}( \hat{ u}_{ 0}) \hat{ u}^{ \prime}_{ 0} \left(e^{ s \mathcal{ L}_{ 0}} \left(W_{ y} \mathbf{ 1}_{ \mathcal{ D}_{ 0}}\right)\right) \left(e^{ s \mathcal{ L}_{ 0}} \left(P_{ 0}^{ \perp} \left(W_{ y} \mathbf{ 1}_{  \mathcal{ D}_{ 0}}\right)\right)\right) {\rm d}x {\rm d}s\right\vert := h_{ 1, \varepsilon}(y) + h_{ 2, \varepsilon}(y) \label{eq:h12}.
\end{align}
Let us treat the two terms separately: we have, by Cauchy-Schwarz inequality and using \eqref{eq:contract_ecL}
\begin{align*}
h_{ 1, \varepsilon}(y) \leq \frac{ C}{ \left\Vert \hat{ u}_{ 0}^{ \prime} \right\Vert^{ 2}_{ L^{ 2}}} \left\Vert f^{ \prime \prime} \right\Vert_{ \infty} \left\Vert \hat{ u}_{ 0}^{ \prime} \right\Vert_{ L^{ 4}}^{ 2} \int_{ 0}^{+\infty} e^{ - \kappa s} {\rm d}s \left\Vert W_{ y} \right\Vert_{ L^{ 2}} \left\vert  \left\langle W_{ y} \mathbf{ 1}_{ \mathcal{ D}_{ 0}}\, ,\, \hat{ u}_{ 0}^{ \prime}\right\rangle \right\vert.
\end{align*} 
Note that $  \left\Vert W_{ y} \right\Vert_{ L^{ 2}}=  \left\Vert W \right\Vert_{ L^{ 2}}$, so that, for another constant $C>0$ independent of $y$,
\begin{align}
\label{eq:h1epsy}
h_{ 1, \varepsilon}(y) \leq C\left\vert  \left\langle W_{ y} \mathbf{ 1}_{ \mathcal{ D}_{ 0}}\, ,\, \hat{ u}_{ 0}^{ \prime}\right\rangle \right\vert = C \int_{ \mathcal{ D}_{ 0}} W(x-y) \hat{ u}_{ 0}^{ \prime}(x) {\rm d}x \leq C \int_{ \mathbb{ R}} W(x-y) \hat{ u}_{ 0}^{ \prime}(x) {\rm d}x.
\end{align}
Recall now the exponential control one has on $ \hat{ u}_{ 0}^{ \prime}$: choose $ \nu\in \left(0, \min \left(\frac{ 4}{ \sigma}, \mu\right)\right)$ where $ \mu>0$ is given by \eqref{eq:uprime_exp} and write
\begin{align*}
h_{ 1, \varepsilon}(y)&\leq C \int_{ \mathbb{ R}} W(x-y)  e^{ - \frac{ \nu}{ 4} \left\vert x \right\vert}\hat{ u}_{ 0}^{ \prime}(x) e^{ \frac{ \nu}{ 4} \left\vert x \right\vert} {\rm d}x,\\
&\leq C\left( \int_{ \mathbb{ R}} W(x-y)^{ 2} e^{ - \frac{ \nu}{ 2} \left\vert x \right\vert} {\rm d}x\right)^{ \frac{ 1}{ 2}} \left\Vert \hat{ u}_{ 0}^{ \prime} e^{ \frac{ \nu}{ 4} \left\vert \cdot \right\vert} \right\Vert_{ L^{ 2}},\\
&\leq C \left( \int_{ \mathbb{ R}} W(u)^{ 2} e^{ - \frac{ \nu}{ 2} \left\vert u+y \right\vert} {\rm d}u\right)^{ \frac{ 1}{ 2}} \leq C e^{ - \frac{ \nu}{ 4} \left\vert y \right\vert} \left( \int_{ \mathbb{ R}} W(u)^{ 2} e^{ \frac{ \nu}{ 2} \left\vert u \right\vert} {\rm d}u\right)^{ \frac{ 1}{ 2}}. 
\end{align*}
Since $ \frac{ \nu}{ 2} < \frac{ 2}{ \sigma}$ (recall the definition of $W$ in \eqref{eq:Wexp}), the last integral is finite and we obtain $h_{ 1, \varepsilon}(y)\leq C e^{ - \frac{ \nu}{ 4} \left\vert y \right\vert}$. 

\medskip

Turn now to the second term $h_{ 2, \varepsilon}$ in \eqref{eq:h12}: divide the integral in \eqref{eq:h12} into
\begin{align*}
h_{ 2, \varepsilon}(y)
&\leq C\int_{ a(y)}^{+\infty}  \int_{ \mathbb{ R}} \left\vert f^{ \prime\prime}( \hat{ u}_{ 0}(x)) \right\vert  \hat{ u}^{ \prime}_{ 0}(x) \left\vert e^{ s \mathcal{ L}_{ 0}} \left(W_{ y} \mathbf{ 1}_{ \mathcal{ D}_{ 0}}\right)(x) \right\vert \left\vert e^{ s \mathcal{ L}_{ 0}} \left(P_{ 0}^{ \perp} \left(W_{ y} \mathbf{ 1}_{  \mathcal{ D}_{ 0}}\right)\right)(x) \right\vert {\rm d}x {\rm d}s\\
&+ C\int_{ 0}^{a(y)}  \int_{ \mathbb{ R}} \left\vert f^{ \prime\prime}( \hat{ u}_{ 0}(x)) \right\vert  \hat{ u}^{ \prime}_{ 0}(x) \left\vert e^{ s \mathcal{ L}_{ 0}} \left(W_{ y} \mathbf{ 1}_{ \mathcal{ D}_{ 0}}\right)(x) \right\vert \left\vert e^{ s \mathcal{ L}_{ 0}} \left(P_{ 0}^{ \perp} \left(W_{ y} \mathbf{ 1}_{  \mathcal{ D}_{ 0}}\right)\right)(x) \right\vert {\rm d}x {\rm d}s, \\
&:= h_{ 3, \varepsilon}(y) + h_{ 4, \varepsilon}(y),
\end{align*}
for some $a(y)$ to be defined later. Consider first $h_{ 3, \varepsilon}$: by Cauchy-Schwarz inequality, for some constant $C>0$ that may change from one line to the other,
\begin{align*}
h_{ 3, \varepsilon}(y)&\leq C \left\Vert f^{ \prime\prime} \right\Vert_{ \infty} \left\Vert \hat{ u}_{ 0}^{ \prime} \right\Vert_{ \infty}\int_{ a(y)}^{+\infty} \left\Vert e^{ s \mathcal{ L}_{ 0}} \left(W_{ y} \mathbf{ 1}_{ \mathcal{ D}_{ 0}}\right)\right\Vert_{ L^{ 2}} \left\Vert e^{ s \mathcal{ L}_{ 0}} \left(P_{ 0}^{ \perp} \left(W_{ y} \mathbf{ 1}_{  \mathcal{ D}_{ 0}}\right)\right)\right\Vert_{ L^{ 2}} {\rm d}s,\\
&\leq C \int_{ a(y)}^{+\infty} e^{ - \kappa s}\left\Vert W_{ y} \mathbf{ 1}_{  \mathcal{ D}_{ 0}}\right\Vert_{ L^{ 2}}^{ 2} {\rm d}s\leq C \left\Vert W_{ y}\right\Vert_{ L^{ 2}}^{ 2} \int_{ a(y)}^{+\infty} e^{ - \kappa s} {\rm d}s = \frac{ C}{ \kappa} \left\Vert W\right\Vert_{ L^{ 2}}^{ 2} e^{ - \kappa a(y)}. 
\end{align*}
Choose now $a(y):= \frac{ 1}{ \kappa} \ln (1+ y^{ 2})$, so that one obtains $h_{ 3, \varepsilon}(y) \leq \frac{ C}{ 1+ y^{ 2}}$ that is integrable in $y$.

\medskip

For this choice of $a(y)$, turn now to the term $h_{ 4, \varepsilon}$ above. 
Writing $P_{ 0}^{ \perp}= I_{ d} - P_{ 0}$ and applying the triangular inequality, one can write
\begin{align*}
h_{ 4, \varepsilon}(y)&\leq C\int_{ 0}^{a(y)}  \int_{ \mathbb{ R}} \left\vert f^{ \prime\prime}( \hat{ u}_{ 0}(x)) \right\vert  \hat{ u}^{ \prime}_{ 0}(x) \left\vert e^{ s \mathcal{ L}_{ 0}} \left(W_{ y} \mathbf{ 1}_{ \mathcal{ D}_{ 0}}\right)(x) \right\vert \left\vert e^{ s \mathcal{ L}_{ 0}} \left(P_{ 0} \left(W_{ y} \mathbf{ 1}_{  \mathcal{ D}_{ 0}}\right)\right)(x) \right\vert {\rm d}x {\rm d}s\\
&+C\int_{ 0}^{a(y)}  \int_{ \mathbb{ R}} \left\vert f^{ \prime\prime}( \hat{ u}_{ 0}(x)) \right\vert  \hat{ u}^{ \prime}_{ 0}(x) \left\vert e^{ s \mathcal{ L}_{ 0}} \left(W_{ y} \mathbf{ 1}_{ \mathcal{ D}_{ 0}}\right)(x) \right\vert^{ 2} {\rm d}x {\rm d}s,\\
&:= h_{ 5, \varepsilon}(y) + h_{ 6, \varepsilon}(y).
\end{align*}
Concerning the first term above, we have
\begin{align*}
 h_{ 5, \varepsilon}(y) &\leq  \frac{ C}{ \left\langle \hat u'_0,\hat u'_0\right\rangle_{m_0}}\int_{0}^{a(y)}  \int_{ \mathbb{ R}} \left\vert f^{ \prime\prime}( \hat{ u}_{ 0}(x)) \right\vert  \left(\hat{ u}^{ \prime}_{ 0}(x)\right)^{ 2} \left\vert e^{ s \mathcal{ L}_{ 0}} \left(W_{ y} \mathbf{ 1}_{ \mathcal{ D}_{ 0}}\right)(x) \right\vert {\rm d}x {\rm d}s \left\vert \left\langle W_{ y} \mathbf{ 1}_{ \mathcal{ D}_{ 0}},\hat u'_0\right\rangle_{m_0} \right\vert,\\
 &\leq  \frac{ C}{ \left\Vert \hat{ u}_{ 0}^{ \prime} \right\Vert_{ L^{ 2}}^{ 2}} \left\Vert f^{ \prime \prime} \right\Vert_{ \infty} \left\Vert \hat{ u}_{ 0}^{ \prime} \right\Vert_{ L^{ 4}} \left\Vert W \right\Vert_{ L^{ 2}}a(y)\left\vert \left\langle W_{ y} \mathbf{ 1}_{ \mathcal{ D}_{ 0}},\hat u'_0\right\rangle_{m_0} \right\vert.
\end{align*}
Proceeding exactly as for \eqref{eq:h1epsy}, we see that for the same parameter $ \nu$ as above, for a constant $C$ independent of $ \varepsilon$,
\begin{align*}
h_{ 5, \varepsilon}(y) &\leq C a(y) e^{ -  \frac{ \nu}{ 4} \left\vert y \right\vert}.
\end{align*}
Hence, it remains to control 
\begin{equation}
\label{eq:h3tilde}
h_{6, \varepsilon}(y):= \int_{ 0}^{a(y)}  \int_{ \mathbb{ R}} \left\vert f^{ \prime\prime}( \hat{ u}_{ 0}(x)) \right\vert  \hat{ u}^{ \prime}_{ 0}(x) \left\vert e^{ s \mathcal{ L}_{ 0}} \left(W_{ y} \mathbf{ 1}_{ \mathcal{ D}_{ 0}}\right)(x) \right\vert^{ 2} {\rm d}x {\rm d}s.
\end{equation}
Note that since $W\geq 0$, we have $e^{ s \mathcal{ L}_{ 0}} \left(W_{ y} \mathbf{ 1}_{ \mathcal{ D}_{ 0}}\right)=e^{ -s}e^{ s \mathcal{ K}_{ 0}} \left(W_{ y} \mathbf{ 1}_{ \mathcal{ D}_{ 0}}\right)\leq e^{ -s}e^{ s \mathcal{ K}_{ 0}} \left(W_{ y}\right)= e^{ s \mathcal{ L}_{ 0}} \left(W_{ y}\right)$, where for any $ \varphi$, $\mathcal{ K}_{ \varphi}(v)(x)= \int_{ \mathbb{ R}}W(z-x) f^{ \prime} \left( \hat{ u}_{ \varphi}(z)\right)v(z) {\rm d}z$.
Hence, it suffices to control
\begin{align}
\label{eq:h7}
h_{ 7}(y)&:= C\int_{ 0}^{a(y)}  \int_{ \mathbb{ R}} \left\vert f^{ \prime\prime}( \hat{ u}_{ 0}(x)) \right\vert  \hat{ u}^{ \prime}_{ 0}(x) \left\vert e^{ s \mathcal{ L}_{ 0}} \left(W_{ y}\right)(x) \right\vert^{ 2}{\rm d}x {\rm d}s.
\end{align}
The point of the following is to prove that, for any $c_{ 0}\in (0,1)$, there are constants $c_{ 1}, c_{ 2}>0$ such that
\begin{equation}
\label{eq:eL0yx}
e^{ s \mathcal{ L}_{ 0}}(W_{ y}) (x) \leq c_{ 1} e^{ c_{ 2}s} e^{ - \frac{ c_{ 0} \left\vert x-y \right\vert}{ \sigma}}.
\end{equation}
Indeed
\begin{equation*}
e^{ s \mathcal{ L}_{ 0}}(W_{ y})(x)= e^{ -s}\sum_{ k=0}^{ +\infty} \frac{ s^{ k}}{ k!} \mathcal{ K}_{ 0}^{ k}(W_{ y})(x).
\end{equation*}
Now, by a change of variables
\begin{align*}
\mathcal{ K}_{ 0}(W_{ y})(x)&= \int_{ \mathbb{ R}} W(z-x) f^{ \prime}(\hat{ u}_{ 0}(z)) W(z-y) {\rm d}z\leq \left\Vert f^{ \prime} \right\Vert_{ \infty} \frac{ 1}{ (2 \sigma)^{ 2}} \int_{ \mathbb{ R}} e^{ - \frac{ \left\vert z+y-x \right\vert}{ \sigma}} e^{ - \frac{ \left\vert z \right\vert}{ \sigma}} {\rm d}z.
\end{align*}
Now, $ \frac{ \left\vert z \right\vert}{ \sigma} \geq c_{ 0} \frac{ \left\vert z \right\vert}{ \sigma}$ so that
\begin{align*}
\mathcal{ K}_{ 0}(W_{ y})(x)&\leq \left\Vert f^{ \prime} \right\Vert_{ \infty} \frac{ 1}{ (2 \sigma)^{ 2}} \int_{ \mathbb{ R}} e^{ - \frac{ \left\vert z+y-x \right\vert}{ \sigma}} e^{ - \frac{ c_{ 0}\left\vert z \right\vert}{ \sigma}} {\rm d}z = \left\Vert f^{ \prime} \right\Vert_{ \infty} \frac{ 1}{ (2 \sigma)^{ 2}}\frac{ 2 \sigma e^{ - \frac{ c_{ 0} \left\vert y-x \right\vert}{ \sigma}}- 2 c_{ 0} \sigma e^{ - \frac{ \left\vert y-x \right\vert}{ \sigma}}}{ 1- c_{ 0}^{ 2}},\\
&\leq \left\Vert f^{ \prime} \right\Vert_{ \infty} \frac{ 1}{2 \sigma}\frac{e^{ - \frac{ c_{ 0} \left\vert y-x \right\vert}{ \sigma}}}{ 1- c_{ 0}^{ 2}}.
\end{align*}
Then
\begin{align*}
\mathcal{ K}_{ 0}^{ 2}(W_{y})(x) &= \int_{ \mathbb{ R}} W(z-x) f^{ \prime}(\hat{ u}_{ 0}(z)) \mathcal{ K}_{ 0}W_{ y}(z) {\rm d}z,\\
&\leq \left\Vert f^{ \prime} \right\Vert_{ \infty}^{ 2} \frac{ 1}{ (2 \sigma)^{ 2}(1-c_{ 0}^{ 2})}\int_{ \mathbb{ R}} e^{ - \frac{ \left\vert z-x \right\vert}{ \sigma}}e^{ - \frac{ c_{ 0} \left\vert z-y \right\vert}{ \sigma}} {\rm d}z\leq  \left\Vert f^{ \prime} \right\Vert_{ \infty}^{ 2} \frac{ 1}{2 \sigma(1-c_{ 0}^{ 2})^{ 2}} e^{ - \frac{ c_{ 0} \left\vert y-x \right\vert}{ \sigma}}
\end{align*}
so that an easy induction leads to
\begin{align*}
\mathcal{ K}_{ 0}^{ k}(W_{ y})(x) \leq \frac{ 1}{ 2 \sigma} \left( \frac{ \left\Vert f^{ \prime} \right\Vert_{ \infty}}{ 1-c_{ 0}^{ 2}}\right)^{ k} e^{ - \frac{ c_{ 0} \left\vert y-x \right\vert}{ \sigma}}.
\end{align*}
In particular,
\begin{align*}
e^{ s \mathcal{ L}_{ 0}}(W_{ y})(x) &\leq e^{ -s} \left(W(x-y) + \frac{ 1}{ 2 \sigma}  \sum_{ k=1}^{ +\infty} \frac{ s^{ k}}{ k!} \left( \frac{ \left\Vert f^{ \prime} \right\Vert_{ \infty}}{ 1-c_{ 0}^{ 2}}\right)^{ k} e^{ - \frac{ c_{ 0} \left\vert y-x \right\vert}{ \sigma}}\right),\\
&\leq  \frac{ e^{ -s}}{ 2 \sigma} e^{ - \frac{ c_{ 0} \left\vert y-x \right\vert}{ \sigma}}\left(1 + \sum_{ k=1}^{ +\infty} \frac{ s^{ k}}{ k!} \left( \frac{ \left\Vert f^{ \prime} \right\Vert_{ \infty}}{ 1-c_{ 0}^{ 2}}\right)^{ k} \right)=\frac{ e^{ -s}}{ 2 \sigma} e^{ s\left( \frac{ \left\Vert f^{ \prime} \right\Vert_{ \infty}}{ 1-c_{ 0}^{ 2}}\right)} e^{ - \frac{ c_{ 0} \left\vert y-x \right\vert}{ \sigma}} 
\end{align*}
which gives \eqref{eq:eL0yx}. Let us turn now to \eqref{eq:h7}: we have
\begin{align*}
h_{ 7}(y)&\leq C c_{ 1}^{ 2}\left\Vert f^{ \prime \prime} \right\Vert_{ \infty}\int_{ 0}^{a(y)}  e^{ 2c_{ 2}s} {\rm d}s \int_{ \mathbb{ R}}\hat{ u}^{ \prime}_{ 0}(x)   e^{ - \frac{ 2c_{ 0} \left\vert x-y \right\vert}{ \sigma}}{\rm d}x,\\
&\leq \frac{ C c_{ 1}^{ 2}}{ 2 c_{ 2}}\left\Vert f^{ \prime \prime} \right\Vert_{ \infty}  e^{ 2c_{ 2}a(y)} e^{ - \frac{ 2c_{ 0} \left\vert y \right\vert}{ \sigma}} \int_{ \mathbb{ R}}\hat{ u}^{ \prime}_{ 0}(x)   e^{ \frac{ 2c_{ 0} \left\vert x \right\vert}{ \sigma}}{\rm d}x.
\end{align*}
It suffices to choose $c_{ 0}\in(0,1)$ sufficiently small to see that the last previous integral is finite and that $h_{ 7}$ is integrable in $y$. This proves \eqref{eq:conv_veps}.}

\medskip

We can now turn to the other terms appearing in the decomposition of $D_t$ in \eqref{aux:Dt}. We have
\begin{align}
\left|D_{2,t}\right|\leq \change{ \frac{ 1}{ 2}}\int_0^t\left(\sum_{j\in \mathbb{Z}}\left({\rm D}^{2}\Theta \left(\hat u_{\Theta(U_s)} \right) \left[ \varpi_{ j}, \varpi_{ j}\right]\right)^2\right)^\frac12
\left(\sum_{j\in \mathbb{Z}}\left(f(U_{j,s^-})-f(\hat u_{\Theta(U_s)})(x_j)\right)^2\right)^\frac12
{\rm d}s. \label{aux:CS_D2}
\end{align}
As before the $s^-$ may be replaced by $s$ since there is almost surely a countable number of jumps before $t_*$, and
\begin{align*}
\sum_{j\in \mathbb{Z}}\left({\rm D}^{2}\Theta \left(\hat u_{\Theta(U_s)} \right) \left[ \varpi_{ j}, \varpi_{ j}\right]\right)^2\leq C \sum_{j\in \mathbb{Z}}\Vert \varpi_{ j}\Vert_{L^2}^4\leq C \varepsilon^{3-\beta},
\end{align*}  
where we applied Lemma~\ref{lem:bound_sum_varpi}. Moreover,
\begin{align}
\sum_{j\in \mathbb{Z}}&\left(f(U_{j,s})-f(\hat u_{\Theta(U_s)})(x_{ j})\right)^2 \label{aux:Deltaf}\\
& = \frac{1}{\varepsilon}\sum_{j\in \mathbb{Z}}\int_{I_j}\left(f(U_{s})(y)-f(\hat u_{\Theta(U_s)})(x_j)\right)^2{\rm d}y \nonumber\\
&=\frac{1}{\varepsilon}\sum_{j\in \mathbb{Z}}\int_{I_j}\left(f(U_{s})(y)-f(\hat u_{\Theta(U_s)})(y)-\int_{ \change{x_j}}^y f'(\hat u_{\Theta(U_s)})(r)\hat u'_{\Theta(U_s)}(r){\rm d}r\right)^2{\rm d}y \nonumber\\
&\leq \frac{2\Vert U_s-\hat u_{\Theta(U_s)}\Vert^2_{L^2}}{\varepsilon}
+C\sum_{j\in \mathbb{Z}}\int_{I_j}\left|\int_{ \change{x_j}}^y \left( \change{\hat{ u}'}_{\Theta(U_s)}\right)^2(r){\rm d} r\right|{\rm d}y \nonumber\\
&\leq \change{C\varepsilon^{- \frac{ 1}{ 4}-2\beta-2\delta}}+C\varepsilon \int_{\mathbb{R}} \left( \change{ \hat{u}^{ \prime}}_{\Theta(U_s)}\right)^2(r){\rm d} r\leq  \change{C\varepsilon^{- \frac{ 1}{ 4}-2\beta-2\delta}},\nonumber
\end{align}
where we recall that we are on the event $ \mathcal{ O}_{ \varepsilon, \delta, t_{ 1}}$ given in \eqref{eq:event_Oeps}. Then,
\begin{align*}
\left|D_{2,t}\right|\leq \change{C \varepsilon^{-1} \varepsilon^{ \frac{ 3}{ 2}- \frac{ \beta}{ 2}}\varepsilon^{ - \frac{ 1}{ 8}-\beta-\delta}\leq C \varepsilon^{ \frac{ 3}{ 8}- \frac{ 3\beta}{2}-\delta}}.
\end{align*}
Concerning $D_{3,t}$ in \eqref{aux:Dt}, we get, relying on Proposition~\ref{prop:regularity isochron} and Lemma~\ref{lem:bound_sum_varpi}, since  we have $ \left\Vert \varpi_{ j} \right\Vert_{ L^{ \infty}}\leq \varepsilon$,
\begin{align*}
\left|D_{3,t}\right|&\leq C\sum_{j\in \mathbb{Z}}\int_0^t \Vert U_s-\hat u_{\Theta(U_s)}\Vert_{L^2}\Vert \varpi_{ j}\Vert_{L^2}\left(\Vert \varpi_{ j}\Vert_{L^2}+\Vert \varpi_{ j}\Vert_{L^\infty}\right)\\
&\leq \change{C \varepsilon^{ -1}\varepsilon^{ \frac{ 3}{ 8}-\beta-\delta}}\sum_{j\in \mathbb{Z}}\left(\Vert \varpi_{ j}\Vert_{L^2}^2+\varepsilon\Vert \varpi_{ j}\Vert_{L^2}\right)\leq \change{C\varepsilon^{ \frac{ 3}{ 8}-2\beta-\delta}}.
\end{align*}
Finally,
\begin{align*}
\left| \change{D_{4,t}}\right|&\leq C\varepsilon^{-1}\sum_{j\in \mathbb{Z}} \Vert \varpi_{ j}\Vert_{L^2}^2\left(\Vert \varpi_{ j}\Vert_{L^2}+\Vert \varpi_{ j}\Vert_{L^\infty}\right)\leq C\varepsilon^{2-\beta}.
\end{align*}
This concludes the proof of Proposition~\ref{prop: bound Theta}, since on the event $\mathcal{A}_{\varepsilon,\delta}$ given in \eqref{eq:event_mcA}, we have $| \Xi_t|\leq C \varepsilon^{-\beta/2-\delta}$.
\end{proof}
\begin{remark}
\label{rem:beta_delta_small}
The reason for the choice of $ \beta$ and $ \delta$ satisfying \eqref{eq:beta_small} and \eqref{eq:delta_small} becomes apparent in the previous proof of Proposition~\ref{prop: bound Theta}: note in particular that, on a time scale of order $ \varepsilon^{ -1}$, all the terms in the Ito decomposition \eqref{eq:decomp Theta} (exception made of the \change{drift term (of order $1$) and the} noise term $ \Xi_{ t}$) are actually vanishing as $ \varepsilon\to 0$ (of order at most $ \change{\varepsilon^{ \frac{ 1}{ 8} - 3 \beta - 3 \delta}}$). The a priori bound on $ \Xi_{ t}$ being (with high probability) of order $ \varepsilon^{ - \frac{ \beta}{ 2} - \delta}$, this is anyway smaller than the $ \varepsilon^{ - \frac{ 3 \beta}{ 4}}$ bound given by \eqref{eq:Theta_3b4}, since \eqref{eq:delta_small} further supposes that $ \delta< \frac{ \beta}{ 4}$.
\end{remark}
\subsection{Proof of Proposition~\ref{prop:closeness M}}

Let us begin with a technical Lemma.
\begin{lemma}
\label{lem:compare_D0_L2}
Let $U\in L^{ 2} \left(\mathcal{ D}_{ 0}\right)$ and define
\begin{equation*}
\bar U = a_{ 1} \mathbf{ 1}_{ \mathcal{ D}_{-}} + U \mathbf{ 1}_{ \mathcal{ D}_{ 0}} + a_{ 2} \mathbf{ 1}_{ \mathcal{ D}_{ +}}.
\end{equation*}
Then, for any $\eta\in (0,1)$, there exist some constants $C, c>0$ such that for all $ \left\vert \psi \right\vert\leq \eta \varepsilon^{ - \beta}$,
\begin{equation*}
\left\Vert \bar U - \hat{ u}_{ \psi} \right\Vert_{ L^{ 2}} \leq \left\Vert U- \hat{u}_{ \psi} \right\Vert_{ L^{ 2} \left(\mathcal{ D}_{ 0}\right)} + C e^{ - c \varepsilon^{ - \beta} }.
\end{equation*}
\end{lemma}
\begin{proof}[Proof of Lemma~\ref{lem:compare_D0_L2}]
  This is an easy consequence of the tail estimates of Lemma~\ref{lem:int_tail_hatu} and the fact that $ \ell - \psi \geq (1-\eta) \varepsilon^{ - \beta}$.
\end{proof}
We are finally in position to prove the main result of this Section.

\begin{proof}[Proof of Proposition~\ref{prop:closeness M}] Recall the definitions of the events $\mathcal{C}_\varepsilon$ and $ \mathcal{ O}_{ \varepsilon, \delta, t_{ 1}}$ in \eqref{eq:event_Cepsilon} and \eqref{eq:event_Oeps} respectively. For all $k=0, \ldots, n_{ \varepsilon}$, define
\begin{equation}
\label{eq:event_Cepsk}
\mathcal{ C}_{ \varepsilon}^{ (k)}:= \mathcal{ C}_{ \varepsilon} \cap \mathcal{ O}_{ \varepsilon, \delta, \varepsilon T_{ k}} \cap \left\lbrace {\rm dist}_{ L^{ 2}} \left(U_{ T_{ k}}, \mathcal{ M}\right) \leq \change{\varepsilon^{ \frac{ 3}{ 8} -  \beta - \frac{ 3 \delta}{ 4}}}\right\rbrace.
\end{equation} 
The point is to prove the following inclusion:
\begin{equation}
\label{eq:inclusion_Ck}
\mathcal{ C}_{ \varepsilon}^{ (k-1)} \subset \mathcal{ C}_{ \varepsilon}^{ (k)},\ k=1, \ldots, n_{ \varepsilon}-1.
\end{equation}
Note that Proposition~\ref{prop:closeness M} follows directly from \eqref{eq:inclusion_Ck}, as then, $ \mathcal{ C}_{ \varepsilon}= \mathcal{ C}_{ \varepsilon}^{ (0)} \subset \cap_{ k=1, \ldots, n_{  \varepsilon}} \mathcal{ C}_{ \varepsilon}^{ (k)}$. Therefore, let $k=1, \ldots, n_{ \varepsilon}-1$ and place ourselves on $ \mathcal{ C}_{  \varepsilon}^{ (k-1)}$.
Note here that we have
\begin{align*}
\left\vert \psi_{ k-1} \right\vert & \leq \left\vert \psi_{ k-1} - \Theta \left( U_{ T_{ k-1}}\right)\right\vert + \left\vert \Theta \left(U_{ T_{ k-1}}\right) \right\vert\\
&= \left\vert {\rm proj}_{ \mathcal{ M}} \left(U_{ T_{ k-1}}\right) - {\rm proj}_{ \mathcal{ M}} \left( \hat{ u}_{ \Theta \left( U_{ T_{ k-1}}\right)}\right)\right\vert + \left\vert \Theta \left(U_{ T_{ k-1}}\right) \right\vert\\
& \leq C_{ {\rm proj}} \left\Vert U_{ T_{ k-1}} - \hat{ u}_{ \Theta \left(U_{ T_{ k-1}}\right)}\right\Vert_{ L^{ 2}} + \left\vert \Theta \left(U_{ T_{ k-1}}\right) \right\vert\\
& \leq C_{ {\rm proj}} C_{ \Theta} \sup_{ 0\leq t \leq T_{ k-1}}{\rm dist}_{ L^{ 2}} \left(U_{ t}, \mathcal{ M}\right) + \left\vert \Theta \left(U_{ T_{ k-1}}\right) \right\vert ,
\end{align*} by Lipschitz continuity of the projection $ {\rm proj}_{ \mathcal{ M}}$ and \eqref{eq:Theta_dist}. Therefore, by Proposition~\ref{prop: bound Theta}, on $ \mathcal{C}_{ \varepsilon}^{ (k-1)}$,
\begin{equation}
\label{eq:apriori_psik}
\left\vert \psi_{ k-1} \right\vert \leq C_{ {\rm proj}} C_{ \Theta} \change{\varepsilon^{ \frac{ 3}{ 8} - \beta - \delta}} +  \varepsilon^{ - 3\beta/4}\leq \frac12 \varepsilon^{-\beta},
\end{equation}
where the last inequality holds for $ \varepsilon>0$ small enough, independent of $k$. 
\medskip

\noindent
{\it First step.} We first prove that, on $ \mathcal{ C}_{ \varepsilon}^{ (k-1)}$
\begin{equation}
\label{eq:step1}
\sup_{ T_{ k-1}\leq t \leq T_{ k}} {\rm dist} \left(U_{ t}, \mathcal{ M}\right) \leq \change{\varepsilon^{ \frac{ 3}{ 8} - \beta - \delta}}.
\end{equation}
To this end, define
\begin{equation}
\label{eq:tk}
t_k :=\inf\left\{t\in [0,T]:\, \left\Vert Y_{ k,t} \right\Vert_{m_{ \psi_{ k-1}}}> \change{\varepsilon^{\frac38-\beta-\frac{7\delta}{8}}}\right\}.
\end{equation}
Our first aim is to show that $t_k=T$. Supposing that it is indeed the case, we will deduce that $\sup_{ 0\leq t \leq T} \left\Vert U_{ T_{ k-1}+t} - \hat{ u}_{ \psi_{ k-1}}\right\Vert_{ L^{ 2} \left( \mathcal{ D}_{ 0}\right)} \leq \change{\varepsilon^{ \frac{ 3}{ 8} - \beta - \frac{ 7 \delta}{ 8}}}$ and hence, by Lemma~\ref{lem:compare_D0_L2}, $\sup_{ T_{ k-1}\leq t \leq T_{ k}} {\rm dist}_{ L^{ 2}} \left(U_{t}, \mathcal{ M}\right) \leq \change{\varepsilon^{ \frac{ 3}{ 8} - \beta - \frac{ 7 \delta}{ 8}}} + C e^{ -c \varepsilon^{ - \beta}} \leq \change{\varepsilon^{ \frac{ 3}{ 8} - \beta -  \delta}}$ for $ \varepsilon>0$ sufficiently small (independently on $k$), which gives \eqref{eq:step1}. Relying on the mild formulation \eqref{eq:Ykt mild}, we obtain the following decomposition, for $t\in [0,T]$ such that $T_{k-1}+t\leq \tau$,
\begin{equation}
\label{eq:mildYk_r}
Y_{k,t} = e^{t L_{\psi_{k-1}}}Y_{k,0} + \int_0^t  e^{ (t-s)L_{\psi_{k-1}}} W\ast g_{k,s}{\bf 1}_{ \mathcal{ D}_{ 0}}{\rm d}s + \sum_{i=-N_\varepsilon}^{N_\varepsilon}\left(r_{k, i,t }^{ (-)}+ r_{k, i, t}^{ (0)} + r_{k, i, t}^{ (+)}\right)+\zeta_{k,t},
\end{equation}
with
\begin{align*}
r_{k, i, t}^{ (0)} &= \int_0^t e^{ (t-s)L_{\psi_{k-1}}} \Delta_{ T_{ k-1}+s, \varepsilon, i}^{ (0)} \mathbf{ 1}_{ I_{ i}}{\rm d}s ,\\
r_{k, i, t}^{ (+)} &= \int_0^t e^{ (t-s)L_{\psi_{k-1}}} \Delta_{ T_{ k-1}+s, \varepsilon, i}^{ (+)} \mathbf{ 1}_{ I_{ i}}{\rm d}s ,\\
r_{k, i, t}^{ (-)}&= \int_0^t e^{ (t-s)L_{\psi_{k-1}}} \Delta_{ T_{ k-1}+s, \varepsilon, i}^{ (-)} \mathbf{ 1}_{ I_{ i}}{\rm d}s,
\end{align*}
where we recall the definitions of $ \Delta^{ (0)}, \Delta^{ (\pm)}$ in Definition~\ref{def:Deltas}. Using \eqref{eq:etL0exp}, we have
\begin{align*}
\left\Vert e^{t L_{\psi_{k-1}}}Y_{k,0}  \right\Vert_{m_{\psi_{k-1}}, \mathcal{ D}_{ 0}}
& \change{\leq e^{-\kappa t} \left\Vert Y_{k,0}  \right\Vert_{m_{\psi_{k-1}}, \mathcal{ D}_{ 0}}}\\
& \change{+ \left\Vert \hat{ u}^{ \prime}_{ \psi_{ k-1}} \mathbf{ 1}_{ \mathcal{ D}_{ 0}^{ c}} \right\Vert_{ m_{ \psi_{ k-1}}} \left(\left\Vert \left(U_{T_{k-1}+t}^{ (\varepsilon)}-\hat u_{\psi_{k-1}}\right)\mathbf{ 1}_{ \mathcal{ D}_{ 0}^{ c}}\right\Vert_{ m_{ \psi_{ k-1}}}  + c_{\mathcal L} t \left\Vert Y_{ k, 0} \right\Vert_{ m_{ \psi_{ k-1}, \mathcal{ D}_{ 0}}} \right).}
\end{align*}
Note that using the apriori bound \eqref{eq:apriori_psik} on $ \psi_{ k}$ and the estimate \eqref{eq:uprime_exp}, we have, for some $c_{ 0}>0$
\begin{align*}
\left\Vert \hat{ u}^{ \prime}_{ \psi_{ k-1}} \mathbf{ 1}_{ \mathcal{ D}_{ 0}^{ c}} \right\Vert_{ m_{ \psi_{ k-1}}} & \leq c_{ 0} e^{ - \frac{ \mu}{ 2} \varepsilon^{ - \beta}}.
\end{align*}
\change{Note also that $\left\Vert \left(U_{T_{k-1}+t}^{ (\varepsilon)}-\hat u_{\psi_{k-1}}\right)\mathbf{ 1}_{ \mathcal{ D}_{ 0}^{ c}}\right\Vert_{ m_{ \psi_{ k-1}}}\leq c_{ 1} e^{ - c_{ 2} \varepsilon^{ - \beta}}$.}
Using that $ \left\Vert Y_{ k, 0} \right\Vert_{ m_{ \psi_{ k-1}, \mathcal{ D}_{ 0}}}= \left\Vert \left(U_{ T_{ k-1}} - \hat{ u}_{ \psi_{ k-1}}\right) \mathbf{ 1}_{ \mathcal{ D}_{ 0}} \right\Vert_{ m_{ \psi_{ k-1}}} \leq \left\Vert U_{ T_{ k-1}} - \hat{ u}_{ \psi_{ k-1}}\right\Vert_{ m_{ \psi_{ k-1}}} $ and \eqref{eq:proj_dist}, we obtain
\begin{align*}
\left\Vert e^{t L_{\psi_{k-1}}}Y_{k,0}  \right\Vert_{m_{\psi_{k-1}}, \mathcal{ D}_{ 0}}
& \change{\leq  C_{ \mathcal{ M}} {\rm dist}_{L^2}\left(U_{T_{k-1}},\mathcal{M}\right) \left(e^{-\kappa t} + c_{ 0} c_{\mathcal L} t e^{ - \frac{ \mu}{ 2} \varepsilon^{ - \beta}} \right)+ c_{ 0} c_{ 1} e^{ - \left(\frac{ \mu}{ 2} + c_{ 2}\right) \varepsilon^{ - \beta}}.}
\end{align*}
Hence, we deduce, on $ \mathcal{ C}_{ \varepsilon}^{ (k-1)}$, for $ \varepsilon>0$ small
\begin{align}
\sup_{0\leq t\leq T} \left\Vert e^{t L_{\psi_{k-1}}}Y_{k,0}\right\Vert_{m_{\psi_{k-1}}} &\leq \change{C_{ \mathcal{ M}} \left(1 + c_{ 0} c_{\mathcal L} T \right) \varepsilon^{ \frac{ 3}{ 8} -  \beta - \frac{ 3 \delta}{ 4}}+ c_{ 0} c_{ 1} e^{ - \left(\frac{ \mu}{ 2} + c_{ 2}\right) \varepsilon^{ - \beta}}}, \label{aux:etYk0}\\
\left\Vert e^{T L_{\psi_{k-1}}}Y_{k,0}\right\Vert_{m_{\psi_{k-1}}} &\leq \change{C_{ \mathcal{ M}}  \left(e^{-\kappa T} + c_{ 0} c_{\mathcal L} T e^{ - \frac{ \mu}{ 2} \varepsilon^{ - \beta}} \right) \varepsilon^{\frac38-\beta - \frac{3\delta}{4}} + c_{ 0} c_{ 1} e^{ - \left(\frac{ \mu}{ 2} + c_{ 2}\right) \varepsilon^{ - \beta}}.} \label{aux:eTYk0}
\end{align}

Concerning the quadratic term in \eqref{eq:mildYk_r}, one has, using \eqref{eq:bound_eLphi},
\begin{align*}
\left\Vert\int_0^t  e^{ (t-s)L_{\psi_{k-1}}} W\ast g_{k,s}{\bf 1}_{ \mathcal{ D}_{ 0}} {\rm d}s\right\Vert_{m_{\psi_{k-1}}}
&\leq \int_0^t \left\Vert W\ast g_{k,s}{\bf 1}_{ \mathcal{ D}_{ 0}} \right\Vert_{m_{\psi_{k-1}}}{\rm d}s\\
&\leq  \left\Vert m_{\psi_{k-1}}\right\Vert_{L^\infty} \left\Vert W  \right\Vert_{ L^2}\int_0^t  \left\Vert g_{k,s} \right\Vert_{ 1}{\rm d}s,
\end{align*}
where we have used Young's convolution inequality. Recalling \eqref{eq:gpsi}, we obtain by boundedness of $f^{ \prime\prime}$ that $ \left\Vert g_{ k, s} \right\Vert_{ 1} \leq C \left\Vert Y_{ k,s} \right\Vert_{L^2 \left(\mathcal{ D}_{ 0}\right)}^{ 2} \leq C \left\Vert Y_{ k,s} \right\Vert_{m_{ \psi_{ k-1}}}^{ 2}\leq \change{\varepsilon^{\frac34-2\beta-\frac{7\delta}{4}}}$, for $ s\leq t_{ k}$ (recall the definition of $t_{ k}$ in \eqref{eq:tk}).

Concerning the last terms in \eqref{eq:mildYk_r}, we use here the bounds that we have found in Proposition~\ref{prop:bound_approx}: since $r_{k, i, t}^{ (0)}= \int_0^t e^{ (t-s)L_{\psi_{k-1}}} \Delta_{ T_{ k-1}+s, \varepsilon, i}^{ (0)} \mathbf{ 1}_{ I_{ i}}{\rm d}s$, we have
\begin{align*}
\left\Vert \sum_{ i=-N_{ \varepsilon}}^{ N_{ \varepsilon}} r_{k, i, t}^{ (0)} \right\Vert_{ m_{ \psi_{ k-1}}} \leq T \left\Vert \sum_{ i=-N_{ \varepsilon}}^{ N_{ \varepsilon}} \Delta_{ T_{ k-1}+s, \varepsilon, i}^{ (0)} \mathbf{ 1}_{ I_{ i}} \right\Vert \leq CT \varepsilon^{ 1- \frac{ \beta}{ 2}},
\end{align*}
by \eqref{eq:Norm_Delta0}. In a same way, we obtain from \eqref{eq:Norm_Delta+}
\begin{equation*}
\left\Vert \sum_{ i=-N_{ \varepsilon}}^{ N_{ \varepsilon}} r_{ k,i,t}^{ (\pm)}\right\Vert_{ m_{ \psi_{ k-1}}} \leq CT \varepsilon. 
\end{equation*}
Recalling finally that we are on the event $ \mathcal{ B}_{ \varepsilon}$ (recall \eqref{eq:event_mcB}), and gathering all the previous estimates in \eqref{eq:mildYk_r} gives:
\begin{align*}
\sup_{ t\in [0, T]} \left\Vert Y_{ k,t} \right\Vert_{ m_{ \psi_{ k-1}}} &\leq \change{C_{ \mathcal{ M}} \left(1 + c_{ 0} c_{\mathcal L} T \right) \varepsilon^{ \frac{ 3}{ 8} -  \beta - \frac{ 3 \delta}{ 4}}+ c_{ 0} c_{ 1} e^{ - \left(\frac{ \mu}{ 2} + c_{ 2}\right) \varepsilon^{ - \beta}}} + CT \change{\varepsilon^{ \frac{ 3}{ 4} - 2 \beta - \frac{ 7\delta}{ 4}}} \\
&\quad + CT \varepsilon^{ 1 - \frac{ \beta}{ 2}} + CT \varepsilon+ C \change{T^{ \frac{ 5}{ 8}} \varepsilon^{ \frac{ 3}{ 8}- \beta}}.
\end{align*}
Under the conditions \eqref{eq:beta_small} and \eqref{eq:delta_small} on the parameters $ \beta, \delta$, we have for $ \varepsilon>0$ sufficiently small (depending only on $T$ fixed defined in \eqref{hyp:T}) $\sup_{ t\in [0, T]} \left\Vert Y_{ k,t} \right\Vert_{ m_{ \psi_{ k-1}}} \leq \change{\varepsilon^{ \frac{ 3}{ 8}-  \beta - \frac{ 3 \delta}{ 4}} \varepsilon^{ - \frac{ \delta}{ 16}}}< \change{\varepsilon^{ \frac{ 3}{ 8}- \beta - \frac{ 7 \delta}{ 8}}}$. Hence $t=t_{ k}$ given in \eqref{eq:tk} and we obtain \eqref{eq:step1}.

\medskip

\noindent
{\it Second step.} We prove that
\begin{equation}
\label{eq:step2}
{\rm dist}_{ L^{ 2}} \left(U_{ T_{ k}}, \mathcal{ M}\right) \leq \change{\varepsilon^{ \frac{ 3}{ 8} - \beta- \frac{ 3 \delta}{ 4}}}.
\end{equation}
Indeed, we have 
\begin{align}
 {\rm dist}_{ L^{ 2}} \left( U_{ T_{ k}}, \mathcal{ M}\right) &\leq \left\Vert U_{ T_{ k}} - \hat{ u}_{ \psi_{ k-1}}\right\Vert_{ L^{ 2}} \leq \left\Vert Y_{ k, T} \right\Vert_{ L^{ 2} \left(\mathcal{ D}_{ 0}\right)} + C e^{ - c \varepsilon^{ - \beta}} \nonumber\\
 &\leq \frac{ 1}{ \sqrt{ f^{ \prime}(a_{ 1})}}\left\Vert Y_{ k, T} \right\Vert_{ m_{ \psi_{ k-1}}} + C e^{ - c \varepsilon^{ - \beta}} \label{aux:UTk_contract}
\end{align}
by Lemma~\ref{lem:compare_D0_L2} and \eqref{eq:norm_equiv}. Going back now to \eqref{eq:mildYk_r} for the choice of $t=T$, we obtain, using \eqref{aux:eTYk0}
\begin{align*}
\left\Vert Y_{ k, T} \right\Vert_{ m_{ \psi_{ k-1}}} &\leq \change{C_{ \mathcal{ M}}  \left(e^{-\kappa T} + c_{ 0} c_{\mathcal L} T e^{ - \frac{ \mu}{ 2} \varepsilon^{ - \beta}} \right) \varepsilon^{\frac38-\beta - \frac{3\delta}{4}} + c_{ 0} c_{ 1} e^{ - \left(\frac{ \mu}{ 2} + c_{ 2}\right) \varepsilon^{ - \beta}}} \\
&\quad + \change{CT \varepsilon^{ \frac{ 3}{ 4} - 2 \beta - \frac{ 7\delta}{ 4}}} + CT \varepsilon^{ 1 - \frac{ \beta}{ 2}} + CT \varepsilon+ \change{CT^{ 5/8} \varepsilon^{ \frac{ 3}{ 8}- \beta}},
\end{align*}
so that, using \eqref{aux:UTk_contract},
\begin{align*}
 {\rm dist}_{ L^{ 2}} \left( U_{ T_{ k}}, \mathcal{ M}\right)  &\leq \change{\varepsilon^{ \frac{ 3}{ 8}-\beta - \frac{3\delta}{4}}}\left\lbrace \frac{ C_{ \mathcal{ M}}}{ \sqrt{ f^{ \prime}(a_{ 1})}} \change{\left(e^{-\kappa T} + c_{ 0} c_{\mathcal L} T e^{ - \frac{ \mu}{ 2} \varepsilon^{ - \beta}} \right)} + \rho_{ T} \left(\varepsilon\right) \right\rbrace,
\end{align*}
with $ \rho_{ T}( \varepsilon)\to 0$ as $ \varepsilon\to 0$. Recall now the definition of $T$ in \eqref{hyp:T} and for this fixed choice of $T$, choose $ \varepsilon>0$ sufficiently small such that $\frac{ C_{ \mathcal{ M}} \change{c_{ 0} c_{\mathcal L} T}}{ \sqrt{ f^{ \prime}(a_{ 1})}} e^{ - \frac{ \mu}{ 2} \varepsilon^{ - \beta}} + \rho_{ T} \left(\varepsilon\right) \leq \frac{ 2}{ 3}$. This gives \eqref{eq:step2} and Proposition~\ref{prop:closeness M} is proven.
\end{proof}

\section{Proof of Theorem~\ref{th:main}}
\label{sec:proof_main}
We prove in this Section the main theorem of the paper, Theorem~\ref{th:main}. Note that the proof of the finite-time convergence \eqref{eq:conv_main_T0} is elementary: it is a direct consequence of the mild formulation \eqref{eq:Ykt mild} applied for $k=0$ and to similar estimates that have been used for the proof of Proposition~\ref{prop:closeness M}. Hence, we leave to the reader the details and prove the more difficult remaining part of Theorem~\ref{th:main}. Recall that $ \beta, \delta$ are defined through \eqref{eq:beta_small} and \eqref{eq:delta_small} respectively. Let us start with the initial condition: the proximity statement of Theorem~\ref{th:main} requires obviously that the initial profile $ U_{ 0}^{ (\varepsilon)}$ should itself be close to the manifold $ \mathcal{ M}$. 
Since
\begin{align*}
U_0^{ (\varepsilon)}(x)=a_1{\bf 1}_{ \mathcal{ D}_-^{ (\varepsilon)}}(x) + \sum_{ i=-N_\varepsilon }^{N_\varepsilon} u_{ 0}^{ (\varepsilon)}(x_{ i}) {\bf 1}_{I_{i}^{ (\varepsilon)}}(x)+a_2{\bf 1}_{ \mathcal{ D}_+^{ (\varepsilon)}}(x),
\end{align*}
we immediately get that
\begin{align*}
\left\Vert U_{ 0}^{ (\varepsilon)} - \hat{ u}_{ 0}\right\Vert_{ L^{ 2}}^{ 2}&= \int_{ \mathcal{ D}_{ -}^{ (\varepsilon)}} \left\vert a_{ 1}- \hat{ u}_{ 0}(x) \right\vert^{ 2} {\rm d}x+ \int_{ \mathcal{ D}_{ +}^{ (\varepsilon)}} \left\vert a_{ 2}- \hat{ u}_{ 0}(x) \right\vert^{ 2} {\rm d}x \\
&\quad + \sum_{ i=- N_{ \varepsilon}}^{ N_{ \varepsilon}} \int_{ I_{ i}^{ (\varepsilon)}} \left\vert u_{ 0}^{ (\varepsilon)}(x_{ i}) - \hat{ u}_{ 0}(x)\right\vert^{ 2} {\rm d}x\\
&\leq 2C e^{ - 2 \mu \left(N_{ \varepsilon} + \varepsilon/2\right)}\\
&\quad  + 2\sum_{ i=- N_{ \varepsilon}}^{ N_{ \varepsilon}} \int_{ I_{ i}^{ (\varepsilon)}} \left\vert u_{ 0}^{ (\varepsilon)}(x_{ i}) - u_{ 0}^{ (\varepsilon)}(x)\right\vert^{ 2} {\rm d}x +2\int_{ \mathcal{ D}_{ 0}^{ (\varepsilon)}} \left\vert u_{ 0}^{ (\varepsilon)}(x) - \hat{ u}_{ 0}(x)\right\vert^{ 2} {\rm d}x,
\end{align*}
by \eqref{eq:int_tail_hatu}. Then, by Lipchitz continuity of $u_{ 0}$, there exists some $C>0$ such that 
\begin{align*}
\sum_{ i=- N_{ \varepsilon}}^{ N_{ \varepsilon}} \int_{ I_{ i}^{ (\varepsilon)}} \left\vert u_{ 0}^{ (\varepsilon)}(x_{ i}) - u_{ 0}^{ (\varepsilon)}(x)\right\vert^{ 2} {\rm d}x \leq C \varepsilon^{ 3} N_{ \varepsilon} = C \varepsilon^{ 2- \beta} ,
\end{align*}
so that, for some constant $C>0$,
\begin{align*}
\left\Vert U_{ 0}^{ (\varepsilon)} - \hat{ u}_{ 0}\right\Vert_{ L^{ 2}} \leq C \left(e^{ - \mu  \varepsilon^{ -1- \beta}} +  \varepsilon^{ 1- \beta/2} + \left\Vert u_{ 0}^{ (\varepsilon)} - \hat{ u}_{ 0}\right\Vert_{ L^{ 2}}\right).
\end{align*}
Therefore, fix any \change{$ \beta_{ 0} >\frac{ 1}{ 2} - 2 \beta - \delta$} and suppose that $ \left\Vert u_{ 0}^{ (\varepsilon)} - \hat{ u}_{ 0}\right\Vert_{ L^{ 2}} \leq C \varepsilon^{ \beta_{ 0}}$ as the hypothesis \eqref{eq:u0_close_to_M} in Theorem~\ref{th:main}. With this definition of $ \beta_{ 0}$, we see from the definition of the event $ \mathcal{ C}_{ \varepsilon, \delta}$ given in \eqref{eq:event_Cepsilon} (together with Proposition~\ref{prop:bound N, R} and Proposition~\ref{prop:bound zeta}) that $ \mathbb{ P} \left( \mathcal{C}_{  \varepsilon, \delta}\right) \xrightarrow[ \varepsilon\to 0]{}1$ and we conclude from Proposition~\ref{prop:closeness M} that, on $ \mathcal{ C}_{ \varepsilon, \delta}$,
\[
\sup_{0\leq t\leq t_f\varepsilon^{-1}} {\rm dist}_{L^2}\left(U^{ (\varepsilon)}_t,\mathcal{M}\right) \xrightarrow[ \varepsilon\to0]{}0.
\]
It remains to identify the phase process $ \psi^{ (\varepsilon)}$ of $ U_{ t \varepsilon^{ -1}}^{ (\varepsilon)}$ along the manifold $ \mathcal{ M}$. We proceed by identifying the asymptotic behavior as $ \varepsilon\to 0$ of the (rescaled) process
\begin{equation}
\label{eq:tildeTheta}
\left(\tilde{ \Theta}_{ u}^{ (\varepsilon)} \right)_{ u\in [0, t_{ f}]}:= \left( \Theta \left(U_{ u \varepsilon^{ -1}}^{ (\varepsilon)}\right)\right)_{ u\in[0, t_{ f}]}.
\end{equation}
Start from the Ito decomposition \eqref{eq:Ito to Theta}, written for $t=u \varepsilon^{ -1}$, $u\in[0, t_{ f}]$:
\begin{align}
\tilde{ \Theta}_{ u}^{ (\varepsilon)}  = &\, \tilde{ \Theta}_{ 0}^{ (\varepsilon)} -\int_0^{ u \varepsilon^{ -1}} {\rm D}\Theta(U_s)U_s  {\bf 1}_{\mathcal{D}_0} {\rm d} s \label{eq:Ito_Thetatilde}\\
&+ \sum_{j\in\mathbb{Z}}\int_0^{ u \varepsilon^{ -1}} \int_0^\infty\left(\Theta(U_s+\varpi_{j}\mathbf{ 1}_{ z\leq \lambda_{ j, s}} )-\Theta(U_s)\right){\rm d} s{\rm d} z+ \Xi_{ u \varepsilon^{ -1}}^{ (\varepsilon)}.\nonumber
\end{align}
The point is to prove that $\left(\tilde{ \Theta}_{ u}^{ (\varepsilon)} \right)_{ u\in [0, t_{ f}]}$ converges in law (as a càdlàg process on $[0, t_{ f}]$) as $ \varepsilon\to 0$ to $ \left( \change{u \bar v}+ \varrho B_{ u}\right)_{ u\in [0, t_{ f}]}$,  where  \change{$ \bar v$ and } $ \varrho$ are given in Remark~\ref{rem:v0} and Theorem~\ref{th:main} and $ B$ is a standard Brownian motion. Note first that, since $ \left\Vert U_{ 0}^{ (\varepsilon)} - \hat{ u}_{ 0}\right\Vert_{ L^{ 2}} \to 0$ as $ \varepsilon\to 0$, $ \tilde{ \Theta}_{ 0}^{ (\varepsilon)} \to \Theta \left( \hat{ u}_{ 0}\right)=0$ as $ \varepsilon\to 0$. Secondly, remark that, from the proof of Proposition~\ref{prop: bound Theta} (recall Remark~\ref{rem:beta_delta_small}),  on the event $ \mathcal{ C}_{ \varepsilon, \delta}$, the drift term in \eqref{eq:Ito_Thetatilde} converges uniformly to $ \change{u \bar v}$ on $[0, t_{ f} \varepsilon^{ -1}]$:
\begin{multline}
\label{eq:Thetadrift_to_0}
\sup_{ u\leq t_{ f}} \Bigg\vert -\int_0^{ u \varepsilon^{ -1}} {\rm D}\Theta(U_s)U_s  {\bf 1}_{\mathcal{D}_0} {\rm d} s 
\\
+ \sum_{j\in\mathbb{Z}}\int_0^{ u \varepsilon^{ -1}} \int_0^\infty\left(\Theta(U_s+\varpi_{j}\mathbf{ 1}_{ z\leq \lambda_{ j, s}} )-\Theta(U_s)\right){\rm d} s{\rm d} z \change{- u \bar v}\Bigg\vert \xrightarrow[ \varepsilon\to 0]{}0.
\end{multline}
Hence, everything boils down to proving that the process
\begin{equation}
\left( \tilde{\Xi}_{ u}^{ (\varepsilon)}\right)_{ u\in [0, t_{ f}]}:=\left(\Xi_{ u \varepsilon^{ -1}}^{ (\varepsilon)}\right)_{ u\in [0, t_{ f}]}
\end{equation}
converges to $ \left(\varrho B_{ u}\right)$. We use the standard procedure of tightness and identification of the limit: starting from \eqref{eq:Nt}, write
\begin{align}
\tilde{ \Xi}_{ u}^{ (\varepsilon)}&= \sum_{j\in\mathbb{Z}}\int_0^{u \varepsilon^{ -1}}\int_0^\infty \left(\Theta(U_s+\varpi_{j}\mathbf{ 1}_{ z\leq \lambda_{ j, s}} )-\Theta(U_s) - {\rm D} \Theta(U_{ s})[ \varpi_{ j} \mathbf{ 1}_{ z\leq \lambda_{ j,s}}]\right)\tilde \pi_j({\rm d} s,{\rm d}z) \nonumber\\
&\quad+\sum_{j\in\mathbb{Z}}\int_0^{u \varepsilon^{ -1}}\int_0^\infty \left({\rm D} \Theta(U_{ s})[ \varpi_{ j} \mathbf{ 1}_{ z\leq \lambda_{ j,s}}]- {\rm D} \Theta(\hat{ u}_{ \Theta(U_{ s})})[ \varpi_{ j} \mathbf{ 1}_{ z\leq \lambda_{ j,s}}]\right)\tilde \pi_j({\rm d} s,{\rm d}z) \nonumber\\
&\quad+\sum_{j\in\mathbb{Z}}\int_0^{u \varepsilon^{ -1}}\int_0^\infty {\rm D} \Theta(\hat{ u}_{ \Theta(U_{ s})})[ \varpi_{ j} \mathbf{ 1}_{ z\leq \lambda_{ j,s}}]\tilde \pi_j({\rm d} s,{\rm d}z):= \tilde{ \Xi}_{ 1, u}^{ (\varepsilon)} + \tilde{ \Xi}_{ 2,u}^{ (\varepsilon)} + \tilde{ \Xi}_{ 3,u}^{ (\varepsilon)}. \label{eq:tildeXi123}
\end{align}
We have by independence of the Poisson measures $ \pi_{ j}$
\begin{align}
\left[ \tilde{ \Xi}_{ 1}^{ (\varepsilon)}\right]_{ u}= \sum_{j\in\mathbb{Z}}\int_0^{u \varepsilon^{ -1}}\int_0^\infty \left(\Theta(U_s+\varpi_{j}\mathbf{ 1}_{ z\leq \lambda_{ j, s}} )-\Theta(U_s) - {\rm D} \Theta(U_{ s})[ \varpi_{ j} \mathbf{ 1}_{ z\leq \lambda_{ j,s}}]\right)^{ 2} \pi_j({\rm d} s,{\rm d}z).
\end{align}
Hence, by Proposition~\ref{prop:regularity isochron} and using that $ \lambda_{ j}\leq 1$
\begin{align*}
\mathbb{ E} \left( \left[ \tilde{ \Xi}_{ 1}^{ (\varepsilon)}\right]_{ u}\right) \leq C \varepsilon^{ -1} \sum_{ j\in \mathbb{ Z}} \left\Vert \varpi_{ j} \right\Vert^{ 4}\leq C \varepsilon^{ 2 - \beta}.
\end{align*}
In particular, by Burkholder-Davis-Gundy inequality, $ \mathbb{ E} \left( \sup_{ u\in[0, t_{ f}]}\left\vert \tilde{ \Xi}_{ 1, u}^{ (\varepsilon)} \right\vert^{ 2}\right) \xrightarrow[ \varepsilon\to 0]{}0$, so that $ \tilde{\Xi}_{ 1}^{ (\varepsilon)}$ does not contribute to the limit. \change{In a same way, on $ \mathcal{ C}_{ \varepsilon, \delta}$, 
\begin{align*}
\left\langle \tilde{\Xi}_{ 2}^{ (\varepsilon)} \right\rangle_{ u} &= \sum_{ j\in \mathbb{ Z}} \int_{ 0}^{u \varepsilon^{ -1}} \int_{ 0}^{+\infty} \left({\rm D} \Theta(U_{ s})[ \varpi_{ j} \mathbf{ 1}_{ z\leq \lambda_{ j,s}}]- {\rm D} \Theta(\hat{ u}_{ \Theta(U_{ s})})[ \varpi_{ j} \mathbf{ 1}_{ z\leq \lambda_{ j,s}}]\right)^{ 2} {\rm d}s {\rm d}z,\\
&\leq C \sum_{ j\in \mathbb{ Z}} \int_{ 0}^{u \varepsilon^{ -1}} \left\Vert U_{ s} - \hat{ u}_{ \Theta(U_{ s})} \right\Vert_{ L^{ 2}}^{ 2} \left\Vert \varpi_{ j} \right\Vert_{ L^{ 2}}^{ 2} {\rm d}s \leq C \varepsilon^{ -1} \varepsilon^{ \frac{ 3}{ 4} - 2 \beta - 2 \delta} \varepsilon^{ 1- \beta} = C \varepsilon^{ \frac{ 3}{ 4} - 3 \beta - 2 \delta} \xrightarrow[ \varepsilon\to 0]{}0
\end{align*} 
and, by \eqref{eq:bound_supvarpi}, $\sup_{ u\in [0, t_{ f}]} \left\vert \Delta \tilde{\Xi}_{ 2, u}^{ (\varepsilon)}\right\vert \leq C \varepsilon^{ \frac{ 3}{ 4} - 2 \beta - 2 \delta} \xrightarrow[ \varepsilon\to 0]{}0$.
Hence, we see from an application of Proposition~\ref{prop:concentration_mart}, as $ \mathbb{ P} \left( \mathcal{ C}_{ \varepsilon, \epsilon}\right)\to 1$ that $\tilde{\Xi}_{2, u}^{ (\varepsilon)}$ converges in probability to $0$ as $ \varepsilon\to 0$, uniformly on $[0, t_{ f}]$, so that this term does not contribute to the limit either.} We finally turn to $ \tilde{ \Xi}_{3}^{ (\varepsilon)} $ in \eqref{eq:tildeXi123}:
\begin{align*}
\left[ \tilde{ \Xi}_{ 3}^{ (\varepsilon)}\right]_{ u}&=\sum_{j\in\mathbb{Z}}\int_0^{u \varepsilon^{ -1}}\int_0^\infty  \left({\rm D} \Theta(\hat{ u}_{ \Theta(U_{ s})})[ \varpi_{ j}]\right)^{ 2} \mathbf{ 1}_{ z\leq \lambda_{ j,s}}\pi_j({\rm d} s,{\rm d}z) := \Gamma_{ 1,u}^{ (\varepsilon)} + \Gamma_{ 2, u}^{ (\varepsilon)}.
\end{align*}
where the decomposition of the last line corresponds to the splitting of the stochastic integral into $ \pi_{ j}({\rm d}s, {\rm d}z)= \tilde{ \pi}_{ j}({\rm d}s, {\rm d}z) + {\rm d}s {\rm d}z$. The first contribution $ \Gamma_{ 1}^{ (\varepsilon)}$ is negligible: its bracket is bounded by $ C\varepsilon^{ -1}\sum_{ j\in \mathbb{ Z}} \left\Vert \varpi_{ j} \right\Vert^{ 4}\leq C \varepsilon^{ 2- \beta}$ so that \change{another application of Proposition~\ref{prop:concentration_mart} gives that $\Gamma_{ 1}^{ (\varepsilon)}$ converges in probability to $0$ as $ \varepsilon\to0$, uniformly on $u\in [0, t_{ f}]$}. Hence everything boils down to estimating
\begin{align*}
\Gamma_{ 2, u}^{ (\varepsilon)}&=\sum_{j\in\mathbb{Z}}\int_0^{u \varepsilon^{ -1}} \left({\rm D} \Theta(\hat{ u}_{ \Theta(U_{ s})})[ \varpi_{ j}]\right)^{ 2} \lambda_{ j,s}{\rm d} s \\
&=\sum_{j\in\mathbb{Z}}\int_0^{u \varepsilon^{ -1}} \left({\rm D} \Theta(\hat{ u}_{ \Theta(U_{ s})})[ \varpi_{ j}]\right)^{ 2} \left(\lambda_{ j,s} - f \left( \hat{ u}_{ \Theta(U_{ s})}\right)(x_{ j})\right){\rm d} s\\
&\quad + \sum_{j\in\mathbb{Z}}\int_0^{u \varepsilon^{ -1}} \left({\rm D} \Theta(\hat{ u}_{ \Theta(U_{ s})})[ \varpi_{ j}]\right)^{ 2}f \left( \hat{ u}_{ \Theta(U_{ s})}\right)(x_{ j}){\rm d} s := \Gamma_{ 3, u}^{ (\varepsilon)} + \Gamma_{ 4, u}^{ (\varepsilon)}.
\end{align*}
Using similar estimates as for \eqref{aux:CS_D2} and \eqref{aux:Deltaf}, we have that on the event $ \mathcal{ C}_{  \varepsilon, \delta}$, $ \Gamma_{ 3}^{ (\varepsilon)}$ is of order $\change{ \varepsilon^{-1} \varepsilon^{ \frac{ 3}{ 2}- \frac{ \beta}{ 2}}\varepsilon^{ - \frac{ 1}{ 8}-\beta-\delta}\leq \varepsilon^{ \frac{ 3}{ 8}- \frac{ 3\beta}{2}-\delta}} \to 0$, \change{so that, using again Proposition~\ref{prop:concentration_mart}, this term goes to $0$ as $ \varepsilon\to0$, in probability}. Concerning $ \Gamma_{ 4}^{ (\varepsilon)}$, one can further use the approximation $ \varpi_{ j} \approx \varepsilon W_{ x_{ j}}$ (as already done in \eqref{eq:approx_varpi_Wy}) to finally obtain (putting each of these estimates into \eqref{eq:tildeXi123}) the final decomposition 
\begin{equation}
\label{eq:decomp_croXi}
\left[\tilde{ \Xi}^{ (\varepsilon)}\right]_{ u} = \Pi_{ 1, u}^{ (\varepsilon)} + \Pi_{ 2, u}^{ (\varepsilon)},
\end{equation}
where $ \sup_{ u\in [0, t_{ f}]} \left\vert \Pi_{ 1, u}^{ (\varepsilon)} \right\vert \to \change{0}$ as $ \varepsilon\to0$ in probability and 
\begin{align}
\Pi_{ 2, u}^{ (\varepsilon)}&:= \varepsilon^{ 2} \sum_{j\in\mathbb{Z}}\int_0^{u \varepsilon^{ -1}} \left({\rm D} \Theta(\hat{ u}_{ \Theta(U_{ s})})[ W_{ x_{ j}}]\right)^{ 2}f \left( \hat{ u}_{ \Theta(U_{ s})}\right)(x_{ j}){\rm d} s \nonumber \\
&=\varepsilon \sum_{j\in\mathbb{Z}}\int_0^{u} \left({\rm D} \Theta(\hat{ u}_{ \tilde{ \Theta}^{ (\varepsilon)}_{ r}})[ W_{ x_{ j}}]\right)^{ 2}f \left( \hat{ u}_{ \tilde{ \Theta}_{ r}^{ (\varepsilon)}}\right)(x_{ j}){\rm d} r \label{eq:croeps_final}
\end{align}
Note here that, by boundedness of $f$, $\sup_{ \varepsilon>0} \sup_{ u\in[0, t_{ f}]} \left\vert \Pi_{ 2, u}^{ (\varepsilon)} \right\vert< \infty$. Hence, we are now in position to apply Aldous tightness criterion \cite[Th.~16.10]{Billingsley1999}: firstly, from \eqref{eq:Ito_Thetatilde}, \eqref{eq:Thetadrift_to_0}, \eqref{eq:decomp_croXi} and Burkholder-Davis-Gundy inequality, we obtain that 
\begin{equation}
\lim_{ \nu\to \infty} \limsup_{ \varepsilon\to0}\mathbb{ P} \left(\sup_{ u\in [0, t_{ f}]} \left\vert \tilde{ \Theta}_{ u}^{ (\varepsilon)} \right\vert \geq \nu\right)=0.
\end{equation}
Secondly, for any $ \eta>0$, for any stopping time $ \tau$, 
for any $ \rho_{ 1}, \rho_{ 2}>0$, choosing $ \varepsilon\to0$ and $ \eta>0$ sufficiently small, we obtain, using again \eqref{eq:Ito_Thetatilde}, \eqref{eq:Thetadrift_to_0}, \eqref{eq:decomp_croXi}, $ \mathbb{ P} \left(\left\vert  \tilde{ \Theta}_{ \tau+ \eta}^{ (\varepsilon)} - \tilde{ \Theta}_{ \tau}^{ (\varepsilon)} \right\vert \geq \rho_{ 1}\right)\leq \rho_{ 2}$, so that Aldous criterion is verified. Hence, the process $ \tilde{ \Theta}^{ (\varepsilon)}$ is tight on $[0, t_{ f}]$. Denote by $ \tilde{ \Theta}^{ (\infty)}$ any of its limit in law as $ \varepsilon\to 0$. By Skorokhod representation Theorem, one can suppose that this convergence is almost-sure so that passing to the limit as $ \varepsilon\to 0$ into \eqref{eq:decomp_croXi} and \eqref{eq:croeps_final} gives, by definition of $ {\rm D}(\hat{ u}_{ \varphi})$ in \eqref{eq:DTheta},
\begin{align*}
\left[\tilde{ \Xi}^{ (\varepsilon)}\right]_{ u} \xrightarrow[ \varepsilon\to 0]{} &\frac{ 1}{ \left\langle \hat u'_0,\hat u'_0\right\rangle_{m_0}^{ 2}}\int_0^{u} \int_{ \mathbb{ R}}\left( \int_{ \mathbb{ R}} W(x-y)\hat u'_{ \tilde{ \Theta}_{ r}^{ (\infty)}}(y) m_{ \tilde{ \Theta}_{ r}^{ (\infty)}}(y) {\rm d}y\right)^{ 2}f \left( \hat{ u}_{ \tilde{ \Theta}_{ r}^{ (\infty)}}\right)(x) {\rm d}x{\rm d} r\\
&=\frac{ u}{ \left\langle \hat u'_0,\hat u'_0\right\rangle_{m_0}^{ 2}}\int_{ \mathbb{ R}}\left( \int_{ \mathbb{ R}} W(x-y)\hat u'_{ 0}(y) m_{ 0}(y) {\rm d}y\right)^{ 2}f \left( \hat{ u}_{ 0}\right)(x) {\rm d}x,
\end{align*}
where we used the changes of variables $y \to y- \tilde{ \Theta}_{ r}^{ (\infty)}$ and $ x \to x- \tilde{ \Theta}_{ r}^{ (\infty)}$ in the last expression, so that the limit does not actually depend on the realisation of the limiting process $\tilde{ \Theta}^{ (\infty)}$. We conclude the proof of Theorem~\ref{th:main} by Lévy's Theorem for characterization of Brownian motion. \qed

\section{Control of the boundary terms}\label{sec: boundary}

The aim of this Section is to give a proof of Proposition~\ref{prop:D2Theta}, that was useful for the proof of Proposition~\ref{prop:closeness M}. We will make the repeated use of the following parameters, whose definitions are gathered here for reading convenience: for any $ 0<\beta^{ \prime} < \beta$, define $ \nu:= \frac{ \beta - \beta^{ \prime}}{ 6}>0$ as well as 
\begin{equation}
\label{eq:alphas}
\alpha_{ 3}:= \frac{ \beta}{ 2} + \nu,\ \alpha_{ 2}:= \frac{ \beta}{ 2} - 2 \nu \text{ and } \alpha_{ 1}:= \frac{ \beta}{ 2} - 3 \nu= \frac{ \beta^{ \prime}}{ 2}.
\end{equation} 
Note that these parameters have been defined in such a way that (and this is the only property that is crucial for the following, the precise definition of these parameters being mostly irrelevant)
\begin{equation}
\label{eq:alphas_prop}
\beta> \alpha_{ 3} + \alpha_{ 2} \text{ and } \alpha_{ 3}> \alpha_{ 2} > \alpha_{ 1}= \frac{ \beta^{ \prime}}{ 2}>0.
\end{equation}
For any $ \varphi\in \mathbb{ R}$, define the bounded operator $\mathcal{ K}_{ \varphi}$ as
\begin{equation}
\label{eq:operator_K}
\mathcal{ K}_{ \varphi}v = W*(f'(\hat u_\varphi)v), v\in L^{ 2},
\end{equation}
as well as its iterates $ \mathcal{ K}^{ k}_{ \varphi}$ for $k\geq1$. Begin with a technical lemma, which will be useful for the proof of Proposition~\ref{prop:D2Theta}.
\begin{lemma}
\label{lem:K}
Fix $0< \beta^{ \prime}< \beta$. Consider $ \alpha_{ 3}>0$ given by \eqref{eq:alphas_prop} and some integer $k\geq0$. Then for $x \leq \change{\varepsilon N_{ \varepsilon}} - \sigma k \varepsilon^{ - \alpha_{ 3}}$, we have
\begin{equation}
\label{eq:tailcK}
\left\vert \mathcal{ K}_{ \varphi}^{ k} \left(\change{P_{ \varphi}^{ \perp}} \left(w \mathbf{ 1}_{ \mathcal{ D}_{ +}}\right)\right) \right\vert(x) \leq \left(1+ \frac{ 1}{ 2 \sqrt{ 2 \sigma}}\right)^{ k} \left\Vert f^{ \prime} \right\Vert_{ \infty}^{ k} e^{ - \varepsilon^{ - \alpha_{ 3}}} \left\Vert w \right\Vert,
\end{equation}
and for $x \geq - \change{ \varepsilon N_{ \varepsilon}}+ \sigma k \varepsilon^{ - \alpha_{ 3}}$,
\begin{equation*}
\left\vert \mathcal{ K}_{ \varphi}^{ k} \left( \change{P_{  \varphi}^{\perp}} \left(w \mathbf{ 1}_{ \mathcal{ D}_{ -}}\right)\right) \right\vert(x) \leq \left(1+ \frac{ 1}{ 2 \sqrt{ 2 \sigma}}\right)^{ k} \left\Vert f^{ \prime} \right\Vert_{ \infty}^{ k} e^{ - \varepsilon^{ - \alpha_{ 3}}} \left\Vert w \right\Vert.
\end{equation*}
\end{lemma}
\begin{proof}[Proof of Lemma~\ref{lem:K}]
We proceed by recursion on $k$, and only prove the first inequality, the proof of the second being similar. The property is obvious for $k=0$, since for $x\leq \change{ \varepsilon N_{ \varepsilon}}$, $ \mathbf{ 1}_{  \mathcal{ D}_{ +}}(x)=0$.  Suppose now the property \eqref{eq:tailcK} true for some $k$ and consider $x \leq \change{ \varepsilon N_{ \varepsilon}} - \sigma(k+1) \varepsilon^{ - \alpha_{ 3}}$. We have, setting $a:= \change{ \varepsilon N_{ \varepsilon}} - \sigma k \varepsilon^{ - \alpha_{ 3}} $
\begin{align*}
  \left\vert \mathcal{ K}_{ \varphi}^{ k+1} \left( \change{P_{ \varphi}^{ \perp}} \left(w \mathbf{ 1}_{ \mathcal{ D}_{ +}}\right)\right) \right\vert(x)
&=\left\vert \int_{ \mathbb{ R}} W(x-y)f^{ \prime}( \hat{ u}_{ \varphi})(y) \mathcal{ K}_{ \varphi}^{ k}\left( \change{P_{ \varphi}^{ \perp}} \left(w \mathbf{ 1}_{ \mathcal{ D}_{ +}}\right)\right)(y) {\rm d}y \right\vert \\
&\leq \left\Vert f^{ \prime} \right\Vert_{ \infty} \int_{ \mathbb{ R}} W(x-y) \left\vert \mathcal{ K}_{ \varphi}^{ k}\left( \change{P_{ \varphi}^{ \perp}} \left(w \mathbf{ 1}_{ \mathcal{ D}_{ +}}\right)\right)(y) \right\vert {\rm d}y,\\
 &=\left\Vert f^{ \prime} \right\Vert_{ \infty} \int_{-\infty}^{ a} W(x-y) \left\vert \mathcal{ K}_{ \varphi}^{ k}\left( \change{P_{ \varphi}^{ \perp}} \left(w \mathbf{ 1}_{ \mathcal{ D}_{ +}}\right)\right)(y) \right\vert {\rm d}y \\
&\quad + \left\Vert f^{ \prime} \right\Vert_{ \infty} \int_{a}^{ +\infty} W(x-y) \left\vert \mathcal{ K}_{ \varphi}^{ k}\left( \change{P_{ \varphi}^{ \perp}} \left(w \mathbf{ 1}_{ \mathcal{ D}_{ +}}\right)\right)(y) \right\vert {\rm d}y.
\end{align*}
  By recursion hypothesis on the first previous integral, we obtain
  \begin{align*}
  \left\Vert f^{ \prime} \right\Vert_{ \infty} \int_{-\infty}^{ a} W(x-y) \left\vert \mathcal{ K}_{ \varphi}^{ k}\left( \change{P_{ \varphi}^{ \perp}} \left(w \mathbf{ 1}_{ \mathcal{ D}_{ +}}\right)\right)(y) \right\vert {\rm d}y \leq  \left(1+ \frac{ 1}{ 2 \sqrt{ 2 \sigma}}\right)^{ k} \left\Vert f^{ \prime} \right\Vert_{ \infty}^{ k+1}e^{ - \varepsilon^{ - \alpha_{ 3}}} \left\Vert w \right\Vert,
  \end{align*}
  since $\int_{-\infty}^{ a} W(x-y) {\rm d}y \leq \int_{ \mathbb{ R}} W=1$. Concerning the second integral, we have, by Cauchy-Schwarz inequality,
  \begin{align*}
   \left\Vert f^{ \prime} \right\Vert_{ \infty} \int_{a}^{ +\infty} W(x-y) \left\vert \mathcal{ K}_{ \varphi}^{ k}\left( \change{P_{ \varphi}^{ \perp}} \left(w \mathbf{ 1}_{ \mathcal{ D}_{ +}}\right)\right)(y) \right\vert {\rm d}y&\leq  \left\Vert f^{ \prime} \right\Vert_{ \infty} \left( \int_{ a}^{+\infty} W(x-y)^{ 2} {\rm d}y\right)^{ \frac{ 1}{ 2}} \left\Vert \mathcal{ K}_{ \varphi}^{ k}\left( \change{P_{ \varphi}^{ \perp}} \left(w \mathbf{ 1}_{ \mathcal{ D}_{ +}}\right)\right) \right\Vert
  \end{align*}
Recall that $x \leq \change{ \varepsilon N_{ \varepsilon}} - \sigma(k+1) \varepsilon^{ - \alpha_{ 3}}= a- \sigma \varepsilon^{ - \alpha_{ 3}}$ so that $\left( \int_{ a}^{+\infty} W(x-y)^{ 2} {\rm d}y\right)^{ \frac{ 1}{ 2}} = \frac{ 1}{ 2 \sqrt{ 2 \sigma}} e^{ - \frac{ a-x}{ \sigma}} \leq \frac{ 1}{ 2 \sqrt{ 2 \sigma}} e^{ - \varepsilon^{ - \alpha_{ 3}}}$. Moreover, by an easy recursion and application of Young's inequality for convolution, we obtain that, \change{since $ \left\Vert P_{ \varphi}^{ \perp} \right\Vert\leq 1$ and $ \left\Vert W \right\Vert_{ 1}=1$,}
\begin{align*}
\left\Vert \mathcal{ K}_{ \varphi}^{ k}\left( \change{P_{ \varphi}^{ \perp}} \left(w \mathbf{ 1}_{ \mathcal{ D}_{ +}}\right)\right) \right\Vert \leq \left\Vert W \right\Vert_{ 1}^{ k} \left\Vert f^{ \prime} \right\Vert_{ \infty}^{ k} \left\Vert \change{P_{ \varphi}^{ \perp} \left(w \mathbf{ 1}_{ \mathcal{ D}_{ +}}\right)} \right\Vert \change{\leq \left\Vert f^{ \prime} \right\Vert_{ \infty}^{ k} \left\Vert w \right\Vert}. 
\end{align*} 
Putting the two estimates together gives the result for $k+1$, so that \eqref{eq:tailcK} is true by recursion.
\end{proof}

We have now the ingredients to prove Proposition~\ref{prop:D2Theta}.

\begin{proof}[Proof of Proposition~\ref{prop:D2Theta}]
We only prove the first the bounds on $\mathcal{D}_+$ for simplicity, the bounds on $\mathcal{D}_-$ being similar. Concerning \eqref{eq:DThetaD+-} we have, for any $\varphi$ satisfying $|\varphi|\leq c \varepsilon^{-\beta}$ with $c\in(0,1)$, recalling \eqref{eq:uprime_exp},
\begin{align*}
\left|{\rm D}\Theta(\hat u_\varphi)[v\mathbf{1}_{\mathcal{D}_+}]\right|
& = \left|\int_{\varepsilon N_\varepsilon+\frac{\varepsilon}{2}}^\infty\frac{v \,\hat u'_\varphi}{\langle \hat u'_\varphi,\hat u'_\varphi\rangle_{m_\varphi}}m_\varphi\right|\\
&\leq C \Vert v\Vert_{L^2} \change{ \left(\int_{\varepsilon N_\varepsilon+\frac{\varepsilon}{2}-\varphi}^\infty (\hat u'_0)^2\right)^{ \frac{ 1}{ 2}}}\leq C e^{-\mu (1-c)\varepsilon^{-\beta}}\Vert v\Vert_{L^2},
\end{align*}
which implies \eqref{eq:DThetaD+-} for $\varepsilon$ small enough.

Let us now focus on the proof of \eqref{eq:D2ThetaD+-}.
Recall the definitions of the parameters $ \alpha_{ 1}, \alpha_{ 2}$ and $ \alpha_{ 3}$ in \eqref{eq:alphas} and the expression of ${\rm D}^{ 2} \Theta\left( \hat{ u}_{ \varphi}\right)$ in \eqref{eq:D2Theta}. Recall also the definition of the projections $P_{ \varphi}$ and $P_{ \varphi}^{ \perp}$ in \eqref{eq:projectionsL2} and write $ w \mathbf{ 1}_{ \mathcal{ D}_{ +}}= P_{\varphi}(w \mathbf{ 1}_{ \mathcal{ D}_{ +}}) + P_{ \varphi}^{ \perp}(w \mathbf{ 1}_{ \mathcal{ D}_{ +}})$ and $ v= P_{\varphi}(v) + P_{ \varphi}^{ \perp}(v)$. Relying on the identity $ {\rm D}^{ 2}\Theta(u_{ \varphi}) \left[ \hat{ u}_{ \varphi}^{ \prime},  \hat{ u}_{ \varphi}^{ \prime}\right]=0$ (which can easily be deduced from \eqref{eq:D2Theta}, as $P_{ \varphi}^{ \perp} \hat{ u}^{ \change{ \prime}}_{ \varphi}=0$) we obtain, for some constant $C>0$ that does not depend on $ \varphi$,
 \begin{align}
  \left\vert {\rm D}^{ 2}\Theta(u_{ \varphi}) \left[ v, w \mathbf{ 1}_{  \mathcal{ D}_{ +}}\right] \right\vert &\leq C\left\vert  \int_{ 0}^{+\infty}  \int_{ \mathbb{ R}} f^{ \prime\prime}( \hat{ u}_{ \varphi}) \left(\hat{ u}^{ \prime}_{ \varphi}\right)^{ 2} \left(e^{ s \mathcal{ L}_{ \varphi}} P_{ \varphi}^{ \perp}v\right)  {\rm d}x {\rm d}s\right\vert \left\vert \frac{ \left\langle w \mathbf{ 1}_{ \mathcal{ D}_{ +}}\, ,\, \hat{ u}_{ \varphi}^{ \prime}\right\rangle_{ \change{m_{ \varphi}}}}{ \left\langle \hat{ u}_{ \varphi}^{ \prime}\, ,\, \hat{ u}_{ \varphi}^{ \prime}\right\rangle_{ \change{m_{ \varphi}}}} \right\vert \nonumber\\
  &\quad +C\left\vert  \int_{ 0}^{+\infty}  \int_{ \mathbb{ R}} f^{ \prime\prime}( \hat{ u}_{ \varphi}) \hat{ u}^{ \prime}_{ \varphi} \left(e^{ s \mathcal{ L}_{ \varphi}} v\right) \left(e^{ s \mathcal{ L}_{ \varphi}} \left(P_{ \varphi}^{ \perp} \left(w \mathbf{ 1}_{  \mathcal{ D}_{ +}}\right)\right)\right) {\rm d}x {\rm d}s\right\vert \nonumber\\
   &= (A)+(B). \label{aux:D2ThetaAB}
  \end{align}
Recall now that $ f^{\prime\prime}$ bounded, $ \hat{ u}_{ \varphi}^{ \prime}\in L^{ 4}$ and that by \eqref{eq:contract_ecL} $ \left\Vert e^{ s \mathcal{ L}_{ \varphi}} P_{ \varphi}^{ \perp}v \right\Vert_{ L^{ 2}} \leq e^{ - \kappa s} \left\Vert v \right\Vert_{ L^{ 2}}$. Hence, by Cauchy-Schwarz inequality, using the exponential tail on $ \hat{ u}_{ 0}^{ \prime}$ in \eqref{eq:uprime_exp}, if $ \varphi \leq c\varepsilon^{ - \beta}$,
  \begin{equation}
  \label{aux:Abound}
  (A) \leq C \left\Vert v \right\Vert\left\vert \left\langle w \mathbf{ 1}_{ \mathcal{ D}_{ +}}\, ,\, \hat{ u}_{ \varphi}^{ \prime}\right\rangle \right\vert \leq C \left\Vert v \right\Vert\left\Vert w \right\Vert e^{ - \mu \left( \varepsilon^{ - \beta} - \varphi\right)} \leq C \left\Vert v \right\Vert\left\Vert w \right\Vert e^{ - \mu(1-c)\varepsilon^{ - \beta}},
  \end{equation}
  where we recall that $c\in (0,1)$. Concerning $(B)$ in \eqref{aux:D2ThetaAB},
  \begin{align}
  (B)&\leq C\left\vert  \int_{ 0}^{ \varepsilon^{ - \alpha_{ 1}}}  \int_{ \mathbb{ R}} f^{ \prime\prime}( \hat{ u}_{ \varphi}) \hat{ u}^{ \prime}_{ \varphi} \left(e^{ s \mathcal{ L}_{ \varphi}} v\right) \left(e^{ s \mathcal{ L}_{ \varphi}} \left(P_{ \varphi}^{ \perp} \left(w \mathbf{ 1}_{  \mathcal{ D}_{ +}}\right)\right)\right) {\rm d}x {\rm d}s\right\vert \nonumber\\
  &\quad + C\left\vert  \int_{ \varepsilon^{ - \alpha_{ 1}}}^{+\infty}  \int_{ \mathbb{ R}} f^{ \prime\prime}( \hat{ u}_{ \varphi}) \hat{ u}^{ \prime}_{ \varphi} \left(e^{ s \mathcal{ L}_{ \varphi}} v\right) \left(e^{ s \mathcal{ L}_{ \varphi}} \left(P_{ \varphi}^{ \perp} \left(w \mathbf{ 1}_{  \mathcal{ D}_{ +}}\right)\right)\right) {\rm d}x {\rm d}s\right\vert \nonumber\\
  &:= (B_{ 1})+ (B_{ 2}). \label{eq:B12s}
  \end{align}
The treatment of the above term $(B_{ 2})$ is easy: using again that $ f^{ \prime\prime}$ and $ \hat{ u}^{ \prime}_{ \varphi}$ are bounded and that $ e^{ s \mathcal{ L}_{ \phi}}$ is dissipative, by Cauchy-Schwarz inequality,
  \begin{equation}
  \label{aux:B2bound}
  (B_{ 2}) \leq C \left\Vert v \right\Vert \left\Vert w \right\Vert \int_{ \varepsilon^{ - \alpha_{ 1}}}^{+\infty} e^{ - \kappa s} {\rm d}s= \frac{ C}{ \kappa} e^{ - \kappa \varepsilon^{ -\alpha_{ 1}}}  \left\Vert v \right\Vert \left\Vert w \right\Vert.
  \end{equation}
The control of $(B_{ 1})$ in \eqref{eq:B12s} requires some more substantial work and is the subject of the remainder of the proof: first, by Cauchy-Schwarz inequality,
  \begin{align}
  \label{aux:B1}
  (B_{ 1})& \leq C \left\Vert f^{ \prime\prime} \right\Vert_{ \infty} \left\Vert v \right\Vert\int_{ 0}^{ \varepsilon^{ - \alpha_{ 1}}} \left( \int_{ \mathbb{ R}} \left(\hat{ u}_{ \varphi}^{ \prime}(x) \left(e^{ s \mathcal{ L}_{ \varphi}} \left(P_{ \varphi}^{\perp} \left(w \mathbf{ 1}_{ \mathcal{ D}_{ +}}\right)\right)\right)(x)\right)^{ 2}{\rm d}x\right)^{ \frac{ 1}{ 2}} {\rm d}s.
  \end{align}
Focus first on the integral within the righthand side of the previous inequality \eqref{aux:B1}: write, for $c'\in (c,1)$,
  \begin{align}
 \int_{ \mathbb{ R}} \bigg(\hat{ u}_{ \varphi}^{ \prime}(x) \left(e^{ s \mathcal{ L}_{ \varphi}} \left(P_{ \varphi}^{\perp} \left(w \mathbf{ 1}_{ \mathcal{ D}_{ +}}\right)\right)\right)&(x)\bigg)^{ 2}{\rm d}x \nonumber\\
 &= \int_{ - \infty}^{ c'\varepsilon^{ - \beta}}\left(\hat{ u}_{ \varphi}^{ \prime}(x) \left(e^{ s \mathcal{ L}_{ \varphi}} \left(P_{ \varphi}^{\perp} \left(w \mathbf{ 1}_{ \mathcal{ D}_{ +}}\right)\right)\right)(x)\right)^{ 2}{\rm d}x \label{aux:Cuprime}\\
  &\quad + \int_{ c'\varepsilon^{ - \beta}}^{ +\infty} \left(\hat{ u}_{ \varphi}^{ \prime}(x) \left(e^{ s \mathcal{ L}_{ \varphi}} \left(P_{ \varphi}^{\perp} \left(w \mathbf{ 1}_{ \mathcal{ D}_{ +}}\right)\right)\right)(x)\right)^{ 2}{\rm d}x \label{aux:Duprime}\\ &:= (C_{ s}) + (D_{ s}) \nonumber,
  \end{align}
  so that $(B_{ 1})$ in \eqref{aux:B1} can be estimated as
  \begin{equation}
  \label{eq:B1vsCD}
  (B_{ 1}) \leq C \left\Vert v \right\Vert \int_{ 0}^{ \varepsilon^{ - \alpha_{ 1}}} \left( \sqrt{(C_{ s})} + \sqrt{ (D_{ s})}\right){\rm d}s.
  \end{equation}
  Concerning $(D_{ s})$ in \eqref{aux:Duprime}, we have, using \eqref{eq:uprime_exp}, for $ \left\vert \varphi \right\vert\leq c \varepsilon^{ - \beta}$,
  \begin{equation}
  \label{eq:Duprime_bound}
(D_{ s}) \leq C e^{ -  2 \mu (c'-c)\varepsilon^{ - \beta}} e^{ - 2 \kappa s}\left\Vert w \right\Vert^{ 2}.
  \end{equation}
 Now turn to $ (C_{ s})$ in \eqref{aux:Cuprime}: concentrate first on the term $e^{ s \mathcal{ L}_{ \varphi}}  \left(P_{ \varphi}^{\perp} \left(w \mathbf{ 1}_{ \mathcal{ D}_{ +}}\right)\right)$. Recall the definition of the integral operator $ \mathcal{ K}_{ \varphi}$ in \eqref{eq:operator_K} 
so that $ \mathcal{ L}_{ \varphi}= -I + \mathcal{ K}_{  \varphi}$. Hence, fixing some $M\geq1$ (to be specified later), we have
  \begin{align*}
  e^{ s \mathcal{ L}_{ \varphi}}  \left(P_{ \varphi}^{\perp} \left(w \mathbf{ 1}_{ \mathcal{ D}_{ +}}\right)\right) &= e^{ -s} \sum_{ k=0}^{ M} \frac{ s^{ k} \left(\mathcal{ K}_{ \varphi}\right)^{ k}  \left(P_{ \varphi}^{\perp} \left(w \mathbf{ 1}_{ \mathcal{ D}_{ +}}\right)\right)}{ k!} \\
  &\quad + e^{ -s} \sum_{ k=M+1}^{ +\infty} \frac{ s^{ k} \left(\mathcal{ K}_{ \varphi}\right)^{ k}  \left(P_{ \varphi}^{\perp} \left(w \mathbf{ 1}_{ \mathcal{ D}_{ +}}\right)\right)}{ k!} .
  \end{align*}
  Inserting this last sum into \eqref{aux:Cuprime}, using that $(a+b)^{ 2}\leq 2a^{ 2}+ 2b^{ 2}$, it suffices to estimate separately the two next terms $(C_{ 1, s})$ and $(C_{ 2, s})$ that are given in \eqref{aux:C1} and \eqref{aux:C2} below. Firstly,
  \begin{align}
  \label{aux:C1}
 (C_{ 1, s}):=\int_{ - \infty}^{ c'\varepsilon^{ - \beta}}\left(\hat{ u}_{ \varphi}^{ \prime}(x)  e^{ -s} \sum_{ k=0}^{ M} \frac{ s^{ k} \left(\mathcal{ K}_{ \varphi}\right)^{ k}  \left(P_{ \varphi}^{\perp} \left(w \mathbf{ 1}_{ \mathcal{ D}_{ +}}\right)\right)}{ k!}(x)\right)^{ 2}{\rm d}x.
  \end{align}
  Using that, by assumption, $s\leq \varepsilon^{ -\alpha_{ 1}}$,
  \begin{equation}
  \label{aux:C12}
 (C_{ 1, s})\leq e^{ -2 s}\int_{ - \infty}^{ c'\varepsilon^{ - \beta}}\hat{ u}_{ \varphi}^{ \prime}(x)^{ 2} \left(\sum_{ k=0}^{ M} \frac{ \varepsilon^{ -k \alpha_{ 1}} \left(\mathcal{ K}_{ \varphi}\right)^{ k}  \left(P_{ \varphi}^{\perp} \left(w \mathbf{ 1}_{ \mathcal{ D}_{ +}}\right)\right)}{ k!}(x)\right)^{ 2}{\rm d}x.
  \end{equation}
This is the point where we want to apply Lemma~\ref{lem:K}. To do so, let us first make an intermediate calculation: a sufficient condition for the inequality $c^{ \prime} \varepsilon^{ - \beta} \leq \change{ \varepsilon N_{ \varepsilon}} - \sigma k \varepsilon^{ - \alpha_{ 3}}$ to hold for any $k=0, \ldots, M$ is to require that $M \leq \frac{1- c^{ \prime}}{ \sigma} \varepsilon^{ - \beta + \alpha_{ 3}} \change{- \frac{ \varepsilon^{ 1+ \alpha_{ 3}}}{ \sigma}}$ \change{(recall that $ \varepsilon N_{ \varepsilon}> \varepsilon^{ - \beta} - \varepsilon$)}. If we choose now $ M:= \left\lfloor \varepsilon^{ - \alpha_{ 2}}\right\rfloor$, this boils down to verifying that $ \varepsilon^{ \beta} \leq \frac{ 1-c^{ \prime}}{ \sigma} \varepsilon^{ \alpha_{ 2}+ \alpha_{ 3}} \change{- \frac{ \varepsilon^{ \alpha_{ 2} + \alpha_{ 3} +1}}{ \sigma}}$, which is true since we have required that $ \beta> \alpha_{ 3}+ \alpha_{ 2}$ (recall \eqref{eq:alphas_prop}). Hence, we are in position to use the estimate of Lemma~\ref{lem:K} for all $k=0, \ldots, M$ into the integral \eqref{aux:C12}: bounding $\int_{ - \infty}^{ c^{ \prime}\varepsilon^{ - \beta}}\hat{ u}_{ \varphi}^{ \prime}(x)^{ 2} {\rm d}x\leq \int_{ \mathbb{ R}}\hat{ u}_{ \varphi}^{ \prime}(x)^{ 2} {\rm d}x=C$, we obtain
  \begin{align}
  \label{aux:C1bound}
 (C_{ 1, s})&
\leq C e^{ -2 s}\left((M+1) \varepsilon^{ -M \alpha_{ 1}}  \left(1+ \frac{ 1}{ 2 \sqrt{ 2 \sigma}}\right)^{ M} \left\Vert f^{ \prime} \right\Vert_{ \infty}^{ M} e^{ - \varepsilon^{ - \alpha_{ 3}}} \left\Vert w \right\Vert\right)^{ 2}\nonumber\\
& \leq C e^{ -2 s}e^{ - \varepsilon^{ - \alpha_{ 3}}} \left\Vert w \right\Vert^{ 2}.
  \end{align}
Note that the last bound in the previous \eqref{aux:C1bound} comes from the fact that $M= \left\lfloor \varepsilon^{ - \alpha_{ 2}}\right\rfloor$ and $ \alpha_{ 3}> \alpha_{ 2}$, again by \eqref{eq:alphas_prop} (and therefore, one $ e^{ - \varepsilon^{ - \alpha_{ 3}}}$ is sufficient to control the rest of the expression, which gives \eqref{aux:C1bound}). In order to conclude with the estimation of \eqref{eq:B1vsCD}, it remains now to control
\begin{align}
(C_{ 2, s})&:= \int_{ - \infty}^{ c^{ \prime}\varepsilon^{ - \beta}}\left(\hat{ u}_{ \varphi}^{ \prime}(x)  e^{ -s} \sum_{ k=M+1}^{ +\infty} \frac{ s^{ k} \left(\mathcal{ K}_{ \varphi}\right)^{ k}  \left(P_{ \varphi}^{\perp} \left(w \mathbf{ 1}_{ \mathcal{ D}_{ +}}\right)\right)}{ k!} (x)\right)^{ 2}{\rm d}x \label{aux:C2}\\
&\leq \left( \sum_{ k=M+1}^{ +\infty} e^{ -s} \frac{ s^{ k}}{ k!}\right) \sum_{ k=M+1}^{ +\infty}\int_{ - \infty}^{ c^{ \prime}\varepsilon^{ - \beta}}\hat{ u}_{ \varphi}^{ \prime}(x)^{ 2}  e^{ -s}  \frac{ s^{ k} \left(\mathcal{ K}_{ \varphi}\right)^{ k}  \left(P_{ \varphi}^{\perp} \left(w \mathbf{ 1}_{ \mathcal{ D}_{ +}}\right)\right)^{ 2}}{ k!} (x){\rm d}x,\nonumber
\end{align}
by Cauchy-Schwarz inequality. Since $ \hat{u}_{ \varphi}^{ \prime}$ is bounded, a rough bound gives
\begin{align*}
(C_{ 2, s})&\leq C \left( \sum_{ k=M+1}^{ +\infty} e^{ -s} \frac{ s^{ k}}{ k!}\right) \sum_{ k=M+1}^{ +\infty}  \frac{e^{ -s} s^{ k}}{ k!} \left\Vert \left(\mathcal{ K}_{ \varphi}\right)^{ k}  \left(P_{ \varphi}^{\perp} \left(w \mathbf{ 1}_{ \mathcal{ D}_{ +}}\right)\right) \right\Vert^{ 2}\\
&\leq C \left( \sum_{ k=M+1}^{ +\infty} e^{ -s} \frac{ s^{ k}}{ k!}\right) \left(\sum_{ k=M+1}^{ +\infty}  \frac{e^{ -s} s^{ k}}{ k!} \left\Vert f^{ \prime} \right\Vert_{ \infty}^{ 2k}\right) \left\Vert w \right\Vert^{ 2}.
\end{align*}
Hence, using the bound $ \sum_{ k=M+1}^{ +\infty} e^{ -y} \frac{ y^{ k}}{ k!}\leq \frac{ y^{ M+1}}{ (M+1)M!}$, 
\begin{align*}
(C_{ 2, s})&\leq C \left\Vert f^{ \prime} \right\Vert_{ \infty}^{ 2(M+1)} \left( \frac{ s^{ M+1}}{ (M+1)M!}\right)^{ 2} e^{ \left( \left\Vert f^{ \prime} \right\Vert_{ \infty}^{ 2}-1\right)s} \left\Vert w \right\Vert^{ 2},
\end{align*}
so that
\begin{align*}
\int_{ 0}^{ \varepsilon^{ - \alpha_{ 1}}}\sqrt{ (C_{ 2, s})} {\rm d}s&\leq C \left\Vert w \right\Vert \frac{ \left\Vert f^{ \prime} \right\Vert_{ \infty}^{ M+1} }{ (M+1) M!}\int_{ 0}^{ \varepsilon^{ - \alpha_{ 1}}}s^{ M+1} e^{ \left( \left\Vert f^{ \prime} \right\Vert_{ \infty}^{ 2}-1\right) \frac{ s}{ 2}} {\rm d}s\\
&\leq C \left\Vert w \right\Vert \frac{ \left\Vert f^{ \prime} \right\Vert_{ \infty}^{ M+1} }{ (M+1) M!} e^{ \frac{ \left( \left\Vert f^{ \prime} \right\Vert_{ \infty}^{ 2}-1\right)}{ 2} \varepsilon^{ - \alpha_{ 1}}} \frac{ \varepsilon^{ - \alpha_{ 1}(M+2)}}{ M+2}.
\end{align*}
Using Stirling formula and recalling that $M= \left\lfloor \varepsilon^{ - \alpha_{ 2}}\right\rfloor$,
\begin{align*}
\int_{ 0}^{ \varepsilon^{ - \alpha_{ 1}}}\sqrt{ (C_{ 2, s})} {\rm d}s&\leq C \left\Vert w \right\Vert \frac{ \left\Vert f^{ \prime} \right\Vert_{ \infty}^{ M+1} e^{ M}}{ M^{ \frac{ 5}{ 2} +M}} e^{ \frac{ \left( \left\Vert f^{ \prime} \right\Vert_{ \infty}^{ 2}-1\right)}{ 2} \varepsilon^{ - \alpha_{ 1}}} \varepsilon^{ - \alpha_{ 1}(M+2)}\\
&\leq C \left\Vert w \right\Vert \exp \bigg( \varepsilon^{ - \alpha_{ 2}} \left\lbrace - \left(\alpha_{ 2} - \alpha_{ 1}\right) \ln \left(1/ \varepsilon\right)+ 1+ \ln \left\Vert f^{ \prime} \right\Vert_{ \infty}\right\rbrace  \\
&\qquad\qquad\qquad \quad + \frac{ 1}{ 2} \left( \left\Vert f^{ \prime} \right\Vert_{ \infty}^{ \change{ 2}} -1\right) \varepsilon^{ - \alpha_{ 1}} +\left(2\alpha_{ 1} -  \frac{ 5}{ 2} \alpha_{ 2}\right) \ln \left(1/ \varepsilon\right)\bigg).
\end{align*}
Since we have supposed that $ \alpha_{ 2}> \alpha_{ 1}$,
\begin{equation}
\label{aux:C2bound}
\int_{ 0}^{ \varepsilon^{ - \alpha_{ 1}}}\sqrt{( C_{ 2, s})} {\rm d}s \leq C \left\Vert w \right\Vert \exp \left( \varepsilon^{ - \alpha_{ 2}} \left\lbrace - \left(\alpha_{ 2} - \alpha_{ 1}\right) \ln \left(1/ \varepsilon\right)+ O(1)\right\rbrace  \right) \leq C \left\Vert w \right\Vert e^{ - \varepsilon^{ - \alpha_{ 2}}},
\end{equation}
for sufficiently small $ \varepsilon>0$. Gather first the bounds \eqref{aux:C1bound}, \eqref{aux:C2bound} and \eqref{eq:Duprime_bound} on $(C_{ 1})$, $(C_{ 2})$ and $(D)$ into the bound \eqref{eq:B1vsCD} for $(B_{ 1})$ and collect the estimates \eqref{aux:Abound}, \eqref{aux:B2bound} to obtain the result.
\end{proof}

\section{ \change{Conclusion}}
\label{sec:extensions}
A natural question concerns the possibility to extend the present results to the case where the traveling wave has nontrivial speed $c\neq 0$. The expected result is that on a time scale of order $ \varepsilon^{ -1}$, the dephasing for the voltage profile $U_{ u \varepsilon^{ -1}}^{ (\varepsilon)}$ should write as $ C_{ u}^{ (\varepsilon)}= c u \varepsilon^{ -1} + v_{ u}^{ (\varepsilon)} + \varrho \psi_{ u}^{ (\varepsilon)}$, where $c$ is the deterministic wave speed and $ \psi^{ (\varepsilon)}$ converges again to a Brownian motion. As far as the case $c\neq 0$ is concerned, start with the obvious: keeping the same Hawkes dynamics given by \eqref{eq:Ui_mod} within the \emph{fixed} bulk $ \mathcal{ D}_{ 0}^{ (\varepsilon)}=\left[  -\varepsilon N_\varepsilon- \frac{ \varepsilon}{ 2}, \varepsilon N_\varepsilon+ \frac{ \varepsilon}{ 2} \right)$ is no longer appropriate. When $c\neq 0$, the traveling wave (and hence $U^{ (\varepsilon)}$ itself) would reach the boundary of the bulk in a time of order $t \approx \frac{ \ell_{ \varepsilon}}{ c} =O \left(\varepsilon^{ - \beta} \right) \ll \varepsilon^{ -1}$, so that the convergence result of \eqref{eq:conv_main} cannot remain true. One needs to define a \emph{moving bulk} $ \mathcal{ D}_{ 0, t}^{ (\varepsilon)}=\left[  -\varepsilon N_\varepsilon- \frac{ \varepsilon}{ 2} + ct, \varepsilon N_\varepsilon+ \frac{ \varepsilon}{ 2} + ct\right)$ in order to see some noise-induced nontrivial correction to the deterministic speed $c$.

One main ingredient of the present analysis is the linear stability result of $ \mathcal{ M}$ proven by Lang and Stannat in \cite{2016JDE...261.4275L}, see Theorem~\ref{th:stability} below. The stability result is actually written in \cite{2016JDE...261.4275L} as a spectral gap inequality for the \emph{frozen wave operator} $L^{ \#}v:= - v + c \partial_{ x} v + W \ast \left(f^{ \prime} \left(\hat{ u}\right)v\right)$ which arises as the linearisation around the traveling wave solution once the moving frame change of variables $(t,x) \mapsto (t, x+ct)$ has been carried out. This spectral gap is expressed in some weighted $L^{ 2}$-space $L^{ 2}(\rho)$ that is, unless $c=0$, \emph{not} equivalent to the flat space $L^{ 2}$. Note that when $c=0$, $L^{ \#}$ coincides with the linearized operator $ \mathcal{ L}_{ \varphi}$ in \eqref{eq:Lvarphi} below. Performing such a moving frame change of variables does not seem to be a good idea at the level of the particle system \eqref{eq:profileU}: the domain of $L^{ \#}$ is $H^{ 1}( \mathbb{ R})$ whereas the microscopic voltage profile $U^{ (\varepsilon)}$ is not even continuous. A strategy here would consist in keeping the original space-time frame $(x,t)$ and take advantage of another stability result of Lang and Stannat, see \cite[Th.~8]{2016JDE...261.4275L}: the spectral gap inequality \eqref{eq:spectralgap_LangStannat} remains true at least in a regime where $c$ is small. The hope is then to extend the present results in such a regime of small wave speed. However, it appears from calculations similar to the linearisation arguments of Proposition~\ref{prop:mart_Y} below that the spectral gap estimate is no longer sufficient for the long-time analysis of $ U^{ (\varepsilon)}$ along $ \mathcal{ M}$ in the case $c\neq 0$. A strategy would be to obtain finer controls on the dynamics around $ \mathcal{ M}$ (see \cite{2021arXiv210702468L,Bates2008} for similar arguments). This will be the object of future works.

\bigskip

\appendix
\renewcommand{\thesection}{\Alph{section}}
\section{Auxiliary lemmas}
\label{sec:app_auxiliary}

\subsection{The truncated operator \texorpdfstring{$L_\varphi$}{Lvarphi}}
\label{app:L0}

The aim of this subsection is to obtained the bounds on $L_\varphi$ given in Proposition~\ref{prop:L0}.

\begin{proof}[Proof of Proposition~\ref{prop:L0}]
For $v,w\in L^{ 2} \left(\mathcal{ D}_{ 0}\right)$, 
\begin{equation}
\label{eq:DirFormsEqual}
\left\langle L_{ \varphi}v\, ,\, w\right\rangle_{ m_{ \varphi}, \mathcal{ D}_{ 0}}= \left\langle \mathcal{ L}_{ \varphi} v\, ,\, w\right\rangle_{ m_{ \varphi}}.
\end{equation}
Indeed, since $v= v \mathbf{ 1}_{ \mathcal{ D}_{ 0}}$ and $w= w \mathbf{ 1}_{ \mathcal{ D}_{ 0}}$
\begin{align*}
\left\langle L_{ \varphi}v\, ,\, w\right\rangle_{ m_{ \varphi}, \mathcal{ D}_{ 0}}&= \int_{ \mathcal{ D}_{ 0}} \mathcal{L}_\phi \left(v \mathbf{ 1}_{ \mathcal{ D}_{ 0}}\right)(x) \mathbf{ 1}_{ \mathcal{ D}_{ 0}}(x) w(x) m_{ \varphi}(x) {\rm d}x,\\
&=\int_{ \mathbb{ R}} \mathcal{L}_\phi \left(v \right)(x) w(x) m_{ \varphi}(x) {\rm d}x = \left\langle \mathcal{ L}_{ \varphi} v\, ,\, w\right\rangle_{ m_{ \varphi}}.
\end{align*}
In particular, using the fact that $\mathcal{L}_\phi$ is selfadjoint in $L^{ 2}_{ m_{ \varphi}}$, we obtain
\begin{align*}
\left\langle L_{\phi}v, w\right\rangle_{ m_\phi, \mathcal{ D}_{ 0}}&= \left\langle \mathcal{ L}_{ \varphi} v\, ,\, w\right\rangle_{ m_{ \varphi}} = \left\langle  v\, ,\,  \mathcal{ L}_{ \varphi} w\right\rangle_{ m_{ \varphi}}= \left\langle v,  L_{\phi}w\right\rangle_{ m_\phi, \mathcal{ D}_{ 0}}
\end{align*}
so $L_{ \varphi}$ is also selfadjoint in $L^{ 2}_{ m_{ \varphi}} \left(\mathcal{ D}_{ 0}\right)$. Note here that $ \mathcal{ L}_{ \varphi}: L^{ 2}_{ m_{ \varphi}}\to L^{ 2}_{ m_{ \varphi}}$ is dissipative: for any $v\in L^{ 2}_{ m_{ \varphi}}$, $ \left\langle \mathcal{ L}_{ \varphi} v\, ,\, v\right\rangle_{ m_{ \varphi}}\leq 0$. Indeed,  since $ \mathcal{ L}_{ \varphi}\hat{ u}^{ \prime}_{ \varphi}=0$ and $ \mathcal{ L}_{ \varphi}$ is selfadjoint, $ \left\langle \mathcal{ L}_{ \varphi}v\, ,\, v\right\rangle_{ m_{ \varphi}}= \left\langle \mathcal{ L}_{ \varphi}P_{ \varphi}^{ \perp} v\, ,\, P_{ \varphi}^{ \perp} v\right\rangle_{ m_{ \varphi}} \leq 0$, by \eqref{eq:spectralgap_LangStannat}. But then, using \eqref{eq:DirFormsEqual}, we have also $\left\langle L_{ \varphi}v\, ,\, v\right\rangle_{ m_{ \varphi}, \mathcal{ D}_{ 0}}\leq 0$ so that $ L_{ \varphi}$ is also dissipative on $L^{ 2}_{ m_{ \varphi}} \left(\mathcal{ D}_{ 0}\right)$. The bound \eqref{eq:bound_eLphi} follows then from standard arguments (see e.g. \cite[Th.~4.3]{Pazy1983}).
Concerning the spectral gap inequality \eqref{eq:SG_L0}, we have, using \eqref{eq:DirFormsEqual} and \eqref{eq:spectralgap_LangStannat}, for $ \change{v= \tilde{ v} \mathbf{ 1}_{ \mathcal{ D}_{ 0}}}\in L^{ 2} \left(\mathcal{ D}_{ 0}\right)$ \change{with $ \tilde{ v}\in L^{ 2}$ such that $ P_{ \varphi} \tilde{v}=0$},
\begin{align*}
\left\langle L_{\phi}v , v\right\rangle_{m_\phi,\mathcal{D}_0} &= \left\langle \mathcal{ L}_{ \varphi} v\, ,\, v\right\rangle_{ m_{ \varphi}} \leq -\kappa\left(\left\Vert v \right\Vert_{ m_\phi, \mathcal{ D}_{ 0}}^{ 2}- \left\langle v, \hat{ u}'_\phi\right\rangle_{ m_\phi}^{2}\right),
\end{align*}
which gives the result, since $v=v \mathbf{ 1}_{ \mathcal{ D}_{ 0}}$. For the last inequality, we have
\begin{align*}
\frac{\rm d}{{\rm d} t}\left\Vert e^{tL_\phi} v\right\Vert_{m_\phi,\mathcal{D}_0}^2
&\leq  -2\kappa\left\Vert e^{tL_\phi}v \right\Vert_{m_\phi, \mathcal{ D}_{ 0}}^{ 2} +2\kappa\left\langle e^{tL_\phi}v , \mathbf{ 1}_{ \mathcal{ D}_{ 0}} \hat{ u}'_\phi\right\rangle_{m_\phi, \mathcal{ D}_{ 0}}^{2},
\end{align*}
and
\begin{align*}
\left\vert \frac{\rm d}{{\rm d} t}\left\langle e^{tL_\phi}v , \mathbf{ 1}_{ \mathcal{ D}_{ 0}} \hat{ u}'_\phi\right\rangle_{m_\phi, \mathcal{ D}_{ 0}}\right\vert
&\leq  \left\Vert e^{tL_\phi}v \right\Vert_{m_\phi, \mathcal{ D}_{ 0}} \left\Vert L_\phi \hat u'_\phi \right\Vert_{m_\phi, \mathcal{ D}_{ 0}}.
\end{align*}
Now, since $\mathcal{L}_\phi \hat u'_\phi = 0$ and $\mathcal{L}_\phi$ is bounded, we have
\begin{equation*}
\left\Vert L_\phi \hat u'_\phi \right\Vert_{m_\phi, \mathcal{ D}_{ 0}}
= \left\Vert {\bf 1}_{\mathcal{ D}_{ 0}} \mathcal{L}_\phi  \left(\hat u'_\phi {\bf 1}_{\mathcal{ D}_{ 0}}\right)\right\Vert_{m_\phi}
= \left\Vert {\bf 1}_{\mathcal{ D}_{ 0}} \mathcal{L}_\phi  \left(\hat u'_\phi {\bf 1}_{\mathcal{ D}_{ 0}^c}\right)\right\Vert_{m_\phi} \leq c_{\mathcal{L}} \left\Vert \hat u'_\phi {\bf 1}_{\mathcal{D}_0^c}\right\Vert_{m_\phi}.
\end{equation*}
Relying on the bound \eqref{eq:bound_eLphi} and the fact that $ \change{\left\langle v , \mathbf{ 1}_{ \mathcal{ D}_{ 0}} \hat{ u}'_\phi\right\rangle_{m_\phi, \mathcal{ D}_{ 0}} = \left\langle \tilde{ v} , \mathbf{ 1}_{ \mathcal{ D}_{ 0}} \hat{ u}'_\phi\right\rangle_{m_\phi} = - \left\langle \tilde{ v} , \mathbf{ 1}_{ \mathcal{ D}_{ 0}^{ c}} \hat{ u}'_\phi\right\rangle_{m_\phi}}$ \change{(recall that $P_{ \varphi} \tilde{v}=0$)}
we obtain that \change{$ \left\vert \left\langle v , \mathbf{ 1}_{ \mathcal{ D}_{ 0}} \hat{ u}'_\phi\right\rangle_{m_\phi, \mathcal{ D}_{ 0}} \right\vert \leq \left\Vert \tilde{ v} \mathbf{ 1}_{ \mathcal{ D}_{ 0}^{ c}}\right\Vert_{ m_{ \varphi}} \left\Vert \hat{ u}_{ \varphi}^{ \prime} \mathbf{ 1}_{ \mathcal{ D}_{ 0}^{ c}}\right\Vert_{ m_{ \varphi}}$}. Hence, 
\begin{align*}
\left\vert \left\langle e^{tL_\phi}v , \mathbf{ 1}_{ \mathcal{ D}_{ 0}} \hat{ u}'_\phi\right\rangle_{m_\phi, \mathcal{ D}_{ 0}}\right\vert
&\leq  \left(\left\Vert \tilde{ v} \mathbf{ 1}_{ \mathcal{ D}_{ 0}^{ c}}\right\Vert_{ m_{ \varphi}}  + c_{\mathcal L} t \left\Vert v \right\Vert_{m_\phi, \mathcal{ D}_{ 0}}\right)\left\Vert \hat u'_\phi {\bf 1}_{\mathcal{D}_0^c}\right\Vert_{m_\phi}.
\end{align*}
This leads to the inequality
\begin{align*}
\frac{\rm d}{{\rm d} t}\left\Vert e^{tL_\phi} v\right\Vert_{m_\phi,\mathcal{D}_0}^2
&\leq  -2\kappa\left\Vert e^{tL_\phi}v \right\Vert_{m_\phi, \mathcal{ D}_{ 0}}^{ 2} + \change{2\kappa \left(\left\Vert \tilde{ v} \mathbf{ 1}_{ \mathcal{ D}_{ 0}^{ c}}\right\Vert_{ m_{ \varphi}}  + c_{\mathcal L} t \left\Vert v \right\Vert_{m_\phi, \mathcal{ D}_{ 0}}\right)^{ 2}\left\Vert \hat u'_\phi {\bf 1}_{\mathcal{D}_0^c}\right\Vert_{m_\phi}^{ 2}}.
\end{align*}
\change{This entails that
\begin{align*}
\left\Vert e^{tL_\phi} v\right\Vert_{m_\phi,\mathcal{D}_0}^2
&\leq e^{ -2 \kappa t}\left\Vert v\right\Vert_{m_\phi,\mathcal{D}_0}^2 + 2 \kappa \left\Vert \hat u'_\phi {\bf 1}_{\mathcal{D}_0^c}\right\Vert_{m_\phi}^{ 2} \int_{ 0}^{t} e^{ -2 \kappa(t-s)}\left(\left\Vert \tilde{ v} \mathbf{ 1}_{ \mathcal{ D}_{ 0}^{ c}}\right\Vert_{ m_{ \varphi}}  + c_{\mathcal L} s \left\Vert v \right\Vert_{m_\phi, \mathcal{ D}_{ 0}}\right)^{ 2}{\rm d}s,\\
&\leq e^{ -2 \kappa t}\left\Vert v\right\Vert_{m_\phi,\mathcal{D}_0}^2 +\left\Vert \hat u'_\phi {\bf 1}_{\mathcal{D}_0^c}\right\Vert_{m_\phi}^{ 2} \left(\left\Vert \tilde{ v} \mathbf{ 1}_{ \mathcal{ D}_{ 0}^{ c}}\right\Vert_{ m_{ \varphi}}  + c_{\mathcal L} t \left\Vert v \right\Vert_{m_\phi, \mathcal{ D}_{ 0}}\right)^{ 2}
\end{align*}
which is the desired result.}
\end{proof}

\subsection{Estimates on \texorpdfstring{$W$}{W}}
\label{sec:estimates_W_varpi}
The point of this section is to gather estimates concerning the asymptotic behaviors as $ \varepsilon\to 0$ of quantities like
\begin{equation}
\sup_{ i\in \mathbb{ U}} \sum_{ j\in \mathbb{ V}}W \left(x_{ i} - x_{ j}\right)^{ k}, k\geq1
\end{equation}
that appear everywhere in the rest of the paper. Here, $ \mathbb{ U}, \mathbb{ V} \in \left\lbrace \mathbb{ D}_{ -}, \mathbb{ D}_{ 0}, \mathbb{ D}_{ +}\right\rbrace$ where $ \mathbb{ D}_{ -}:= \left\lbrace j\in \mathbb{ Z}, j<-N_{ \varepsilon}\right\rbrace$, $ \mathbb{ D}_{ 0}:= \left\lbrace j\in \mathbb{ Z}, -N_{ \varepsilon}\leq j\leq N_{ \varepsilon}\right\rbrace$ and $ \mathbb{ D}_{ +}:= \left\lbrace j\in \mathbb{ Z}, j>N_{ \varepsilon}\right\rbrace$ is the discrete counterpart of the partition given in \eqref{eq:partition_R}. Ultimately, we prove at the end of the section Lemma~\ref{lem:bound_sum_varpi}. 
\begin{lemma}
\label{lem:W_IN_IN}
For all $k\geq1$, there exists some $c,C>0$ and $ \varepsilon_{ 0}>0$ depending only on $W$ and $k$ such that for all $ \varepsilon \in (0, \varepsilon_{ 0})$, the following holds.
\begin{enumerate}
\item The $ \mathbb{ U}=\mathbb{ D}_{ 0}$ vs $ \mathbb{ V}=\mathbb{ D}_{ 0}$ case:
\begin{align}
\label{eq:supWialpha_inside}
\sup_{ i=-N_{ \varepsilon}, \ldots, N_{ \varepsilon}} \varepsilon\sum_{ j=-N_{ \varepsilon}}^{ N_{ \varepsilon}} W \left(x_{ i} - x_{ j}\right)^{ k}& \leq  C,
\end{align}
\item The $ \mathbb{ U}=\mathbb{ D}_{ 0}$ vs $ \mathbb{ V}=\mathbb{ D}_{ \pm}$ case: for all $i=-N_{ \varepsilon}, \ldots, N_{ \varepsilon}$,
\begin{equation}
\label{eq:supWialpha_outside}
\begin{cases}
\varepsilon\sum_{ j >N_{ \varepsilon}} W \left(x_{ i} - x_{ j}\right)^{ k}\leq cW(x_{ i}- x_{ N})^{ k}\leq C,\\
\varepsilon\sum_{ j <-N_{ \varepsilon}} W \left(x_{ i} - x_{ j}\right)^{ k}\leq  cW(x_{ i}- x_{ -N})^{ k}\leq C.
\end{cases}
\end{equation}
\item The $ \mathbb{ U}=\mathbb{ D}_{ \pm}$ vs $ \mathbb{ V}=\mathbb{ D}_{ 0}$ case: 
\begin{equation}
\label{eq:supWialpha_outside2}
\begin{cases}
\varepsilon\sum_{ j=-N_{ \varepsilon}}^{ N_{ \varepsilon}} W \left(x_{ i} - x_{ j}\right)^{ k} \leq cW(x_{ i}- x_{ N})^{ k}\leq C,& \text{ for } i> N_{ \varepsilon}\\
\varepsilon\sum_{ j=-N_{ \varepsilon}}^{ N_{ \varepsilon}} W \left(x_{ i} - x_{ j}\right)^{ k} \leq  cW(x_{ i}- x_{ -N})^{ k}\leq C,& \text{ for }i<-N_{ \varepsilon}.
\end{cases}
\end{equation}
\end{enumerate}
In particular, the following estimate holds:
\begin{equation}
\label{eq:W_doublesum}
\sum_{ i\in \mathbb{ Z}} \sum_{ j=-N_{ \varepsilon}}^{ N_{ \varepsilon}} W(x_{ i}- x_{ j})^{ k} \leq C \varepsilon^{ -2- \beta}.
\end{equation}
\end{lemma}
\begin{proof}[Proof of Lemma~\ref{lem:W_IN_IN}]
Let $ k\geq1$. For $ -N_\varepsilon \leq i \leq N_\varepsilon$,
\begin{align*}
\sum_{j=-N_\varepsilon}^{N_\varepsilon}W(x_{i}-x_{j})^{ k} &=\frac{1}{(2\sigma)^{ k}}\sum_{j=-N_\varepsilon}^{i}e^{-\frac{ k(x_i-x_j)}{\sigma}} + \frac{1}{(2\sigma)^{ k}}\sum_{j=i+1}^{N_\varepsilon}e^{-\frac{ k(x_j-x_i)}{\sigma}}\\
&= \frac{e^{-\frac{ k x_i}{\sigma}}}{(2\sigma)^{ k}}\frac{e^{\frac{ k x_{i+1}}{\sigma}}-e^{-\frac{ k N_\varepsilon\varepsilon}{\sigma}}}{e^{\frac{ k\varepsilon}{\sigma}}-1} + \frac{e^{\frac{ k x_i}{\sigma}}}{(2\sigma)^{ k}}\frac{e^{-\frac{ k(N_{\varepsilon}+1)\varepsilon }{\sigma}}-e^{-\frac{ k x_{i+1}}{\sigma}}}{e^{-\frac{ k\varepsilon}{\sigma}}-1}\\
& = \frac{e^{\frac{ k\varepsilon}{\sigma}}-e^{-\frac{ k (N_\varepsilon+i)\varepsilon}{\sigma}}}{(2\sigma)^{ k}\left(e^{\frac{ k\varepsilon}{\sigma}}-1\right)}+ \frac{e^{\frac{ k(i-N_{\varepsilon}-1)\varepsilon }{\sigma}}-e^{-\frac{ k\varepsilon}{\sigma}}}{(2\sigma)^{ k}\left(e^{-\frac{ k\varepsilon}{\sigma}}-1\right)}\\
& = \frac{1+e^{\frac{ k\varepsilon}{\sigma}}-e^{-\frac{ k (N_\varepsilon+i)\varepsilon}{\sigma}}-e^{\frac{ k (i-N_\varepsilon)\varepsilon}{\sigma}}}{ (2\sigma)^{ k}\left(e^{\frac{ k\varepsilon}{\sigma}}-1\right)},
\end{align*}
so that
\begin{align*}
\varepsilon\sum_{j=-N_\varepsilon}^{N_\varepsilon}W(x_{i}-x_{j})^{ k}\leq\frac{ \varepsilon(1+e^{\frac{ k\varepsilon}{\sigma}})}{ (2\sigma)^{ k}\left(e^{\frac{ k\varepsilon}{\sigma}}-1\right)}\leq \frac{ 3}{ (2\sigma)^{ k} \frac{ k }{ \sigma}}:= C, 
\end{align*}
at least for small $ \varepsilon>0$ (depending on fixed $ \sigma$). The proof of \eqref{eq:supWialpha_outside} is similar, we prove it only in the case $j>N_{ \varepsilon}$: for $ k\geq1$, since $ j>i$, 
  \begin{equation}
  \label{eq:sumjsupN}
\varepsilon\sum_{ j> N_{ \varepsilon}} W(x_i-x_j)^{ k}
= \frac{ \varepsilon}{ (2 \sigma)^{ k}}\frac{ e^{ - \frac{ \varepsilon k\left(N_{ \varepsilon}-i\right)}{ \sigma}}}{ e^{ k\frac{ \varepsilon}{ \sigma}}-1}= \varepsilon\frac{ W(x_{ N}- x_{ i})^{ k}}{ e^{ k\frac{ \varepsilon}{ \sigma}}-1}\leq c W(x_{ N}- x_{ i})^{ k}\leq C.
  \end{equation}
 Turn now to \eqref{eq:supWialpha_outside2}, proven again only in the first case $i> N_{ \varepsilon}$, in a very similar way: if $i> N_\varepsilon$ we get, for $k\geq1$
\begin{align*}
\sum_{j=-N_\varepsilon}^{N_\varepsilon}W(x_{i}-x_{j})^k &=\frac{1}{(2 \sigma)^{ k}}\sum_{j=-N_\varepsilon}^{N_\varepsilon}e^{-\frac{k(x_i-x_j)}{\sigma}}\\
& = \frac{e^{-\frac{k(x_i-x_{N_\varepsilon})}{\sigma}}}{(2\sigma)^{ k}}\frac{e^{\frac{k \varepsilon}{\sigma}}-e^{-\frac{2k N_\varepsilon\varepsilon}{\sigma}}}{e^{\frac{k\varepsilon}{\sigma}}-1} \leq \frac{C}{\varepsilon} W(x_i-x_{N_\varepsilon})^k,
\end{align*}
which gives the result. We give finally the proof of \eqref{eq:W_doublesum}: we have
\begin{align}
\sum_{ i\in \mathbb{ Z}} \sum_{ j=-N_{ \varepsilon}}^{ N_{ \varepsilon}} W(x_{ i}- x_{ j})^{ k} &= \sum_{ \left\vert i \right\vert> N_{ \varepsilon}} \sum_{ j=-N_{ \varepsilon}}^{ N_{ \varepsilon}} W(x_{ i}- x_{ j})^{ k} + \sum_{ i=-N_{ \varepsilon}}^{ N_{ \varepsilon}} \sum_{ j=-N_{ \varepsilon}}^{ N_{ \varepsilon}} W(x_{ i}- x_{ j})^{ k}, \nonumber\\
&\leq 2\frac{ c}{ \varepsilon} \sum_{ i> N_{ \varepsilon}} W(x_{ i}- x_{ N})^{ k} + \frac{ C}{ \varepsilon} N_{ \varepsilon} \label{aux:double_sum}
\end{align}
where we have used \eqref{eq:supWialpha_outside2} for the first term and \eqref{eq:supWialpha_inside} for the second. Using now \eqref{eq:supWialpha_outside} for the first bound, we see that the first term in \eqref{aux:double_sum} is of order $ \varepsilon^{ -2}$ whereas the second is of order $ \varepsilon^{ -2- \beta}$, by definition of $N_{ \varepsilon}$.
\end{proof}
The following result gives the next order to \eqref{eq:supWialpha_outside}:
\begin{lemma}
\label{lem:W_IN_OUT}
For all $k\geq1$, there exist $C>0$ and $ \varepsilon_{ 0}>0$ depending only $W$ and $k$ such that for all $ \varepsilon\in (0, \varepsilon_{ 0})$, for $i=-N_{ \varepsilon}, \ldots, N_{ \varepsilon}$, we have
\begin{equation}
\label{eq:W_IN_OUT}
\begin{cases}
\left\vert  \varepsilon \sum_{ j> N_{ \varepsilon}} W(x_i-x_j)^{ k} -  \int_{ \mathcal{ D}_{ +}} W(x_{ i}-y)^{ k} {\rm d}y \right\vert  \leq C \varepsilon^{ 2} e^{ - \frac{ k\varepsilon}{ \sigma} \left(N_{ \varepsilon}-i\right)},\\
\left\vert  \varepsilon \sum_{ j<-N_{ \varepsilon}} W(x_i-x_j)^{ k} -  \int_{ \mathcal{ D}_{ -}} W(x_{ i}-y)^{ k} {\rm d}y \right\vert  \leq C \varepsilon^{ 2} e^{ - \frac{ k\varepsilon}{ \sigma} \left(N_{ \varepsilon}+i\right)}
\end{cases}
\end{equation}
\end{lemma}
\begin{proof}[Proof of Lemma~\ref{lem:W_IN_OUT}]
We only prove the first part of \eqref{eq:W_IN_OUT}. Let $k\geq1$, for $x\in \mathcal{ D}_{ 0}$
  \begin{align*}
  \int_{ \mathcal{ D}_{ +}} W(x-y)^{ k} {\rm d}y&= \frac{ 1}{ (2 \sigma)^{ k}} e^{ \frac{ kx}{ \sigma}}\int_{ \varepsilon(N_{ \varepsilon} + \frac{ 1}{ 2})}^{+\infty} e^{ - \frac{ ky}{ \sigma}}{\rm d}y = \frac{ \sigma}{ k(2 \sigma)^{ k}} e^{ - \frac{ k}{ \sigma} \left( \varepsilon \left(N_{ \varepsilon} + \frac{ 1}{ 2}\right)-x\right)}
  \end{align*}
  In particular, using \eqref{eq:sumjsupN},
  \begin{align*}
  \varepsilon \sum_{ j> N_{ \varepsilon}} W(x_i-x_j)^{ k} -  \int_{ \mathcal{ D}_{ +}} W(x_{ i}-y)^{ k} {\rm d}y&= \frac{ e^{ - \frac{ \varepsilon k\left(N_{ \varepsilon}-i\right)}{ \sigma}}}{ (2 \sigma)^{ k}} \frac{ \sigma}{ k}\left(\frac{ \varepsilon k/ \sigma}{ e^{ \frac{ \varepsilon k}{ \sigma}}-1}- e^{ - \frac{ \varepsilon k}{ 2\sigma}}\right)
\end{align*}
Since there exists some numerical constant $C>0$ such that for all $u\in [0, 1]$ we have $ \left\vert \frac{ u}{ e^{ u} -1} - e^{ -u/2}\right\vert \leq C u^{ 2}$, we obtain \eqref{eq:W_IN_OUT}.
\end{proof}
We are now in position to prove Lemma~\ref{lem:bound_sum_varpi}. Recall the definition of $ \varpi_{j}^{ ( \varepsilon)}$ in \eqref{eq:varpij}.
\begin{proof}[Proof of Lemma~\ref{lem:bound_sum_varpi}]
First remark that, since $\left\langle  \mathbf{ 1}_{ I_{ i_1}}, \mathbf{ 1}_{ I_{ i_2}}   \right\rangle_m=0$ for $i_1\neq i_2$,
\begin{align*}
\Vert \varpi_{j}^{ (\varepsilon)}\Vert_m^2 &= \varepsilon^2 \sum_{i_1=-N_\varepsilon}^{N_\varepsilon}\sum_{i_2=-N_\varepsilon}^{N_\varepsilon} W(x_{i_1}-x_{j})W(x_{i_2}-x_{j})\left\langle  \mathbf{ 1}_{ I_{ i_1}}, \mathbf{ 1}_{ I_{ i_2}}   \right\rangle_m\\
&\leq C \varepsilon^3 \sum_{i=-N_\varepsilon}^{N_\varepsilon}W(x_{i}-x_{j})^2,
\end{align*}
so that \eqref{eq:estim_varpi_j} is a direct consequence of \eqref{eq:supWialpha_inside} and \eqref{eq:supWialpha_outside2}.
Using this last result, the symmetry of $W$, \eqref{eq:supWialpha_outside} and the definition of $N_{ \varepsilon}$ in \eqref{eq:N_ell_eps}, we obtain
 \begin{align}
\sum_{ j\in \mathbb{Z}} \Vert \varpi_{ \varepsilon, j}\Vert_{ m}^k  &\leq C\varepsilon^k\left(2(N_\varepsilon+1)+ 2\sum_{j>N_{\varepsilon}}W(x_j-x_{N_{\varepsilon}})^{ k}\right)\nonumber \\
&\leq C\varepsilon^k\left( \varepsilon^{ -1 - \beta} + \varepsilon^{ -1}\right)\leq C\varepsilon^{k-1-\beta}. \label{eq:bound sum vert pi 3}
\end{align}
This concludes the proof of Lemma~\ref{lem:bound_sum_varpi}.
\end{proof}

\subsection{Discrete to continuous}\label{sec:discrete continuous}

The purpose of the present section is to prove Proposition~\ref{prop:bound_approx}.

\begin{proof}[Proof of Proposition~\ref{prop:bound_approx}]
Concerning \eqref{eq:Delta0_def}, we have $ \Delta_{ s, \varepsilon, i}^{ (0)}(x)= \Delta_{ s, \varepsilon, i}^{ (0,1)} + \Delta_{ s, \varepsilon, i}^{ (0,2)}(x)$ where
\begin{align*}
 \Delta_{ s, \varepsilon, i}^{ (0,1)} &:= \varepsilon \sum_{ j=-N_{ \varepsilon}}^{ N_{ \varepsilon}} W(x_i-x_j)\Big(f(U_{j,s^-})-f(U_{j,s})\Big),\\
 \Delta_{ s, \varepsilon, i}^{ (0,2)}(x) &:= \varepsilon \sum_{ j=-N_{ \varepsilon}}^{ N_{ \varepsilon}} W(x_i-x_j)f(U_{j,s}) - \int_{ \mathcal{ D}_{ 0}} W(x-y) f \left(U_{s}(y)\right) {\rm d}y.
\end{align*}
Concerning $ \Delta_{ s, \varepsilon, i}^{ (0,1)}$, since $f$ is Lipschitz, $W$ is bounded and the processes $(Z_{ j})$ a.s. do not jump together, there is a constant $C=C_{ W, f}>0$ such that, a.s. for all $j\in \mathbb{ Z}$, $s>0$, $ \left\vert \Delta U_{ j, s} \right\vert \leq C \varepsilon$, hence, using \eqref{eq:supWialpha_inside}, $ \left\vert \Delta_{ s, \varepsilon, i}^{ (0,1)} \right\vert = \left| \varepsilon\sum_{ j=-N_{ \varepsilon}}^{ N_{ \varepsilon}} W(x_i-x_j)\Big(f(U_{j,s^-})-f(U_{j,s})\Big) \right|\leq C \varepsilon$.
So, using the fact that $ \left\langle \mathbf{ 1}_{ I_{ i}}\, ,\, \mathbf{ 1}_{ I_{ j}}\right\rangle_{ L^{ 2}} = \varepsilon \mathbf{ 1}_{ i=j}$,
\begin{equation}
\label{aux:Delta01}
\left\Vert \sum_{ i=-N_{ \varepsilon}}^{ N_{ \varepsilon}} \Delta_{ s, \varepsilon, i}^{ (0,1)} \mathbf{ 1}_{ I_{ i}} \right\Vert_{L^{ 2}}^{ 2}  = \varepsilon \sum_{ i=-N_{ \varepsilon}}^{ N_{ \varepsilon}} \left\vert \Delta_{ s, \varepsilon, i}^{ (0,1)} \right\vert^{ 2} \leq C \varepsilon^{ 3} N_{ \varepsilon}= C \varepsilon^{ 2- \beta}.
\end{equation} 
Turn now to $ \Delta_{ s, \varepsilon, i}^{ (0,2)}$: we have, since $U_s(x)=U_{  j,s}$ for $x\in I_j$,
\begin{equation*}
\Delta_{ s, \varepsilon, i}^{ (0,2)}(x) = \sum_{ j=-N_\varepsilon}^{ N_\varepsilon} f(U_{  j, s}) \int_{I_j}\left( W(x_i-x_j)-W(x-z)\right){\rm d}z,
\end{equation*}
and, since $f$ is bounded by $1$,
\begin{align*}
\left\Vert \sum_{i=-N_\varepsilon}^{N_{\varepsilon}} \Delta_{ s, \varepsilon, i}^{ (0,2)} \mathbf{ 1}_{ I_{ i}} \right\Vert_{L^{ 2}}^2 &\leq \sum_{i=-N_\varepsilon}^{N_{\varepsilon}}\int_{I_i}\left(\sum_{ j=-N_\varepsilon}^{ N_\varepsilon} \left\vert \int_{I_j}\left( W(x_i-x_j)-W(x-z)\right){\rm d} z \right\vert \right)^2 {\rm d} x.
\end{align*}
Decomposing the sum $\sum_{ j=-N_{ \varepsilon}}^{ N_{ \varepsilon}}$ above into $\sum_{ j=-N_{ \varepsilon}}^{ i} + \sum_{ j=i+1}^{ N_{ \varepsilon}}$, it remains to treat the corresponding contributions separately. We only address the first one (and leave the similar second to the reader), that is
\begin{align*}
(I)&:= \sum_{i=-N_\varepsilon}^{N_{\varepsilon}}\int_{I_i}\left(\sum_{ j=-N_\varepsilon}^{ i} \left\vert \int_{I_j}\left( W(x_i-x_j)-W(x-z)\right){\rm d} z \right\vert \right)^2 {\rm d} x
\end{align*}
Noting that since $j\leq i$, $x_{ j}\leq x_{ i}$ and $z\leq x$ for $x\in I_{ i}$, $z\in I_{ j}$, we have
\begin{align*}
(I)&= \sum_{ i=-N_{ \varepsilon}}^{ N_{ \varepsilon}} \int_{ I_{ i}} \left(\frac{ 1}{ 2 \sigma} \sum_{ j=-N_{ \varepsilon}}^{ i} \left\vert \int_{ I_{ j}} \left( e^{ - \frac{ x_{ i}- x_{ j}}{ \sigma}} - e^{ - \frac{ x-z}{ \sigma}}\right) {\rm d}z\right\vert\right)^{ 2} {\rm d}x\\
&= \frac{ 1}{ \left(2 \sigma\right)^{ 2}}\sum_{ i=-N_{ \varepsilon}}^{ N_{ \varepsilon}} \left(\sum_{ j=-N_{ \varepsilon}}^{ i} e^{ \frac{ j \varepsilon}{ \sigma}} \right)^{ 2} \int_{ I_{ i}} \left\vert \varepsilon e^{ - \frac{ i \varepsilon}{ \sigma}} - \sigma e^{ - \frac{ x}{ \sigma}} \left(e^{ \frac{ \varepsilon}{ \sigma}}-1\right)\right\vert^{ 2} {\rm d}x\\
&=\frac{ 1}{ \left(2 \sigma\right)^{ 2}}\sum_{ i=-N_{ \varepsilon}}^{ N_{ \varepsilon}} \left( \frac{ e^{ - \frac{ N_{ \varepsilon} \varepsilon}{ \sigma}} - e^{ \frac{ \left(i+1\right) \varepsilon}{ \sigma}}}{ 1- e^{ \frac{ \varepsilon}{ \sigma}}} \right)^{ 2} e^{ - \frac{ 2 i \varepsilon}{ \sigma}}  \chi \left(\varepsilon, \sigma\right),
\end{align*}
for $ \chi(\varepsilon, \sigma):= \left(\varepsilon^{ 3} + 2 \sigma^{ 2} \varepsilon \left(2 - e^{ \frac{ \varepsilon}{ \sigma}} - e^{ - \frac{ \varepsilon}{ \sigma}}\right) - \frac{ \sigma^{ 3}}{ 2} \left(e^{ \frac{ \varepsilon}{ \sigma}}-1\right)^{ 2} \left(e^{ - \frac{ 2 \varepsilon}{ \sigma}}-1\right)\right)$ which satisfies the bound $ \left\vert \chi \left(\varepsilon, \sigma\right) \right\vert\leq C \varepsilon^{ 5}$ for small $ \varepsilon>0$. Hence,
\begin{align*}
(I)&\leq \frac{ C \varepsilon^{ 5}}{ \left(1- e^{ \frac{ \varepsilon}{ \sigma}}\right)^{ 2}}\sum_{ i=-N_{ \varepsilon}}^{ N_{ \varepsilon}} \left(e^{ - \frac{ \left(N_{ \varepsilon} +i\right) \varepsilon}{ \sigma}} - e^{ \frac{ \varepsilon}{ \sigma}}\right)^{ 2} .
\end{align*}
Expanding the previous square gives
\begin{align*}
(I)&\leq \frac{ C \varepsilon^{ 5}}{ \left(1- e^{ \frac{ \varepsilon}{ \sigma}}\right)^{ 2}} \left( \left( \frac{ 1- e^{ - \frac{ \left(4N_{ \varepsilon} +1\right) \varepsilon}{ \sigma}}}{ 1- e^{ - \frac{ 2 \varepsilon}{ \sigma}}}\right) - 2 e^{ \frac{ \varepsilon}{ \sigma}} \left( \frac{ 1- e^{ - \frac{ -(2N_{ \varepsilon} +1) \varepsilon}{ \sigma}}}{ 1- e^{ - \frac{ \varepsilon}{ \sigma}}}\right) + \left(2N_{ \varepsilon}+1\right)e^{ \frac{ 2 \varepsilon}{ \sigma}}\right)\\
&\leq  C \varepsilon^{ 2- \beta}.
\end{align*}
Hence the result \eqref{eq:Norm_Delta0}.

We now turn to $ \Delta_{ s, \varepsilon, i}^{ (+)}$ given by \eqref{eq:Delta+_def} (we leave the similar case of \eqref{eq:Delta-_def} to the reader): recalling the definitions \eqref{eq:Ui_outside} of $U_{ j,t}$ for $j>N_{ \varepsilon}$ and \eqref{eq:profileU}, we obtain that
\begin{align*}
\Delta_{ s, \varepsilon, i}^{ (+)}(x) &= \Delta_{ \varepsilon, i}^{ (+)}(x)= a_2 \left( \varepsilon\sum_{ j> N_{ \varepsilon}} W(x_i-x_j)- \int_{\mathcal{D}_+}W(x-y) {\rm d}y \right)\\
&= a_2 \left( \varepsilon\sum_{ j> N_{ \varepsilon}} W(x_i-x_j)- \int_{\mathcal{D}_+}W(x_i-y) {\rm d}y \right) \\
&\quad + a_2 \int_{\mathcal{D}_+}\big( W(x_i-y)- W(x-y)\big) {\rm d}y\\
& = \Delta_{\varepsilon, i}^{ (+,1)} + \Delta_{ \varepsilon, i}^{ (+,2)} (x). 
\end{align*}
Concerning $ \Delta_{ \varepsilon, i}^{ (+, 1)}$ we have, using Lemma~\ref{lem:W_IN_OUT},
\begin{align*}
\left\Vert \sum_{i=-N_\varepsilon}^{N_{\varepsilon}} \Delta_{ \varepsilon, i}^{ (+,1)} \mathbf{ 1}_{ I_{ i}} \right\Vert_{L^{ 2}}^2 &= \varepsilon\sum_{ i=-N_{ \varepsilon}}^{ N_{ \varepsilon}} \left\vert \Delta_{ \varepsilon, i}^{ (+, 1)} \right\vert^{ 2}\leq C \varepsilon^{ 5} \sum_{ i=-N_{ \varepsilon}}^{ N_{ \varepsilon}} e^{ - \frac{ 2\varepsilon}{ \sigma} \left(N_{ \varepsilon}-i\right)}\\
&\leq C \varepsilon^{ 5} \frac{ e^{ \frac{ 2 \varepsilon}{ \sigma}} -  e^{ - \frac{ 4 \varepsilon N}{ \sigma}}}{ e^{ \frac{ 2 \varepsilon}{ \sigma}}- 1} \leq C \varepsilon^{ 4}.
\end{align*}
Turn now to $ \Delta_{ \varepsilon, i}^{ (+, 2)}$. A straightforward computation leads to
\begin{align*}
\Delta_{ \varepsilon, i}^{ (+, 2)}(x) &=\frac{a_2}{2} e^{- \frac{ \varepsilon \left(N_{ \varepsilon}+ \frac{ 1}{ 2}\right)}{ \sigma}} \left(e^{x/\sigma}-e^{x_i/\sigma}\right), 
\end{align*}
and thus,
\begin{align*}
\left\Vert \sum_{i=-N_\varepsilon}^{N_{\varepsilon}} \Delta_{ \varepsilon, i}^{ (+, 2)} \mathbf{ 1}_{ I_{ i}}\right\Vert_{L^{ 2}}^2 &= \sum_{ i=-N_{ \varepsilon}}^{ N_{ \varepsilon}} \int_{ I_{ i}} \left\vert \Delta_{ \varepsilon, i}^{ (+, 2)}(x) \right\vert^{ 2}{\rm d}x\\
&= \frac{ a_{ 2}^{ 2}}{ 4} e^{- \frac{ 2\varepsilon \left(N_{ \varepsilon}+ \frac{ 1}{ 2}\right)}{ \sigma}}  \left( \sigma\int_{- \frac{ \varepsilon}{ 2 \sigma}}^{ \frac{ \varepsilon}{ 2 \sigma}}\left(e^{u}-1\right)^2 {\rm d} u\right)\sum_{i=-N_\varepsilon}^{N_{\varepsilon}} e^{ 2 x_{ i}/ \sigma}\\
&= \frac{ a_{ 2}^{ 2}}{ 4} \left( \varepsilon - 4 \sigma\sinh \left( \frac{ \varepsilon}{ 2 \sigma}\right) + \sigma\sinh \left( \frac{ \varepsilon}{ \sigma}\right)\right)  \frac{ e^{ \frac{ \varepsilon}{ \sigma}}- e^{ - \frac{ (4N_{ \varepsilon}+1) \varepsilon}{ \sigma}}}{ e^{ \frac{ 2 \varepsilon}{ \sigma}}-1}.
\end{align*}
We have now the Taylor expansion as $ \varepsilon\to0$, $\varepsilon - 4 \sigma\sinh \left( \frac{ \varepsilon}{ 2 \sigma}\right) + \sigma\sinh \left( \frac{ \varepsilon}{ \sigma}\right)= \frac{ \varepsilon^{ 3}}{ 12 \sigma^{ 2}} + o( \varepsilon^{ 3})$.
Since there is $C>0$ such that $ \left\vert \frac{ e^{ \frac{ \varepsilon}{ \sigma}}- e^{ - \frac{ (4N_{ \varepsilon}+1) \varepsilon}{ \sigma}}}{ e^{ \frac{ 2 \varepsilon}{ \sigma}}-1} \right\vert \leq \frac{ C \sigma}{ \varepsilon}$ as $ \varepsilon\to0$, we obtain finally that 
\begin{equation*}
\left\Vert \sum_{i=-N_\varepsilon}^{N_{\varepsilon}} \Delta_{ \varepsilon, i}^{ (+, 2)} \mathbf{ 1}_{ I_{ i}}\right\Vert_{L^{ 2}}^2\leq C \varepsilon^{ 2}.
\end{equation*}
This gives the desired result \eqref{eq:Norm_Delta+}.
\end{proof}

\subsection{Estimates on \texorpdfstring{$\hat{ u}_{ 0}$}{hatu0}}
We gather in this paragraph some useful estimates and properties concerning the profile $ \hat{ u}_{ 0}$.

\begin{lemma}
\label{lem:u0H2}
The profile $ \hat{ u}_{ 0}$ satisfies
\begin{equation}
\int_{ \mathbb{ R}} \left( \hat{ u}_{ 0}^{ \prime \prime}\right)^{ 2}(x) {\rm d}x<+\infty.
\end{equation}
\end{lemma}
\begin{proof}[Proof of Lemma~\ref{lem:u0H2}]
  Start with equation \eqref{eq:profile_hatu} (recall that $c=0$): $ \hat{ u}_{ 0}= W\ast f \left( \hat{ u}_{ 0}\right)$. The point is to differentiate twice the previous convolution, letting one derivative be upon $W$ and the second upon $ f(\hat{ u}_{ 0})$. There is however need to be a little careful due to the lack of differentiability of $W$ in $0$. Differentiating once
  \begin{align*}
  \hat{ u}_{ 0}(x)&= \frac{ 1}{ 2 \sigma} \int_{ -\infty}^{ x} e^{ - \frac{ x-y}{ \sigma}}f( \hat{ u}_{ 0}(y)) {\rm d}y +  \frac{ 1}{ 2 \sigma} \int_{ x}^{ +\infty} e^{ - \frac{ y-x}{ \sigma}}f( \hat{ u}_{ 0}(y)) {\rm d}y
  \end{align*}
  gives
  \begin{align*}
  \hat{ u}_{ 0}^{ \prime}(x)&= \frac{ f \left( \hat{ u}_{ 0}(x)\right)}{ 2 \sigma} - \frac{ 1}{ 2 \sigma^{ 2}} \int_{ -\infty}^{x} e^{ - \frac{ x-y}{ \sigma}} f( \hat{ u}_{ 0}(y)) {\rm d}y -  \frac{ f \left( \hat{ u}_{ 0}(x)\right)}{ 2 \sigma} \\
  &\quad + \frac{ 1}{ 2 \sigma^{ 2}} \int_{ x}^{+\infty} e^{ - \frac{ y-x}{ \sigma}} f( \hat{ u}_{ 0}(y)) {\rm d}y,\\
  &= \frac{ 1}{ 2 \sigma^{ 2}} \left(- \int_{ 0}^{ +\infty} e^{ - \frac{ v}{ \sigma}} f \left( \hat{ u}_{ 0}(x-v)\right) {\rm d}v + \int_{ -\infty}^{0} e^{ \frac{ v}{ \sigma}} f \left( \hat{ u}_{ 0}(x-v)\right) {\rm d}v\right),
  \end{align*}
  and hence,
  \begin{multline*}
  \hat{ u}_{ 0}^{ \prime \prime}(x)= \frac{ 1}{ 2 \sigma^{ 2}} \Bigg(- \int_{ 0}^{ +\infty} e^{ - \frac{ v}{ \sigma}} f^{ \prime} \left( \hat{ u}_{ 0}(x-v)\right) \hat{ u}_{ 0}^{ \prime}(x-v) {\rm d}v\\
   + \int_{ -\infty}^{0} e^{ \frac{ v}{ \sigma}} f^{ \prime} \left( \hat{ u}_{ 0}(x-v)\right)\hat{ u}_{ 0}^{ \prime}(x-v) {\rm d}v\Bigg).
  \end{multline*}
  Hence, in particular, some rough triangular inequality gives
  \begin{equation}
  \left\vert \hat{ u}_{ 0}^{ \prime \prime}(x) \right\vert \leq \frac{ 1}{ \sigma} W \ast \left(f^{ \prime}(\hat{ u}_{ 0}) \hat{ u}_{ 0}^{ \prime}\right)(x)
  \end{equation}
  so that by Young's inequality on convolution, since $f^{ \prime}$ is bounded
  \begin{equation}
  \left\Vert \hat{ u}_{ 0}^{ \prime \prime} \right\Vert_{ L^{ 2}} \leq \frac{ \left\Vert f^{ \prime} \right\Vert_{ \infty}}{ \sigma} \left\Vert W \right\Vert_{ L^{ 1}} \left\Vert \hat{ u}_{ 0}^{ \prime} \right\Vert_{ L^{ 2}}
  \end{equation}
  But the tail estimate \eqref{eq:uprime_exp} gives that $\left\Vert \hat{ u}_{ 0}^{ \prime} \right\Vert_{ L^{ 2}}<\infty$ and so is $\left\Vert \hat{ u}_{ 0}^{ \prime \prime} \right\Vert_{ L^{ 2}}<+\infty$.
\end{proof}

\section{Rate-based Neural Field Equation}
\label{sec:app rate based NFE}

\begin{proof}[Proof of Proposition~\ref{prop:correspondanceNFE}]
Concerning Item~\eqref{it:v}, the existence and uniqueness of a continuous solution $ t \mapsto v_{ t}$ to \eqref{eq:rateNFE} follows from standard Picard Iteration procedure (see e.g. \cite{Riedler2013}). The positivity and boundedness statement follows from the following equivalent mild formulation of \eqref{eq:rateNFE}
\begin{equation}
\label{eq:rateNFE_mild}
v_{ t}= e^{ - t} v_{ 0} + \int_{ 0}^{t} e^{ -(t-s)} f \left( \int_{ \mathbb{ R}} W(x-y) v_{ s}(y) {\rm d}y\right){\rm d}s
\end{equation}
together with the positivity and boundedness of $v_{0}$ and $f$.

\medskip

Turn now to Item~\eqref{it:utov}: let $u_{ 0}$ continuous and nonnegative and let $v_{ 0}$ satisfies \eqref{eq:u0Wv0}. Differentiating \eqref{eq:vtfromut} w.r.t. $t$ gives that 
\begin{equation}
\label{aux:vt_f_ut}
\partial_{ t} v_{ t}= - v_{ t} +f \left(u_{ t}\right).
\end{equation} 
Hence, $v$ will solve \eqref{eq:rateNFE} once we have been able to prove that $u_{ t}= W\ast v_{ t}$ for all $t\geq0$. To prove so, take the convolution of \eqref{aux:vt_f_ut} w.r.t. $W$: we have
\begin{align*}
\partial_{ t} \left(W\ast v_{ t}\right)+ W\ast v_{ t} = W\ast f(u_{ t}) = \partial_{ t} u_{ t} + u_{ t},
\end{align*}
where we have used the fact that $u$ solves \eqref{eq:NFE conv} in the last equality. Hence,
\begin{equation*}
\partial_{ t} \left(e^{ t} W\ast v_{ t}\right)= \partial_{ t} \left(e^{ t} u_{ t}\right),
\end{equation*}
so that integrating this identity and using that at time $t=0$ \eqref{eq:u0Wv0} is true by assumption, we obtain indeed that $u_{ t}= W\ast v_{ t}$. Putting this identity into \eqref{aux:vt_f_ut} gives that $v_{ t}$ solves \eqref{eq:rateNFE}. Conversely, let $v_{ 0}$ nonnegative, continuous and bounded and let $(v_{ t})_{ t\geq0}$ solution to \eqref{eq:rateNFE} with initial condition $v_{ 0}$. By Young inequality ($v\in L^{ \infty}( \mathbb{ R})$ and $W\in L^{ 1}(\mathbb{ R})$), \eqref{eq:utWvt} indeed defines a bounded and continuous function $u$ on $ \mathbb{ R}$. The fact that $u$ given by \eqref{eq:utWvt} is solution to \eqref{eq:NFE conv} follows from a simple convolution w.r.t. $W$ of \eqref{eq:rateNFE}.

\medskip

Turn now to Item~\eqref{it:TW}: suppose that \eqref{eq:rateNFE} admits a traveling wave solution $(x,t) \mapsto v_{ t}(x):= \hat{ v}(x-c^{ \prime}t)$ with profile $ \hat{ v}$ et speed $ c^{ \prime}$. Then, by Item~\eqref{subit:utfromvt}, $u_{ t}:= W\ast v_{ t}$ is a solution to \eqref{eq:NFE conv}. It is also a traveling wave solution to \eqref{eq:NFE conv}, as
\begin{align*}
u(x,t)= \int_{ \mathbb{ R}}W(x-y) \hat{ v}(y-c^{ \prime}t) {\rm d}y= \int_{ \mathbb{ R}}W(x-c^{ \prime}t -z) \hat{ v}(z) {\rm d}z = \left(W\ast \hat{ v}\right)(x-c^{ \prime}t).
\end{align*}
By uniqueness of the couple $ \left(\hat{ u}, c\right)$ (recall Theorem~\ref{th:EML}), we necessarily have that $ c^{ \prime}= c$ and the existence of some $ \varphi\in \mathbb{ R}$ such that 
\begin{equation}
\label{eq:Wastuphi}
W\ast \hat{ v}= \hat{ u}(\cdot + \varphi),
\end{equation} 
which proves \eqref{eq:hatuVSv}. Secondly, note that $ \hat{ v}$ necessarily solves 
\begin{equation}
\label{eq:hatv_c}
- c^{ \prime} \partial_{ \xi} \hat{ v}(\xi) = - \hat{ v}(\xi) + f \left(W\ast \hat{ v}\right)(\xi).
\end{equation}
In the case $c^{ \prime}=c=0$, \eqref{eq:hatv_c} boils down to 
\begin{align*}
\hat{ v}(\xi)= f \left(W\ast \hat{ v}\right)(\xi)= f \left( \hat{ u}( \xi+ \varphi)\right).
\end{align*}
\change{Consider now the case $c^{ \prime}=c> 0$ (the other case $c<0$ is left to the reader):} \eqref{eq:hatv_c} can be written as $\partial_{ \xi} \left(e^{ - \frac{ \xi}{ c}}\hat{ v}(\xi)\right) =  - \frac{ e^{ - \frac{ \xi}{ c}}}{ c} f \left(W\ast \hat{ v}\right)(\xi)$. Integrating this identity between $x$ and $+\infty$ gives
\begin{align*}
\hat{ v}(x)= \frac{ e^{ \frac{ x}{ c}}}{ c} \int_{ x}^{ +\infty} e^{ - \frac{ \xi}{ c}} f \left(W\ast \hat{ v}\right)(\xi) {\rm d}\xi=\int_{ 0}^{ +\infty} \frac{ e^{ - \frac{ z}{ c}}}{ c} f \left(W\ast \hat{ v}\right)(z+x) {\rm d}z.
\end{align*}
Using now \eqref{eq:Wastuphi} gives
\begin{equation*}
\hat{ v}(x)=\int_{ 0}^{ +\infty} \frac{ e^{ - \frac{ z}{ c}}}{ c} f \left( \hat{ u}(z+x+ \varphi)\right) {\rm d}z,
\end{equation*}
which gives uniqueness (up to translation) of $ \hat{ v}$, as well as the representation \eqref{eq:hatv}. It remains to show that such a profile $ \hat{ v}$ indeed defines a traveling wave solution to \eqref{eq:rateNFE}, with the desired properties. The monotonicity and the existence of appropriate limits in $\pm \infty$ follow directly from the corresponding properties for $ \hat{ u}$, the monotonicity of $f$ and the dominated convergence theorem. Define now $v_{ t}(x):= \hat{ v}(x-ct)$ and prove that it solves \eqref{eq:rateNFE}. First note that $v_{ t}$ may be written as \eqref{eq:vtfromut} for an appropriate initial condition $v_{ 0}$: indeed,
\begin{align}
v_{ t}(x)&= \hat{ v}(x-ct)= \frac{ 1}{ c}  \int_{ 0}^{+\infty} e^{ - \frac{ z}{ c}} f \left(\hat{ u}(x-ct+z+ \varphi)\right) {\rm d}z,\nonumber\\
&= \frac{ 1}{ c}  \int_{ -ct}^{+\infty} e^{ -(t+ z^{ \prime}/c)} f \left(\hat{ u}(x+z^{ \prime}+ \varphi)\right) {\rm d}z^{ \prime}=\int_{ -\infty}^{t} e^{ -(t-s)} f \left(\hat{ u}(x-cs+ \varphi)\right) {\rm d}s \nonumber\\
&= \int_{ 0}^{t} e^{ -(t-s)} f \left( \hat{ u}(x-cs+ \varphi)\right) {\rm d}s+ v_{ 0}e^{ -t}, \label{eq:vt_TW}
\end{align} 
for
\begin{equation}
\label{eq:v0_TW}
 v_{ 0}(x):= \int_{ 0}^{+\infty} e^{-s} f \left( \hat{ u}(x+cs+ \varphi)\right) {\rm d}s.
\end{equation}
Hence, applying Item~\eqref{subit:vtfromut} to the particular case of $u(x,t)= \hat{ u}(x-ct+ \varphi)$ gives that $v$ given by \eqref{eq:vt_TW} is a solution to \eqref{eq:rateNFE} as long as the compatibility condition \eqref{eq:u0Wv0} for the initial conditions is satisfied in this particular case: we need to check that $v_{ 0}$ given by \eqref{eq:v0_TW} satisfies
\begin{equation*}
u(x,0)=\hat{ u}(x+ \varphi)= W\ast v_{ 0}(x),\ x \in \mathbb{ R},
\end{equation*}
that is
\begin{equation}
\label{eq:pt_fixe_hatu}
\hat{ u}(x+ \varphi)= \int_{ 0}^{+\infty} e^{-s} \int_{ \mathbb{ R}} W(x-z)f \left( \hat{ u}(z+cs+ \varphi)\right) {\rm d}s {\rm d}z, \ x\in \mathbb{ R}, \ \varphi\in \mathbb{ R}.
\end{equation}
This is indeed the case, as identity \eqref{eq:pt_fixe_hatu} is a direct consequence of \eqref{eq:profile_hatu} together with similar calculations as before: we obtain from \eqref{eq:profile_hatu} that $\partial_{ \xi} \left(e^{ - \frac{ \xi}{ c}} \hat{ u}(\xi)\right) =  - \frac{ e^{ - \frac{ \xi}{ c}}}{ c} W\ast f( \hat{ u})(\xi)$ so that $  \hat{ u}(x)= e^{ \frac{ x}{ c}}\int_{ x}^{ +\infty} \frac{ e^{ - \frac{ \xi}{ c}}}{ c} W\ast f( \hat{ u})(\xi) {\rm d} \xi=  \int_{ x}^{ +\infty} \frac{ e^{ - \frac{ \xi-x}{ c}}}{ c} \int_{ \mathbb{ R}}W(\xi-y)\ast f( \hat{ u}(y)) {\rm d}y {\rm d} \xi$ wich gives \eqref{eq:pt_fixe_hatu} by the changes of variables $s= \frac{ \xi-x}{ c}$ and $z=y-cs- \varphi$. Hence, $v$ given by \eqref{eq:vt_TW} is indeed the solution to \eqref{eq:rateNFE} with initial condition $v_{ 0}$ and hence a traveling wave solution to \eqref{eq:rateNFE}.
\end{proof}

\section{Regularity of the isochron map}\label{sec:app regularity isochron}

\begin{proof}[Proof of Proposition~\ref{prop:regularity isochron}]
Let us denote by $T_t(u_0)$ the flow of solutions associated to \eqref{eq:NFE conv}. Following the usual theory of regularity with respect to the initial condition (see for example \cite{sell2013dynamics}), the two successive Frechet derivatives with respect to the initial condition are given by $h_t={\rm D} T_t(u_0)[h]$ with
\begin{equation}
\label{eq:ht}
\partial_t h_t = - h_t + W *(f'(u_t)h_t),\quad h_0=h,
\end{equation}
and $\xi_t={\rm D}^2 T_t(u_0)[h^1,h^2]$ with
\begin{equation*}
\partial_t \xi_t = - \xi_t + W *\left(f'(u_t)\xi_t+ f''(u_t)h^1_t h^2_t\right),\quad \xi_0=0, \quad h^i_t={\rm D}T(u_0)[h^i].
\end{equation*}
For the third derivative one needs to be a bit more careful since Hölder's and Young's inequality only imply
\begin{equation*}
\left\Vert  W *\left(f^{(3)}(u_t)h^1_t h^2_t h^3_t\right)\right\Vert_{L^2} \leq \left\Vert  W \right\Vert_{L^2} \left\Vert  f^{(3)}(u_t)\right\Vert_{L^\infty} \left\Vert h^1_t \right\Vert_{L^2} \left\Vert h^2_t\right\Vert_{L^2}\left\Vert h^3_t\right\Vert_{L^\infty}.
\end{equation*}
We will prove that, considering the space $\mathcal{X} = L^2\cap L^\infty$ (endowed with the norm  $\Vert \cdot \Vert_{L^2}+\Vert \cdot \Vert_{L^\infty}$), if $h\in \mathcal{X}$ then ${\rm D}T_t(u_0)[h]\in \mathcal{X}$. Then, for $h^3\in \mathcal{X}$ and $h^1,h^2\in L^2$, relying again on the usual regularity theory \cite{sell2013dynamics}, one can prove that the map $u_0 \mapsto {\rm D}^2T_t(u_0)[h^2,h^3]$ is differentiable, with differential ${\rm D} {\rm D}^2T_t(u_0)[h^2,h^3][h^1] =:{\rm D}^3 T_t(u_0)[h^1,h^2,h^3]$ given by, for $\zeta_t = {\rm D}^3 T(u_0)[h^1,h^2,h^3]$,
\begin{multline*}
\partial_t \zeta_t = - \zeta_t + W *\left(f'(u_t)\zeta_t+ f''(u_t)\left(\xi^{1,2}_t h^3_t +\xi^{1,3}_t h^2_t +\xi^{2,3}_t h^1_t  \right)+f^{(3)}(u_t)h^1_t h^2_t h^3_t\right),\\
 \zeta_0=0, \quad h^i_t={\rm D}T(u_0)[h^i], \quad \xi^{i,j}_t = {\rm D}^2 T(u_0)[h^i,h^j].
\end{multline*}
Here ${\rm D} T_{t}$, ${\rm D}^{ 2} T_{t}$ and \change{${\rm D}^{ 3} T_{t}$} are respectively elements of $C(\mathcal{N}, \mathcal{B}(L^2;L^2))$, $C(\mathcal{N}, \mathcal{B}((L^2)^2;L^2))$ and 
$C(\mathcal{N}, \mathcal{B}((L^2)^2\times \mathcal{X};L^2))$.
The idea of the following proof is to show that, for an increasing sequence of times $(t_n)_{n\in \mathbb{N}}$ going to infinity, $({\rm D} T_{t_n})_{n\in \mathbb{N}}$, $({\rm D}^{ 2} T_{t_n})_{n\in \mathbb{N}}$ and $({\rm D}^{ 3} T_{t_n})_{n\in \mathbb{N}}$ are Cauchy sequences.

\medskip

\noindent
{\it First step.} We show that ${\rm D} T_t$ is uniformly bounded in time. Considering $u_0\in \mathcal{N}$ such that $\Theta(u_0)=\varphi$ and remarking that
\begin{equation*}
\partial_t h_t = \mathcal{L}_\varphi h_t + W*\left( (f'(u_t)-f'(\hat u_\varphi)) h_t\right),
\end{equation*}
we get, since $u_0\in \mathcal{N}$, $\Vert u_0-\hat u_\varphi\Vert_{L^2}\leq C_0$,
\begin{align*}
\Vert h_t\Vert_{L^2}&\leq \left\Vert e^{t\mathcal{L}_\varphi} h_0\right\Vert_{L^2} +\int_0^t\left\Vert e^{(t-s)\mathcal{L}_\varphi}W*\left( (f'(u_s)-f'(\hat u_\varphi)) h_s\right)\right\Vert_{L^2} {\rm d} s\\
&\leq C_\mathcal{L}\left\Vert  h_0\right\Vert_{L^2} + C_0 C_\mathcal{L}C_\lambda\Vert W\Vert_{L^2} \Vert f''\Vert_{L^\infty}\int_0^t  e^{-\lambda s} \Vert h_s\Vert_{L^2} {\rm d} s,
\end{align*}
where we recall the definitions of the constants $ C_{  \mathcal{ L}}$ and $ C_{ \lambda}$ in \eqref{eq:contract_ecL} and \eqref{eq:convergence ut}. This implies by Gronwall's inequality,
\begin{equation*}
\sup_{t\geq 0}\left\Vert {\rm D} T_t(u_0)[h_0]\right\Vert_{L^2}\leq C_\mathcal{L}e^{\frac{C_0C_\mathcal{L}C_\lambda\Vert W\Vert_{L^2}\Vert f''\Vert_{L^\infty}}{\lambda}}\left\Vert h_0\right\Vert_{L^2}.
\end{equation*}
Moreover, writing \eqref{eq:ht} in a mild form, we obtain a similar $L^{ \infty}$ estimate:
\begin{align*}
\Vert h_t\Vert_{L^\infty} &\leq e^{-t} \Vert h_0\Vert_{L^\infty} + \int_0^t e^{-(t-s)}\left\Vert W*(f'(u_s)h_s) \right\Vert_{L^\infty} {\rm d} s\\
&\leq e^{-t} \Vert h_0\Vert_{L^\infty} +\Vert W\Vert _2 \Vert f'\Vert_{L^\infty}  \int_0^t e^{-(t-s)} \Vert h_s\Vert_{L^2}{\rm d} s\\
&\leq e^{-t} \Vert h_0\Vert_{L^\infty} +\Vert W\Vert _2 \Vert f'\Vert_{L^\infty} \sup_{s\geq 0}\left\Vert {\rm D} T_s(u_0)[h_0]\right\Vert_{L^2}.
\end{align*}
So there exist constants $c_1>0$ and $c_2>0$ such that for all $u_0\in \mathcal{N}$, all $t\geq 0$ and all $h_0\in \mathcal{X}$,
\begin{equation}\label{eq:DT_t bounded}
\left\Vert {\rm D} T_t(u_0)[h_0]\right\Vert_{L^2} \leq c_1 \Vert h_0\Vert_{L^2},\quad \left\Vert {\rm D} T_t(u_0)[h_0]\right\Vert_{L^\infty} \leq e^{-t}\Vert h_0\Vert_{L^\infty}+c_2 \Vert h_0\Vert_{L^2}.
\end{equation}

\medskip

\noindent
{\it Second step.} We prove that $({\rm D} T_{t_n})_{n\in \mathbb{N}}$ is a Cauchy sequence in $C(\mathcal{N}, \mathcal{B}(L^2;L^2))$. 
For $t_m>t_n$ we have the decomposition
\begin{align*}
h_{t_m}-h_{t_n} &= e^{t_n\mathcal{L}_\varphi}\left(e^{(t_m-t_n)\mathcal{L}_\varphi}-1\right)h_0\\
&\quad +\int_0^{t_n} e^{(t_n-s)\mathcal{L}_\varphi}\left(e^{(t_m-t_n)\mathcal{L}_\varphi}-1\right)W*\left((f'(u_s)-f'(\hat u_\varphi) )h_s\right){\rm d} s\\
&\quad +\int_{t_n}^{t_m}e^{(t_m-s)\mathcal{L}_\varphi}W*\left((f'(u_s)-f'(\hat u_\varphi) )h_s\right){\rm d} s.\\
&= K_1+ K_2+ K_3.
\end{align*}
Since
\begin{equation}\label{eq:eL-I}
e^{t\mathcal{L}_\varphi}-1=\left(e^{t\mathcal{L}_\varphi}-1\right) P^\perp_\varphi,
\end{equation}
we obtain
\begin{equation*}
\left\Vert K_1\right\Vert_{L^2}\leq C_{\mathcal{L}} e^{-\kappa t_n}\Vert h_0\Vert_{L^2}.
\end{equation*}
Moreover,
\begin{align*}
\left\Vert K_2\right\Vert_{L^2} & \leq C_{\mathcal{L}}\int_0^{t_n} e^{-\kappa (t_n-s)} \left\Vert W*\left((f'(u_s)-f'(\hat u_\varphi) )h_s\right)\right\Vert_{L^2}{\rm d}s \\
& \leq C_0 C_{\mathcal{L}}C_\lambda \Vert W\Vert_{L^2} \Vert f''\Vert_{L^\infty} \int_0^{t_n} e^{-\kappa (t_n-s)} e^{-\lambda s} \Vert h_s\Vert_{L^2} {\rm d} s\\
&\leq C e^{-\lambda t_n} \Vert h_0\Vert_{L^2},
\end{align*}
and
\begin{align*}
\left\Vert K_3\right\Vert_{L^2} & \leq C_{\mathcal{L}}\int_{t_n}^{t_m} \left\Vert W*\left((f'(u_s)-f'(\hat u_\varphi) )h_s\right)\right\Vert_{L^2}{\rm d}s \\
& \leq C_0 C_{\mathcal{L}} C_\lambda \Vert W\Vert_{L^2} \Vert f''\Vert_{L^\infty} \int_{t_n}^{t_m} e^{-\lambda s} \Vert h_s\Vert_{L^2} {\rm d} s\\
&\leq C e^{-\lambda t_n} \Vert h_0\Vert_{L^2}.
\end{align*}
We deduce that there exists a constant $c_3>0$ such that for $u_0\in \mathcal{N}$, all $t_m>t_n \geq 0$ and all $h_0\in \mathcal{X}$,
\begin{equation*}
\left\Vert h_{t_m }-h_{t_n}\right\Vert_{L^2} \leq c_3 e^{-\lambda t_n} \Vert h_0\Vert_{L^2}.
\end{equation*}
Hence $(h_{ t_{ n}})_{ n\geq1}$ is indeed a Cauchy sequence in $L^{ 2}$ and converges in $L^2$ as $n\to\infty$ to a $h_\infty\in L^2$, with
\begin{equation}\label{eq:DT_t CV L2}
\Vert h_t-h_\infty\Vert_{L^2} \leq c_3 e^{-\lambda t}\Vert h_0\Vert_{L^2}.
\end{equation}
Moreover
\begin{align*}
\left\Vert P^\perp_\varphi h_t\right\Vert_{L^2} &\leq \left\Vert P^\perp_\varphi e^{t\mathcal{L}_\varphi } h_0\right\Vert_{L^2}+\int_0^t \left\Vert P^\perp_\varphi e^{(t-s)\mathcal{L}_\varphi}W*\left((f'(u_s)-f'(\hat u_\varphi) )h_s\right)\right\Vert_{L^2}{\rm d} s\\
&\leq e^{-\lambda t} \Vert h_0\Vert_{L^2}+C_0 C_\lambda  \Vert W\Vert_{L^2} \Vert f''\Vert_{L^\infty} \sup_{s\geq 0}\Vert h_s\Vert_{L^2}\int_0^t e^{-\lambda(t-s)}e^{-\lambda s} {\rm d} s\\
&\leq C(1+t)e^{-\lambda t}\Vert h_0\Vert_{L^2},
\end{align*}
so that $P^\perp_\varphi h_\infty=0$, $h_\infty$ is proportional to $\hat u_\varphi'$.
Let us make some further calculation that will be useful in the following: concerning the $L^\infty$ norm, remarking that
\begin{equation}
h_t  = e^{-t}h_0 + \int_0^t e^{-(t-s)} W*\left(f'(u_s)h_s\right){\rm d} s,
\end{equation}
we obtain, decomposing $f'(u_s)h_s=f'(\hat u_\varphi)h_\infty+(f'(u_s)-f'(\hat u_\varphi)) h_\infty +f'(u_s)(h_s-h_\infty)$ and relying on similar estimates as above, for some constant $c_4>0$,
\begin{equation*}
\left\Vert h_t - \left(1-e^{-t}\right) W*\left(f'(\hat u_\varphi)  h_\infty\right)\right\Vert_{L^\infty}\leq { c_4(1+t) e^{-\lambda t} \Vert h_0 \Vert_{L^2}.}
\end{equation*}
Relying on the fact that $h_\infty=W*\left(f'(\hat u_\varphi)  h_\infty\right)$ (recall that $P^\perp_\varphi h_\infty=0$) and on \eqref{eq:DT_t bounded}, we deduce that, for some constant $c_5>0$,
{
\begin{equation}\label{eq:DT_t CV Linfty}
\left\Vert h_t -  h_\infty\right\Vert_{L^\infty}\leq c_4 (1+t) e^{-\lambda t} \Vert h_0 \Vert_{L^2}+ c_5 e^{-t} \Vert h_0\Vert_{L^\infty}.
\end{equation}}

\medskip

Note that one may now deduce \eqref{eq:Theta_dist}: for $v\in \mathcal{N}$ and $w\in \mathcal{M}$ such that $\Vert v-w\Vert_{L^2}={\rm dist}_{L^2}(v,\mathcal{M})$, one has, since $T_t(w)=w$,
\begin{align*}
\Vert v-T_\infty(v)\Vert_{L^2} &\leq \Vert v-w\Vert_{L^2}+\Vert T_\infty(w)-T_\infty(v)\Vert_{L^2}\\
&\leq \Vert v-w\Vert_{L^2}+\int_0^1 \Vert {\rm D}T_\infty(v+r(w-v))[w-v]\Vert_{L^2}{\rm d} r\\
&\leq (1+c_1){\rm dist}_{L^2}(v,\mathcal{M}).
\end{align*}

\medskip

\noindent
{\it Third step.} We prove that ${\rm D}^2 T_t(u_0)[h^1_0,h^2_0]$ is uniformly bounded in $C(\mathcal{N},\mathcal{B}((L^2)^2;L^2))$.
For $t>0$, denoting $\xi_t ={\rm D}^2 T_t(u_0)[h^1_0,h^2_0]$ we have
\begin{align*}
\xi_t& = \int_0^t e^{(t-s)\mathcal{L}_\varphi}W *\big((f'(u_s)-f'(\hat u_\varphi))\xi_s \big){\rm d} s + \int_0^t e^{(t-s)\mathcal{L}_\varphi}W *\big(f''(u_s)h^1_s h^2_s \big){\rm d} s\\
&=\int_0^t e^{(t-s)\mathcal{L}_\varphi}W *\big((f'(u_s)-f'(\hat u_\varphi))\xi_s \big){\rm d} s + \int_0^t e^{(t-s)\mathcal{L}_\varphi}W *\big(f''(\hat u_\varphi)h^1_\infty h^2_\infty \big){\rm d} s\\
&\quad + \int_0^t e^{(t-s)\mathcal{L}_\varphi}W *\big((f''(u_s)-f''(\hat u_\varphi))h^1_\infty h^2_\infty \big){\rm d} s\\
&\quad + \int_0^t e^{(t-s)\mathcal{L}_\varphi}W *\big(f''(u_s)(h^1_s-h^1_\infty) h^2_\infty \big){\rm d} s\\
&\quad + \int_0^t e^{(t-s)\mathcal{L}_\varphi}W *\big(f''(u_s)h^1_s( h^2_s-h^2_\infty) \big){\rm d} s:=I_1+I_2+I_3+I_4+I_5.
\end{align*}
Relying on similar estimates as above, recalling in particular \eqref{eq:convergence ut}, \eqref{eq:DT_t bounded} and \eqref{eq:DT_t CV L2}, we get
\begin{equation*}
\Vert I_1\Vert_{L^2}\leq C\int_0^t e^{-\lambda s} \Vert \xi_s\Vert_{L^2} {\rm d} s, \quad \Vert I_3+I_4+I_5\Vert_{L^2}\leq C \Vert h^1_0\Vert_{L^2}\Vert h^2_0\Vert_{L^2}.
\end{equation*}
Moreover, since $h^i_\infty= \eta_i \hat u'_\varphi$ for some constants $ \eta_i$, by translation invariance of $\mathcal{M}$,
\begin{align*}
P_\varphi\left(W*\big(f''(\hat u_\varphi)h^1_\infty h^2_\infty\big)\right) = \eta_1 \eta_2 \hat u_\varphi \int \hat u'_0(x) W(x-y)f''(\hat u_0)(y) \hat u'_0(y) m_0(x) {\rm d} x{\rm d} y =0,
\end{align*}
where we used the fact that $\hat u'_0$ and $m_0$ are even while $f''(\hat u_0)$ is odd. So
\begin{equation*}
\Vert I_2\Vert_{L^2}\leq C_\mathcal{L} \int_0^t e^{-\kappa(t-s)}\left\Vert W *\big(f''(\hat u_\varphi)h^1_\infty h^2_\infty \big)\right\Vert_{L^2} {\rm d} s\leq C  \Vert h^1_0\Vert_{L^2}\Vert h^2_0\Vert_{L^2}.
\end{equation*}
Applying again Gronwall's inequality we deduce that there exist a constant $c_6>0$ such that for all $u_0\in \tilde{\mathcal{N}}$, all $t\geq 0$ and all $h^1_0, h^2_0\in L^2$,
\begin{equation}\label{eq:D2T bounded}
\left\Vert {\rm D}^2 T_t(u_0)[h^1_0,h^2_0]\right\Vert_{L^2} \leq c_6 \Vert h^1_0\Vert_2 \Vert h^2_0\Vert_{L^2}.
\end{equation}

\medskip

\noindent
{\it Fourth step.} We prove that $({\rm D}^2 T_{t_n})_{n\in \mathbb{N}}$ is a Cauchy sequence in $C(\mathcal{N}, \mathcal{B}((L^2)^2;L^2))$.
For $t_m>t_n$ we have the decomposition
\begin{align*}
\xi_{t_m}-\xi_{t_n} &= \int_0^{t_n}e^{(t_n-s)\mathcal{L}_\varphi}\left(e^{(t_m-t_n)\mathcal{L}_\varphi}-1\right) W*\big((f'(u_s)-f'(\hat u_\varphi))\xi_s \big){\rm d} s\\
&\quad+\int_0^{t_n}e^{(t_n-s)\mathcal{L}_\varphi}\left(e^{(t_m-t_n)\mathcal{L}_\varphi}-1\right) W*\big(f''(u_s)h^1_sh^2_s {-f''(\hat u_\varphi)h^1_\infty h^2_\infty}\big){\rm d} s\\
&\quad+\int_{t_n}^{t_m}e^{(t_m-s)\mathcal{L}_\varphi}W*\big((f'(u_s)-f'(\hat u_\varphi))\xi_s \big){\rm d} s\\
&\quad+\int_{t_n}^{t_m}e^{(t_m-s)\mathcal{L}_\varphi}W*\big(f''(u_s)h^1_sh^2_s {-f''(\hat u_\varphi)h^1_\infty h^2_\infty}\big){\rm d} s\\
&\quad {+\int_{t_n}^{t_m}e^{s\mathcal{L}_\varphi}W*\big(f''(\hat u_\varphi)h^1_\infty h^2_\infty\big)){\rm d} s}\\
&=: J_1+J_2+J_3+J_4 {+J_5}.
\end{align*}
Relying on \eqref{eq:convergence ut}, \eqref{eq:eL-I} and \eqref{eq:D2T bounded}, we get
\begin{align*}
\Vert J_1\Vert_{L^2} &\leq C_0 C_\mathcal{L} C_\lambda \Vert W\Vert_{L^2}\Vert f''\Vert_{L^\infty} \int_0^{t_n}e^{-\kappa(t_n-s)}e^{-\lambda s}\Vert \xi_s\Vert_{L^2}{\rm d} s \\
& \leq Ce^{-\lambda t_n} \Vert h^1_0\Vert_{L^2}\Vert h^2_0\Vert_{L^2}.
\end{align*}
Moreover,
\begin{align*}
J_2 &=\int_0^{t_n}e^{(t_n-s)\mathcal{L}_\varphi}\left(e^{(t_m-t_n)\mathcal{L}_\varphi}-1\right) W*\big(f''(u_s)h^1_s(h^2_s -h^2_\infty)\big){\rm d} s\\
&\quad +\int_0^{t_n}e^{(t_n-s)\mathcal{L}_\varphi}\left(e^{(t_m-t_n)\mathcal{L}_\varphi}-1\right) W*\big(f''(u_s)(h^1_s-h^1_\infty)h^2_\infty \big){\rm d} s\\
&\quad +\int_0^{t_n}e^{(t_n-s)\mathcal{L}_\varphi}\left(e^{(t_m-t_n)\mathcal{L}_\varphi}-1\right) W*\big((f''(u_s)-f''(\hat u_\varphi))h^1_\infty h^2_\infty \big){\rm d} s \\
&= J_2^1+J_2^2+J_2^3.
\end{align*}
We remark first that, relying in particular on \eqref{eq:convergence ut} and \eqref{eq:DT_t bounded},
\begin{align*}
\left\Vert J^3_2\right\Vert_{L^2}&\leq C_\mathcal{L} \Vert W\Vert_{L^2}\Vert f^{(3)}\Vert_{L^\infty} \int_0^{t_n}e^{-\kappa(t_n-s)} \Vert u_s-\hat u_\varphi\Vert_{L^2} \Vert h^1_\infty\Vert_{L^2}\Vert h^2_\infty\Vert_{L^\infty} {\rm d} s\\
&\leq  C_\mathcal{L} \Vert W\Vert_{L^2}\Vert f^{(3)}\Vert_{L^\infty}  \int_0^{t_n} e^{-\kappa(t_n-s)}C_0e^{-\lambda s} c_1\Vert h^1_0\Vert_{L^2}c_2\Vert h^2_0\Vert_{L^2}{\rm d} s\\
&\leq C e^{-\lambda t_n}\Vert h^1_0\Vert_{L^2}\Vert h^2_0\Vert_{L^2},
\end{align*}
and similarly
\begin{align*}
\left\Vert J^2_2\right\Vert_{L^2} \leq C e^{-\lambda t_n}\Vert h^1_0\Vert_{L^2}\Vert h^2_0\Vert_{L^2}, \quad \left\Vert J^1_2\right\Vert_{L^2} \leq C e^{-\lambda t_n}\Vert h^1_0\Vert_{L^2}\Vert h^2_0\Vert_{L^2}.
\end{align*}
For the remaining terms we have, recalling \eqref{eq:D2T bounded},
\begin{align*}
\left\Vert J_3\right\Vert_{L^2} &\leq C_0 C_{\mathcal{L}} C_\lambda \Vert W\Vert_{L^2}\Vert f''\Vert_{L^\infty} \int_{t_n}^{t_m}e^{-\lambda s} \Vert\xi_s\Vert_{L^2}{\rm d} s\leq Ce^{-\lambda t_n}\Vert h^1_0\Vert_{L^2}\Vert h^2_0\Vert_{L^2},
\end{align*}
and, relying on a similar decomposition as for $J_2$ (in particular with a term equal to zero), we get
\[
\left\Vert J_4\right\Vert_{L^2} \leq Ce^{-\lambda t_n}\Vert h^1_0\Vert_{L^2}\Vert h^2_0\Vert_{L^2}.
\]
{ Finally, since $P_\varphi\left(W*\big(f''(\hat u_\varphi)h^1_\infty h^2_\infty\big)\right) = 0$, we have
\begin{equation*}
 \Vert J_5\Vert_{L^2}\leq C_{\mathcal{L}}\Vert W\Vert_{L^2}\Vert f''\Vert_\infty \Vert h^1_0\Vert_{L^2}\Vert h^2_0\Vert_{L^2}\int_{t_n}^{t_m}e^{-\kappa s}{\rm d} s \leq C e^{-\kappa t_n}\Vert h^1_0\Vert_{L^2}\Vert h^2_0\Vert_{L^2}.
\end{equation*}}
We deduce that there exists a constant $c_7>0$ such that, for all $u_0\in \mathcal{N}$, all $t_m>t_n\geq 0$ and all $h^1_0,h^2_0\in L^2$,
\begin{equation*}
\Vert \xi_{t_m}-\xi_{t_n}\Vert_{L^2}\leq c_7 e^{-\lambda t_n} \Vert h^1_0\Vert_{L^2}\Vert h^2_0\Vert_{L^2},
\end{equation*}
and thus  $ \left(\xi_{ t_{ n}}\right)_{ n\geq1}$ converges as $n\to\infty$ in $L^2$ to a $\xi_\infty\in L^2$, with
\begin{equation}\label{eq:D2T CV}
\Vert \xi_{t}-\xi_{\infty}\Vert_{L^2}\leq c_7 e^{-\lambda t} \Vert h^1_0\Vert_{L^2}\Vert h^2_0\Vert_{L^2}.
\end{equation}

\medskip

\noindent
{\it Fifth step.} we prove that ${\rm D}^3 T_t(u_0)[h^1_0,h^2_0,h^3_0]$ is uniformly bounded in $C(\mathcal{N},\mathcal{B}((L^2)^2\times\mathcal{X};L^2))$.
For $t>0$, denoting $\zeta_t ={\rm D}^3 T_t(u_0)[h^1_0,h^2_0,h^3_0]$ and $\xi^{i,j}_t = {\rm D}^2 T(u_0)[h^i_0,h^j_0]$ we have
\begin{align}
\zeta_t =&\int_0^t e^{(t-s)\mathcal{L}_\varphi}W *\big((f'(u_s)-f'(\hat u_\varphi))\zeta_s \big){\rm d} s\nonumber \\
&+ \int_0^t e^{(t-s)\mathcal{L}_\varphi}W *\left(f''(u_s)\left(\xi^{1,2}_s h^3_s +\xi^{1,3}_s h^2_s +\xi^{2,3}_s h^1_s  \right)+f^{(3)}(u_s)h^1_s h^2_s h^3_s\right){\rm d} s.\label{eq:zeta mild}
\end{align}
Let us first remark that
\begin{equation}\label{eq:Ptvarphi(...)=0}
P_\varphi\left(W *\left(f''(\hat u_\varphi)\left(\xi^{1,2}_\infty h^3_\infty +\xi^{1,3}_\infty h^2_\infty +\xi^{2,3}_\infty h^1_\infty  \right)+f^{(3)}(\hat u_\varphi)h^1_\infty h^2_\infty h^3_\infty\right)\right)=0.
\end{equation}
Indeed, differentiating two times the identity
\[
T_\infty(u_0) = W*(f(T_\infty(u_0)),  
\]
one obtains
\begin{align*}
{\rm D}^2T_\infty(u_0)[h^1_0,h^2_0] = W*\Big(&f'(T_\infty(u_0)){\rm D}^2T_\infty(u_0)[h^1_0,h^2_0]\\
&+f''(T_\infty(u_0)){\rm D}T_\infty(u_0)[h^1_0] {\rm D}T_\infty(u_0)[h^2_0]\Big),
\end{align*}
and thus
\begin{align*}
{\rm D}^2T_\infty&(\tilde u_0)[h^1_0,h^2_0]-{\rm D}^2T_\infty(u_0)[h^1_0,h^2_0] \\
&\quad - W*\Big(f'(T_\infty(u_0))({\rm D}^2T_\infty(\tilde u_0)[h^1_0,h^2_0]-{\rm D}^2T_\infty(u_0)[h^1_0,h^2_0])\Big)\\
& = W *\Big((f'(T_\infty(\tilde u_0))-f'(T_\infty(u_0))){\rm D}^2T_\infty(\tilde u_0)[h^1_0,h^2_0]\\
&\qquad \qquad +f''(T_\infty(\tilde u_0)){\rm D}T_\infty(\tilde u_0)[h^1_0] {\rm D}T_\infty(\tilde u_0)[h^2_0]\\
&\qquad \qquad -f''(T_\infty(u_0)){\rm D}T_\infty(u_0)[h^1_0] {\rm D}T_\infty(u_0)[h^2_0]\Big).
\end{align*}
Now $P_\varphi$ applied to the left-hand side gives $0$, since for any $v$
\begin{align*}
P_\varphi(W*(f'(\hat u_\varphi)v))&=\int W*(f'(\hat u_\varphi)v) \hat u'_\varphi f'(\hat u_\varphi) = \int f'(\hat u_\varphi) { v} W*(f'(\hat u_\varphi) u'_\varphi)\\
&=\int f'(\hat u_\varphi) { v} u'_\varphi = P^t_\varphi(v),
\end{align*}
and thus $P^t_\varphi$ applied to the right-hand side gives also $0$, which implies \eqref{eq:Ptvarphi(...)=0} when taking $\hat u_0=u_0+\varepsilon h^3_0$ and $\varepsilon$ going to $0$.

With \eqref{eq:Ptvarphi(...)=0} in mind one can then perform a decomposition of \eqref{eq:zeta mild} similar to the one made for the second derivative, replacing 
\begin{align*}
f''(u_s)&\left(\xi^{1,2}_s h^3_s +\xi^{1,3}_s h^2_s +\xi^{2,3}_s h^1_s  \right)+f^{(3)}(u_s)h^1_s h^2_s h^3_s
\end{align*}
with
\begin{align*}
f''(\hat u_\varphi)\left(\xi^{1,2}_\infty h^3_\infty +\xi^{1,3}_\infty h^2_\infty +\xi^{2,3}_\infty h^1_\infty  \right)+f^{(3)}(\hat u_\varphi)h^1_\infty h^2_\infty h^3_\infty.
\end{align*}
This replacement produces several terms, as for example
\[
I = \int_0^t e^{(t-s)\mathcal{L}_\varphi} W*\Big(f^{(3)}(\hat u_s)h^1_s h^2_s (h^3_s-h^3_\infty)\Big){\rm d} s,
\]
which may be tackled in the following way, relying in particular on \eqref{eq:DT_t bounded} and \eqref{eq:DT_t CV Linfty}
\begin{align*}
\Vert I \Vert_{L^2}\leq C\Vert h^1_0\Vert_{L^2}\Vert h^2_0\Vert_{L^2} \int_0^t (c_4 e^{-s}\Vert h_0^3\Vert_{L^\infty} +c_5 e^{-\min(\lambda,1)s}\Vert h_0^3\Vert_{L^2}){\rm d} s.
\end{align*}
The other terms may be bounded with similar arguments (recalling that $\Vert h^3_\infty\Vert_{L^\infty}\leq m_2\Vert h^3_0\Vert_{L^2}$), which means that there exists a constant $c_8>0$ such that for any $u_0\in \tilde{N}$, $t\geq 0$,  $h^1_0,h^2_0\in L^2$ and $h^3_0\in \mathcal{X}$,
\begin{equation}\label{eq: bound D3Tt}
\Vert {\rm D}^3 T_t(u_0)[h^1_0,h^2_0,h^3_0]\Vert_{L^2} \leq c_8 \Vert h^1_0\Vert_{L^2}\Vert h^2_0\Vert_{L^2}(\Vert h^3_0\Vert_{L^\infty}+\Vert h^3_0\Vert_{L^2}).
\end{equation}

\medskip

\noindent
{\it Sixth step.} It remains to prove that $({\rm D}^3 T_{t_n})_{n\in \mathbb{N}}$ is a Cauchy sequence in $C(\mathcal{N}, \mathcal{B}((L^2)^2\times \mathcal{X};L^2))$. We do not give in details all the computations, which involve several terms. 
With a similar decomposition as in Step 4, one obtains in particular the term
\[
J = \int_0^{t_n}e^{(t_n-s)\mathcal{L}_\varphi}\left(e^{(t_m-t_n)\mathcal{L}_\varphi}-1\right) W*\big(f^{(3)}(u_s)h^1_sh^2_s(h^3_s -h^3_\infty)\big){\rm d} s
\]
which may be tackled as follows:
\begin{align*}
\Vert J\Vert_{L^2}&\leq C_\mathcal{L} \Vert W\Vert_{L^2}\Vert f^{(3)}\Vert_{L^\infty} \int_0^{t_n}e^{-\kappa(t_n-s)} \Vert h^1_s\Vert_{L^2}\Vert h^2_s\Vert_{L^2} \Vert h^3_s-h^3_\infty\Vert_{L^\infty}{\rm d} s\\
&\leq  C \Vert h^1_0\Vert_{L^2}\Vert h^2_0\Vert_{L^2}   \int_0^{t_n} e^{-\kappa(t_n-s)}(c_4 e^{-s}\Vert h_0^3\Vert_{L^\infty} +c_5 e^{-\min(\lambda,1)s}\Vert h_0^3\Vert_{L^2}){\rm d} s\\
&\leq C e^{-c t_n}\Vert h^1_0\Vert_{L^2}\Vert h^2_0\Vert_{L^2}(\Vert h^3_0\Vert_{L^\infty}+\Vert h^3_0\Vert_{L^2}),
\end{align*}
for some $c>0$. The other terms may be tackled similarly, recalling that $\Vert h_\infty^3\Vert_{L^\infty}\leq c_2\Vert h^3_0\Vert_{L^2}$, and at the end one obtains, for $u_0\in \mathcal{N}$ , $h^1_0,h^2_0\in L^2$ and $h^3_0\in \mathcal{X}$, 
\begin{equation*}
\Vert \zeta_{t_m}-\zeta_{t_n}\Vert_{L^2}\leq c_9 e^{-c t_n} \Vert h^1_0\Vert_{L^2}\Vert h^2_0\Vert_{L^2}(\Vert h^3_0\Vert_{L^\infty}+\Vert h^3_0\Vert_{L^2}).
\end{equation*}

\medskip

\noindent
{\it Derivatives of the isochron.} From the above calculations, we obtain that for $u_{ 0} \in \mathcal{ N}$,
\begin{align*}
{\rm D}T_{ \infty} (u_{ 0}) [v] &= -{\rm D} \Theta(u_{ 0})[v] \hat{ u}^{ \prime}_{ \Theta(u_{ 0})},\\
{\rm D}^2 T_\infty(u_0)[v,w]&=- {\rm D}^2\Theta(u_0)[v,w]\hat u_{\Theta(u_0)} '+{\rm D}\Theta(u_0)[v]{\rm D}\Theta(u_0)[w]\hat u_{\Theta(u_0)} '',\\
\change{{\rm D}^3 T_\infty(u_0)[v,w,h]}& \change{=- {\rm D}^3\Theta(u_0)[v,w,h]\hat u_{\Theta(u_0)} '+{\rm D^2}\Theta(u_0)[v,w]{\rm D}\Theta(u_0)[h]\hat u_{\Theta(u_0)} ''}\\
&\quad \change{ +{\rm D^2}\Theta(u_0)[v,h]{\rm D}\Theta(u_0)[w]\hat u_{\Theta(u_0)} ''+{\rm D^2}\Theta(u_0)[w,h]{\rm D}\Theta(u_0)[v]\hat u_{\Theta(u_0)} ''}\\
&\quad \change{ -{\rm D}\Theta(u_0)[v]{\rm D}\Theta(u_0)[w]{\rm D}\Theta(u_0)[h] \hat u^{(3)}_{\Theta(u_0)}}.
\end{align*} 
Consequently,
\begin{align}
{\rm D} \Theta(u_{ 0})[v]&= - \frac{ \left\langle {\rm D}T_{ \infty} (u_{ 0})[v]\, ,\, \hat{ u}^{ \prime}_{ \Theta(u_{ 0})}\right\rangle_{ m_{ \Theta(u_{ 0})}}}{ \left\langle \hat{ u}^{ \prime}_{0}\, ,\, \hat{ u}^{ \prime}_{ 0}\right\rangle_{ m_{0}}}, \label{eq:DTheta_expr}\\
{\rm D}^2\Theta(u_0)[v,w]&= \frac{ \left\langle -{\rm D}^2 T_\infty(u_0)[v,w]\, ,\, \hat{ u}^{ \prime}_{ \Theta(u_{ 0})}\right\rangle_{ m_{ \Theta(u_{ 0})}}}{ \left\langle \hat{ u}^{ \prime}_{0}\, ,\, \hat{ u}^{ \prime}_{ 0}\right\rangle_{ m_{0}}}\nonumber\\
&\quad +\frac{ {\rm D}\Theta(u_0)[v]{\rm D}\Theta(u_0)[w]\left\langle \hat u_{0} ''\, ,\, \hat{ u}^{ \prime}_{ 0}\right\rangle_{ m_{ 0}}}{ \left\langle \hat{ u}^{ \prime}_{0}\, ,\, \hat{ u}^{ \prime}_{ 0}\right\rangle_{ m_{0}}},\label{eq:D2Theta_expr}
\end{align}
and
{
\begin{align*}
& {\rm D}^3 \Theta(u_{ 0})[v,w,h]=- \frac{ \left\langle {\rm D}^3T_{ \infty} (u_{ 0})[v,w,h]\, ,\, \hat{ u}^{ \prime}_{ \Theta(u_{ 0})}\right\rangle_{ m_{ \Theta(u_{ 0})}}}{ \left\langle \hat{ u}^{ \prime}_{0}\, ,\, \hat{ u}^{ \prime}_{ 0}\right\rangle_{ m_{0}}}\\
&\quad +\frac{{\rm D}^2\Theta(u_0)[v,w]{\rm D}\Theta(u_0)[h]\left\langle \hat u_{0} ''\, ,\, \hat{ u}^{ \prime}_{ 0}\right\rangle_{ m_{ 0}}}{ \left\langle \hat{ u}^{ \prime}_{0}\, ,\, \hat{ u}^{ \prime}_{ 0}\right\rangle_{ m_{0}}}+\frac{{\rm D}^2\Theta(u_0)[v,h]{\rm D}\Theta(u_0)[w]\left\langle \hat u_{0} ''\, ,\, \hat{ u}^{ \prime}_{ 0}\right\rangle_{ m_{ 0}}}{ \left\langle \hat{ u}^{ \prime}_{0}\, ,\, \hat{ u}^{ \prime}_{ 0}\right\rangle_{ m_{0}}}\\
&\quad+\frac{{\rm D}^2\Theta(u_0)[w,h]{\rm D}\Theta(u_0)[v]\left\langle \hat u_{0} ''\, ,\, \hat{ u}^{ \prime}_{ 0}\right\rangle_{ m_{ 0}}}{ \left\langle \hat{ u}^{ \prime}_{0}\, ,\, \hat{ u}^{ \prime}_{ 0}\right\rangle_{ m_{0}}}-\frac{{\rm D}\Theta(u_0)[v]{\rm D}\Theta(u_0)[w]{\rm D}\Theta(u_0)[h]\left\langle \hat u_{0} ^{(3)}\, ,\, \hat{ u}^{ \prime}_{ 0}\right\rangle_{ m_{ 0}}}{ \left\langle \hat{ u}^{ \prime}_{0}\, ,\, \hat{ u}^{ \prime}_{ 0}\right\rangle_{ m_{0}}},
\end{align*}}
where we have used the identities
\[
\left\langle \hat{ u}^{ \prime}_{ \Theta(u_{ 0})}\, ,\, \hat{ u}^{ \prime}_{ \Theta(u_{ 0})}\right\rangle_{ m_{ \Theta(u_{ 0})}}=  \left\langle \hat{ u}^{ \prime}_{0}\, ,\, \hat{ u}^{ \prime}_{ 0}\right\rangle_{ m_{0}},\quad  \left\langle \hat u_{\Theta(u_0)} ''\, ,\, \hat{ u}^{ \prime}_{ \Theta(u_{ 0})}\right\rangle_{ m_{ \Theta(u_{ 0})}}= \left\langle \hat u_{0} ''\, ,\, \hat{ u}^{ \prime}_{ 0}\right\rangle_{ m_{ 0}},
\]
{and
\[
 \left\langle \hat u_{\Theta(u_0)}^{(3)}\, ,\, \hat{ u}^{ \prime}_{ \Theta(u_{ 0})}\right\rangle_{ m_{ \Theta(u_{ 0})}}= \left\langle \hat u_{0}^{(3)}\, ,\, \hat{ u}^{ \prime}_{ 0}\right\rangle_{ m_{ 0}}.
\]
Therefore, we indeed deduce from \eqref{eq:DT_t bounded}, \eqref{eq:D2T bounded} and \eqref{eq: bound D3Tt} the bound \eqref{eq:bound D3Theta} and that $\sup_{ u_{ 0}\in \mathcal{ N}} \left\Vert {\rm D} \Theta(u_{ 0}) \right\Vert< +\infty$ and $\sup_{ u_{ 0}\in \mathcal{ N}} \left\Vert {\rm D}^{ 2} \Theta(u_{ 0}) \right\Vert< +\infty$.}
\medskip

\noindent
{\it Computations on $\mathcal{M}$.}
Let us now compute the two first derivatives of $\Theta$ at $u_0=\hat u_\varphi $. Applying first \eqref{eq:DTheta_expr} to $u_{ 0}:= \hat{ u}_{ \varphi}$, we have,
since ${\rm D} T_t(\hat u_\phi)[v]= e^{t\mathcal{L}_\phi}v$,
\begin{equation*}
{\rm D}\Theta(\hat u_\phi)[v] = - \frac{\left\langle v,\hat u'_\phi\right\rangle_{m_\phi}}{\left\langle \hat u'_0,\hat u'_0\right\rangle_{m_0}}.
\end{equation*}
Moreover, remarking that
\begin{equation*}
\langle \hat u'_0,\hat u''_0\rangle_{m_0} = \int_{\mathbb{R}} \hat u'_0(x) \hat u''_0(x) f'(\hat u_0)(x) {\rm d} x =0,
\end{equation*}
since $\hat u'_0$ is even, $\hat u''_0$ is odd and $f'(\hat u_0)$ is even, we obtain
\begin{align*}
{\rm D}^2\Theta(\hat u_\phi)[v,w]&=- \frac{\left\langle {\rm D}^2 T_\infty(\hat u_\phi)[v,w],\hat u'_\phi\right\rangle_{m_\phi}}{\left\langle \hat u_0',\hat u_0'\right\rangle_{m_0}}\\
&=\frac{\left\langle \int_0^t e^{(t-s)\mathcal{L}_\phi} W*\big(f''(\hat u_\phi)\,  e^{s\mathcal{L}_\phi}v\,  e^{s\mathcal{L}_\phi}w\big){\rm d}s,\hat u'_\phi\right\rangle_{m_\phi}}{\left\langle \hat u'_0,\hat u'_0\right\rangle_{m_0}}\\
&=- \frac{\int_\mathbb{R}\int_0^\infty f''(\hat u_\phi)(x) \hat u'_\phi(x)  e^{s\mathcal{L}_\phi}v(x)  e^{s\mathcal{L}_\phi}w(x) {\rm d}s {\rm d} x}{\left\langle \hat u'_0,\hat u'_0\right\rangle_{m_0}},
\end{align*}
where we used the identity $ \hat{ u}'_\phi = W*(m_\phi \hat{ u}'_\phi)$.
One obtains then \eqref{eq:D2Theta} remarking that, since $f''(\hat u_0)$ is odd while $ \hat{u}'_0$ is even and $e^{s \mathcal{L}_0} \hat{ u}'_0= \hat{ u}'_0$, we have  ${\rm D}^2\Theta(\hat u_\phi)[\hat u'_\phi,\hat u'_\phi]={\rm D}^2\Theta(\hat u_0)[\hat u'_0,\hat u'_0]=0$. 
\end{proof}

\subsection*{Acknowledgments.}
E.L. acknowledges the support of projects CHAMANE ANR-19-CE40-002 and HAPPY ANR-23-CE40-0007 of the French National Research Agency. 
C.P. acknowledges the support of project CONVIVIALITY ANR-23-CE40-0003 of the French National Research Agency. \change{We would like to thank warmly the reviewers for their very careful reading of the paper and their remarks that considerably helped to improve the results.}

\def\cprime{$'$}

\end{document}